\documentclass[10pt, a4paper]{article}

\pdfoutput=1

\usepackage[utf8]{inputenc}
\usepackage[T1]{fontenc}
\usepackage[protrusion=true,expansion]{microtype}

\usepackage[a4paper]{geometry}
\usepackage{mathtools}

\usepackage{ifthen}
\usepackage{suffix}

\usepackage{xcolor}
	\definecolor {odarkblue} {rgb} {.12,.3,.5}
	\definecolor {odarkred} {rgb} {.6,.22,.18}
	\definecolor {odarkgreen} {rgb} {.45,.6,.25}
\usepackage{hyperref}
\hypersetup{
	pdftitle=Local limits of Markov-branching trees and their volume growth,
	pdfauthor=Camille Pagnard,
	colorlinks=true,
	linkcolor=odarkblue,
	citecolor=odarkred,
	urlcolor=odarkgreen}

\usepackage{subcaption}
\usepackage{changepage}

\usepackage[charter]{mathdesign}

\DeclareSymbolFontAlphabet{\mathcal}{symbols}
\DeclareSymbolFontAlphabet{\mathscr}{symbols}
\DeclareSymbolFontAlphabet{\mathbb}{symbols}
\DeclareSymbolFontAlphabet{\mathfrak}{symbols}

\usepackage[cal=euler,calscaled=.945,scr=boondoxo,scrscaled=.955,bb=pazo,bbscaled=.935]{mathalfa}

\makeatletter
\g@addto@macro \normalsize {
	\setlength\abovedisplayskip {4pt plus 4pt minus 2pt}
	\setlength\belowdisplayskip {8pt plus 4pt minus 2pt}
	\setlength\belowdisplayshortskip {3pt plus 3pt minus 3pt}}
\makeatother

\newcommand\abs[2][0]{
	\ifthenelse{\equal{#1}{0}}{\vert #2 \vert}{
	\ifthenelse{\equal{#1}{1}}{\bigl\vert #2 \bigr\vert}{
	\ifthenelse{\equal{#1}{2}}{\Bigl\vert #2 \Bigr\vert}{
	\ifthenelse{\equal{#1}{3}}{\biggl\vert #2 \biggr\vert}{
	\ifthenelse{\equal{#1}{4}}{\Biggl\vert #2 \Biggr\vert}{}}}}}}
\WithSuffix\newcommand\abs*[1]
	{\left\vert #1 \right\vert}
\newcommand\norm[2][0]{
	\ifthenelse{\equal{#1}{0}}{\Vert #2 \Vert}{
	\ifthenelse{\equal{#1}{1}}{\bigl\Vert #2 \bigr\Vert}{
	\ifthenelse{\equal{#1}{2}}{\Bigl\Vert #2 \Bigr\Vert}{
	\ifthenelse{\equal{#1}{3}}{\biggl\Vert #2 \biggr\Vert}{
	\ifthenelse{\equal{#1}{4}}{\Biggl\Vert #2 \Biggr\Vert}{}}}}}}
\WithSuffix\newcommand\norm*[1]
	{\left\Vert #1 \right\Vert}
\newcommand\Norm[2][0]{
	\ifthenelse{\equal{#1}{0}}{\Vvert #2 \Vvert}{
	\ifthenelse{\equal{#1}{1}}{\bigl\Vvert #2 \bigr\Vvert}{
	\ifthenelse{\equal{#1}{2}}{\Bigl\Vvert #2 \Bigr\Vvert}{
	\ifthenelse{\equal{#1}{3}}{\biggl\Vvert #2 \biggr\Vvert}{
	\ifthenelse{\equal{#1}{4}}{\Biggl\Vvert #2 \Biggr\Vvert}{}}}}}}
\WithSuffix\newcommand\Norm*[1]
	{\left\Vvert #1 \right\Vvert}
\newcommand\esp[2][0]{
	\mathbb E
	\ifthenelse{\equal{#1}{0}}{[ #2 ]}{
	\ifthenelse{\equal{#1}{1}}{\bigl[ #2 \bigr]}{
	\ifthenelse{\equal{#1}{2}}{\Bigl[ #2 \Bigr]}{
	\ifthenelse{\equal{#1}{3}}{\biggl[ #2 \biggr]}{
	\ifthenelse{\equal{#1}{4}}{\Biggl[ #2 \Biggr]}{}}}}}}
\WithSuffix\newcommand\esp*[1]
	{\mathbb E \left[ #1 \right]}
\newcommand\prob[2][0]{
	\mathbb P
	\ifthenelse{\equal{#1}{0}}{[ #2 ]}{
	\ifthenelse{\equal{#1}{1}}{\bigl[ #2 \bigr]}{
	\ifthenelse{\equal{#1}{2}}{\Bigl[ #2 \Bigr]}{
	\ifthenelse{\equal{#1}{3}}{\biggl[ #2 \biggr]}{
	\ifthenelse{\equal{#1}{4}}{\Biggl[ #2 \Biggr]}{}}}}}}
\WithSuffix\newcommand\prob*[1]
	{\mathbb P \left[ #1 \right]}

\usepackage{xspace}
	\newcommand{\ie}{i.e.\xspace}
	\newcommand{\iid}{i.i.d.\xspace}
	\newcommand{\as}{a.s.\xspace}
	
	\newcommand{\tiff}{i.f.f.\xspace}

\usepackage{pgf}
\usepackage{tikz}
	\usetikzlibrary{arrows,calc}

\usepackage{titlesec}
	\newcommand{\periodafter}[1]{#1.}
	\titleformat*{\section}{\large\bfseries\scshape\filcenter}
	\titleformat*{\subsection}{\large\bfseries}
	\titleformat{\subsubsection}[runin]{\normalsize\bfseries}{\thesubsubsection}{.5em}{\periodafter}[\hspace*{.5em}]
	\titleformat{\paragraph}[runin]{\normalsize\itshape}{}{.25em}{\periodafter}
	\titleformat*{\subparagraph}{\itshape}

\usepackage{fancyhdr}
\usepackage{enumitem}
	\setlist{nolistsep}
	\setdescription{font=\normalfont}

\usepackage{titling}
	\pretitle{
		\begin{center}
		\begin{minipage}{.75\textwidth}
		\centering
		\LARGE}
	\posttitle{
		\end{minipage}
		\end{center}}
	\preauthor{
		\begin{center}
		\large}
	\postauthor{
		\end{center}}

\usepackage[hyperref,amsmath,thmmarks]{ntheorem}

\newcommand {\qedsym} {\ensuremath\square}

\theoremstyle {plain}
\theoremheaderfont {\normalfont\bfseries}
\theorembodyfont {\itshape}
\theoremsymbol {}
\theoremseparator {.}
\newtheorem {theorem} {Theorem} [section]
\newtheorem {lemma} [theorem] {Lemma}
\newtheorem {proposition} [theorem] {Proposition}
\newtheorem {corollary} [theorem] {Corollary}

\theoremsymbol {}
\theorembodyfont {\upshape}
\theoremheaderfont {\itshape}
\newtheorem {remark} {Remark} [section]

\makeatletter
\newtheoremstyle{nonumnoparplain}
	{\item[\theorem@headerfont\hskip\labelsep ##1\theorem@separator]}
	{\item[\theorem@headerfont\hskip\labelsep ##1\ ##3\theorem@separator]}
\makeatother

\theoremstyle {nonumnoparplain}
\theoremheaderfont {\itshape}
\theoremsymbol {\qedsym}
\newtheorem {proof} {Proof}

\providecommand\mathup\mathrm
\providecommand\Beta{\operatorname{B}}
\providecommand\Vbar{\mathrlap\perp\,\perp}
\providecommand\lBrack\llbracket
\providecommand\rBrack\rrbracket

\providecommand\calP{\mathcal P}

\providecommand\calN{\mathcal N}
\providecommand\calU{\mathcal U}
\providecommand\calL{\mathcal L}
\providecommand\calI{\mathcal I}
\providecommand\calT{\mathcal T}

\providecommand\calM{\mathcal M}

\providecommand\scrF{\mathscr F}
\providecommand\scrR{\mathscr R}
\providecommand\scrT{\mathscr T}
\providecommand\scrI{\mathscr I}
\providecommand\ttb{\mathtt b}
\providecommand\tts{\mathtt s}
\providecommand\ttt{\mathtt t}
\providecommand\ttT{\mathtt T}
\providecommand\ttv{\mathtt v}
\providecommand\bbR{\mathbb R}
\providecommand\bbN{\mathbb N}
\providecommand\bbZ{\mathbb Z}
\providecommand\bbK{\mathbb K}
\providecommand\bbT{\mathbb T}
\providecommand\upd{\mathup d}
\providecommand\upD{\mathup D}
\providecommand\upC{\mathup C}
\providecommand\upB{\mathup B}
\providecommand\upF{\mathup F}
\providecommand\bfr{\mathbf r}
\providecommand\bfs{\mathbf s}

\providecommand\bfx{\mathbf x}
\providecommand\bfy{\mathbf y}
\providecommand\bfT{\mathbf T}
\providecommand\bfX{\mathbf X}
\providecommand\bfY{\mathbf Y}
\providecommand\bfG{\mathbf G}

\DeclareMathOperator\pr{pr}
\DeclareMathOperator\dis{dis}
\newcommand\GW{\operatorname{GW}}
\newcommand\MB{\operatorname{MB}}
\newcommand\loc{{\operatorname{loc}}}
\newcommand\GHP{{\operatorname{GHP}}}
\newcommand\TV{{\operatorname{TV}}}
\newcommand\ord{{\operatorname{ord}}}
\newcommand\gr{{\operatorname{gr}}}
\newcommand\AG{{\operatorname{AG}}}
\newcommand\GT{{\operatorname{GT}}}

\providecommand{\ind}{\mathbb 1}
\providecommand{\D}{\mathup d}
\providecommand{\E}{\mathup e}
\newcommand\Sdec{\mathcal S^{\smash[t]\downarrow}}

\renewenvironment{abstract}
	{\begin{small}
	\begin{center}
		\bfseries\scshape\abstractname\vspace{-.5em}\vspace{0pt}
	\end{center}
	\begin{adjustwidth}{.125\textwidth}{.125\textwidth}
	\hspace{\parindent}}
	{\end{adjustwidth}
	\end{small}}

\title{Local limits of Markov Branching trees and their volume growth}
\author{Camille Pagnard\\
	\footnotesize
	Universit\'e Paris-Dauphine, Ceremade\\
	\scriptsize
	\texttt{pagnard@ceremade.dauphine.fr}}
\date{\today}

\begin{document}

\maketitle

\begin{abstract}
	We are interested in the local limits of families of random trees
	that satisfy the Markov branching property, which is fulfilled by a wide range of models.
	Loosely, this property entails that given the sizes of the sub-trees above the root,
	these sub-trees are independent and their distributions only depend upon their respective sizes.
	The laws of the elements of a Markov branching family are characterised by a sequence of probability distributions
	on the sets of integer partitions which describes how
	the sizes of the sub-trees above the root are distributed.
	
	We prove that under some natural assumption on this sequence of probabilities,
	when their sizes go to infinity, the trees converge in distribution to an infinite tree
	which also satisfies the Markov branching property.
	Furthermore, when this infinite tree has a single path from the root to infinity,
	we give conditions to ensure its convergence in distribution
	under appropriate rescaling of its distance and counting measure
	to a self-similar fragmentation tree with immigration.
	In particular, this allows us to determine
	how, in this infinite tree, the ``volume'' of the ball of radius $R$
	centred at the root asymptotically grows with $R$.
	
	Our unified approach will allow us to develop various new applications,
	in particular to different models of growing trees and cut-trees,
	and to recover known results.
	An illustrative example lies in the study of Galton-Watson trees:
	the distribution of a critical Galton-Watson tree conditioned on its size
	converges to that of Kesten's tree when the size grows to infinity.
	If furthermore, the offspring distribution has finite variance,
	under adequate rescaling, Kesten's tree converges to Aldous' self-similar CRT
	and the total size of the $R$ first generations asymptotically behaves like $R^2$.
\end{abstract}

\section{Introduction}\label{sec:intro}

The focus of this work is to study the asymptotic behaviour of sequences of random trees
which satisfy the Markov branching property first introduced by Aldous in \cite[Section~4]{aldous1996cladograms}
and later extended for example in~\cite{bdms2008heightofincreasingtrees,haas2012scaling,hmpw2008continuumtree}.
See Haas~\cite{haas2016markovbranchingsurvey} for an overview of this general model
and Lambert~\cite{lambert2016saopaulo} for applications to models used in evolutionary biology. 
Our study will therefore encompass various models,
like Galton-Watson trees conditioned on their total progeny or their number of leaves,
certain models of cut-trees (see Bertoin~\cite{bertoin2012fires,bertoin2014percolationrecursivetrees,bertoin2015cutrecursivetrees})
or recursively built trees (see R\'emy~\cite{remy1985binarygw}, Chen-Ford-Winkel~\cite{ford2009alphagamma}
and Haas-Stephenson~\cite{haas2014kary})
as well as models of phylogenetic trees (Ford's $\alpha$-model~\cite{ford2005alpha}
and Aldous' $\beta$-splitting model~\cite{aldous1996cladograms}).

Informally, a sequence $(T_n)_n$ of random trees satisfies the Markov branching property
if for all $n$, $T_n$ has ``size'' $n$, and conditionally on the event ``$T_n$ has $p$ sub-trees above its root
with respective sizes $n_1\geq\dots\geq n_p$'', these sub-trees are independent and
for each $i=1,\dots, p$, the $i^{\text{th}}$ largest sub-tree is distributed like $T_{n_{\smash i}}$.
The sequence of distributions of $(T_n)_n$ is characterised by a family $q = (q_n)_n$ of probability distributions,
referred to as ``first-split distributions'' (see next paragraph),
where $q_n$ is supported by the set of partitions of the integer $n$.
We will detail two different constructions of Markov branching trees corresponding to a given sequence $q$
for two different notions of size: the number of leaves or the number of vertices.

\smallskip

Let $(q_n)_n$ be a sequence of first-split distributions.
A tree with $n$ leaves with distribution in the associated Markov branching family is built with the following process.
Consider a cluster of $n$ identical particles and with probability $q_n(\lambda_1,\dots,\lambda_p)$,
split it into $p$ smaller clusters containing $\lambda_1, \dots, \lambda_p$ particles respectively.
For each $i=1,\dots,p$, independently of the other sub-clusters,
split the $i^{\text{th}}$ cluster according to $q_{\lambda_{\smash i}}$.
Repeat this procedure until all the sub-clusters are empty.
The genealogy of these splits may be encoded as a tree with $n$ leaves,
the distribution of which we'll denote by $\MB^{\smash\calL,q}_n$.
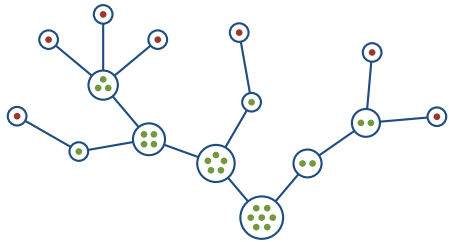
\begin{figure}[ht]
	\centering
		\begin{tikzpicture}[every node/.style={font=\scriptsize}, thick, baseline={(0,0)}, scale=.75]

		\newcommand\ballspattern[3]{
			\ifthenelse{\equal{#1}{7}}{
				\node [ball] at ($#2$) {};
				\foreach \i in {1,...,6}
					\node [ball] at ($#2+(60*\i:1.1*#3)$) {};
			}{}
			\ifthenelse{\equal{#1}{5}}{
				\foreach \i in {1,...,5}
					\node [ball] at ($#2+0.85*(18+72*\i:#3)$) {};
			}{}
			\ifthenelse{\equal{#1}{4}}{
				\foreach \i in {1,...,4}
					\node [ball] at ($#2+0.7*(45+90*\i:#3)$) {};
			}{}
			\ifthenelse{\equal{#1}{3}}{
				\foreach \i in {1,...,3}
					\node [ball] at ($#2+0.58*(-30+120*\i:#3)$) {};
			}{}
			\ifthenelse{\equal{#1}{2}}{
				\node [ball] at ($#2+0.5*(0:#3)$) {};
				\node [ball] at ($#2-0.5*(0:#3)$) {};
			}{}
			\ifthenelse{\equal{#1}{1}}{
				\node [ball] at ($#2$) {};
			}{}
			\ifthenelse{\equal{#1}{0}}{
				\node [ball, odarkred] at ($#2$) {};
			}{}}
		\newcommand\edgelength{1.25}
		
		\tikzstyle{phantomvertex} = [inner sep=0pt, outer sep=0pt, minimum width=0pt]
		\tikzstyle{vertex} = [circle, inner sep=0pt, outer sep=0pt, minimum width=4pt, fill=odarkblue]
		\tikzstyle{ball} = [circle, inner sep=0pt, outer sep=0pt, minimum width=2.5pt, fill=odarkgreen]
		\tikzstyle{container} = [circle, inner sep=.5pt, outer sep=0pt, minimum width=6pt, fill=white, draw=odarkblue]
		\tikzstyle{edge} = [draw, -, odarkblue, shorten >=-.5pt, shorten <=-.5pt]
		
		\def\vertices{
				{(0,0)/0/7/16},
					{($(0)+(130:\edgelength)$)/1/5/14},
						{($(1)+(160:\edgelength)$)/11/4/12},
							{($(11)+(190:\edgelength)$)/111/1/7},
								{($(111)+(150:\edgelength)$)/1111/0/7},
							{($(11)+(130:\edgelength)$)/112/3/11},
								{($(112)+(140:\edgelength)$)/1121/0/7},
								{($(112)+(90:\edgelength)$)/1122/0/7},
								{($(112)+(40:\edgelength)$)/1123/0/7},
						{($(1)+(60:\edgelength)$)/12/1/7},
							{($(12)+(100:\edgelength)$)/121/0/7},
					{($(0)+(50:\edgelength)$)/2/2/10.5},
						{($(2)+(35:\edgelength)$)/21/2/10.5},
							{($(21)+(85:\edgelength)$)/211/0/7},
							{($(21)+(5:\edgelength)$)/212/0/7}}
		\def\edges{
			0/1, 1/11, 11/111, 111/1111, 11/112, 112/1121, 112/1122, 112/1123, 1/12, 12/121,
			0/2, 2/21, 21/211, 21/212}
		\foreach \pos/\name/\mass/\radius in \vertices
			\node[phantomvertex] (\name) at \pos {};
		\foreach \source/\dest in \edges {
			\path[edge] (\source) -- (\dest);}
		\foreach \pos/\name/\mass/\radius/\frame in \vertices {
			\node [container, minimum width=\radius pt] at (\name) {};
			\ballspattern{\mass}{(\name)}{5pt};}
	\end{tikzpicture}
	\caption{Example of a tree with $7$ leaves (in red) and first-split equal to $(5,2)$}
\end{figure}

A Markov branching tree with a given number of vertices, say $n$,
is built with a slightly different procedure
and we will note $\MB^q_n$ its distribution.
Section~\ref{sec:finite-mb-trees} will rigorously detail the constructions 
of both $\MB^q_n$ and $\MB^{\smash\calL,q}_n$.
Rizzolo~\cite{rizzolo2015scaling} considered a more general notion of size
and described the construction of corresponding Markov branching trees.

\bigskip

One way of looking at the behaviour of large trees is through the local limit topology.
For a given tree $\ttt$ and $R\geq 0$, we denote by $\ttt\vert_R$
the subset of vertices of $\ttt$ at graph distance less than $R$ from its root.
We will say that a sequence $\ttt_n$ converges locally to a limit tree $\ttt_\infty$
if for any radius $R$, $\ttt_n\vert_R = \ttt_\infty\vert_R$ for sufficiently large $n$.
There is considerable literature on the study of the local limits
of certain classes of random trees or, more generally, of graphs.
For instance, see Abraham and Delmas~\cite{abraham2014condensation,abraham2014gwkesten},
Stephenson~\cite{stephenson2014localmultitypegw},
Stef\'ansson~\cite{stefansson2009ford,stefansson2012markov}
or a recent paper by Broutin and Mailler~\cite{broutin2015andortrees},
as well as references therein, for studies related to our work.

\smallskip

Let us present in this Introduction the simplest, and most common, case in which Markov branching trees have local limits.
Let $(T_n)_n$ be a sequence of Markov branching trees indexed by their size
with corresponding family of first-split distributions $(q_n)_n$.
Let $p$ be a non-negative integer and $\lambda_1\geq\dots\geq\lambda_p > 0$
be a non increasing family of integers with sum $L$.
For $n$ large enough, consider $q_n(n-L,\lambda_1,\dots,\lambda_p)$,
that is the probability that $T_n$ gives birth to $p+1$ sub-trees
among which the $p$ smallest have respective sizes $\lambda_1,\dots,\lambda_p$.
Assume that for any such $p$ and $\lambda$,
$q_n(n-L,\lambda_1,\dots,\lambda_p)$ converges to $q_*(\lambda_1,\dots,\lambda_p)$
for some probability measure $q_*$ on the set of non-increasing finite sequences of positive integers.
Under this natural assumption,
we will prove in a rather straightforward way that
$T_n$ locally converges to some ``infinite Markov branching tree'' $T_\infty$
with a single path from the root to infinity, called its \emph{infinite spine}.
The distribution of $T_\infty$ is characterised by the family $(q_n)_n$
and the measure $q_*$ which describes the distribution of the sizes
of the finite sub-trees grafted on the spine of $T_\infty$.
See Theorem~\ref{thm:local-limits-markov-branching} for a more precise and general statement.

\bigskip

A drastically different approach to understand the behaviour of large random trees is that of scaling limits.
Aldous was the first to study scaling limits of random trees as a whole, see~\cite{aldous1991crt1},
and notably introduced the celebrated Brownian tree as the limit of rescaled
critical Galton-Watson trees conditioned on their size
with any offspring law that has finite variance.
See also Le-Gall~\cite{legall2006randomrealtrees} for a survey on random ``continuous'' trees.

In this context, we will consider $T_n$ as a metric space rescaled by some factor $a_n$,
\ie the edges of $T_n$ will be viewed as real segments of length $a_n$,
and denote by $a_n T_n$ this rescaled metric space.
Scaling limits for Markov branching trees were studied
in~\cite{haas2012scaling,hmpw2008continuumtree} by Haas-Miermont \emph{et al}.
Their main result is that under simple conditions on the sequence
$(q_n)_n$ of first-split distributions, $T_n$ converges in distribution, under appropriate rescaling,
to a self-similar fragmentation tree.
These objects were introduced by Haas and Miermont~\cite{haas2004genealogy}
and notably encompass Aldous' Brownian tree
as well as Duquesne and Le-Gall's stable trees~\cite{legall2002stabletrees}.

\smallskip

Haas and Miermont's result in particular gives an asymptotic relation
between the size and height of a finite Markov branching tree.
When considering an infinite Markov branching tree $T$,
we may wonder if a similar relation exists,
namely how many vertices or leaves are typically found at height less than some large integer $R$.
This seemingly simple question, the study of the integer sequence $(\#T\vert_R)_R$,
leads us to consider the scaling limits of the weighted tree $(T,\mu_T)$,
where $\mu_T$ is the counting measure on either the vertices of $T$ or on its leaves.

In Theorem~\ref{thm:scaling-limits-infinite-mb-trees},
we consider the case in which $T$ is an infinite Markov branching tree with a unique infinite spine
with distribution characterised by a family $(q_n)_n$ of first-split distributions
and a probability measure $q_*$ associated to the sizes of the finite sub-trees grafted on the spine.
We prove that under the assumptions of Haas and Miermont's theorem on the family $(q_n)_n$
and an additional condition on the measure $q_*$,
when $R$ goes to infinity, the tree $T/R$ endowed with the adequately rescaled measure $\mu_T$
converges in distribution to a self-similar fragmentation tree with immigration.
These infinite continuous trees were introduced by Haas~\cite{haas2007fragmentationinitialmass}.
They include Aldous' self-similar CRT~\cite{aldous1991crt1}
(which will appear as the limit in many of our applications)
and Duquesne's immigration L\'evy trees~\cite{duquesne2009immigrationlevytrees}.

As a result, under appropriate rescaling,
the ``volume'' of the ball of radius $R$ centred at the root of $T$
converges in distribution to the measure of the ball with radius $1$ centred at the root
of a self-similar fragmentation tree with immigration.
Proposition~\ref{prop:volume-growth-cv} actually gives the stronger convergence of the whole ``volume growth'' process.

\bigskip

The unified framework used here will yield multiple applications.
As a first example, Theorem~\ref{thm:local-limits-markov-branching} will allow us to recover
known results on the local limits of conditioned Galton-Watson trees towards Kesten's tree
(see Abraham Delmas~\cite{abraham2014gwkesten} for instance)
and Theorem~\ref{thm:scaling-limits-infinite-mb-trees} will give an alternative proof to Duquesne's results
(see~\cite{duquesne2009immigrationlevytrees})
on the convergence of rescaled infinite critical Galton-Watson trees to immigration L\'evy trees.
We will give similar results for some models of cut-trees,
which encodes the genealogy of the random dismantling of trees,
studied by Bertoin~\cite{bertoin2012fires,bertoin2014percolationrecursivetrees,bertoin2015cutrecursivetrees}.
We will also study some models of sequentially growing trees
described in~\cite{ford2009alphagamma,haas2014kary,marchal2008stabletree,remy1985binarygw}
and models of phylogenetic trees \cite{aldous1996cladograms,ford2005alpha}.

\bigskip

This paper will be organised as follows.
In Section~\ref{sec:def-mb-trees}, we will define finite and infinite Markov branching trees
and give a natural criterion for their convergence under the local limit topology
in Theorem~\ref{thm:local-limits-markov-branching}.
In Section~\ref{sec:ghp-topology-ssf-trees} we will detail the background needed
for our main result, Theorem~\ref{thm:scaling-limits-infinite-mb-trees},
\ie the study of the scaling limits of infinite Markov branching trees.
Section~\ref{sec:infinite-mb-scaling-limits} will focus on the proof of this result.
Finally, Section~\ref{sec:applications} will give applications of our unified approach
to various Markov branching models.

\section{Markov branching trees and their local limits}\label{sec:def-mb-trees}

\subsection{Trees and partitions}

\subsubsection{Background on trees}\label{sec:def-trees}
First of all, let us recall Neveu's formalism for trees,
first introduced in~\cite{neveu1986arbres}.
Let $\calU := \bigcup_{n\geq 0} \bbN^n$ be the set of finite words on $\bbN$
with the conventions $\bbN = \{1,2,3,\dots\}$ and $\bbN^0 = \{\varnothing\}$.
We then call a \emph{plane tree} or \emph{ordered rooted tree}
any non-empty subset $\ttt\subset\calU$ such that:
\begin{itemize}
	\item The empty word $\varnothing$ belongs to $\ttt$,
	it will be thought of as its ``root'',
	\item If $u=(u_1,\dots,u_n)$ is in $\ttt$,
	then its parent $\pr(u) := (u_1,\dots,u_{n-1})$
	is also in $\ttt$,
	\item For all $u$ in $\ttt$,
	there exists a finite integer $c_u(\ttt)\geq 0$
	such that $u \, i := (u_1, \dots, u_n, i)$ is in $\ttt$
	for every $1\leq i\leq c_u(\ttt)$.
	We will say that $c_u(\ttt)$ is the number of children of $u$ in~$\ttt$.
\end{itemize}
Let $\ttT^\ord$ be the set of plane trees.
Observe that if $\ttt$ is an infinite plane tree,
this definition requires the number of children
of each of its vertices to be finite.

\smallskip

Plane trees are endowed with a total order which is of limited interest to us.
Because of this, we define an equivalence relation
on $\ttT^\ord$ to allow us to consider as identical
two trees which have the same shape.

Say that two plane trees $\ttt$ and $\ttt'$ are equivalent (noted $\ttt\sim\ttt'$)
\tiff there exists a bijection $\sigma:\ttt\to\ttt'$
such that $\sigma(\varnothing) = \varnothing$
and for all $u\in\ttt\setminus\{\varnothing\}$,
$\pr[\sigma(u)] = \sigma[\pr(u)]$.
Finally, set $\ttT := \ttT^\ord/\sim$.
From now on, unless otherwise stated, we will only consider unordered trees,
\ie by ``tree'' we will mean an element of $\ttT$.

\smallskip

Let $\ttt$ be a tree.
We say that a vertex $u$ on $\ttt$ is a leaf if it has no children,
\ie if $c_u(\ttt) = 0$.
Define $\#\ttt$ as the total number of vertices of $\ttt$
and $\#_\calL\ttt$ as its number of leaves.
For any positive integer $n$, let $\ttT_n$ and
$\ttT_n^{\smash\calL}$ be the sets of finite trees
with $n$ vertices and $n$ leaves respectively.
Moreover, note $\ttT_\infty$ the set of infinite trees.

\smallskip

We will use the following operations on trees:
\begin{itemize}
	\item Let $\ttt_1, \dots, \ttt_d$ be trees;
	their \emph{concatenation} is the tree $\lBrack \ttt_1, \dots, \ttt_d \rBrack$
	obtained by attaching each of their respective roots to a new common root,
	see Figure~\ref{fig:concatenation},
	\item Let $\ttt$ and $\tts$ be two trees and $u$ be a vertex of $\ttt$;
	set $\ttt\otimes(u,\tts)$ the grafting of $\tts$ on $\ttt$ at $u$,
	\ie the tree obtained by glueing the root of $\tts$ on $u$,
	see Figure~\ref{fig:grafting},
	\item Fix $\ttt$ a tree, a non-repeating family $(u_i)_{i\in\calI}$ of vertices of $\ttt$,
	and a family of trees 
	$(\tts_i)_{i\in\calI}$;
	let $\ttt \bigotimes_{i\in\calI} (u_i,\tts_i)$ be the tree obtained by grafting
	$\tts_i$ on $\ttt$ at $u_i$ for each $i$ in $\calI$.
\end{itemize}
\begin{figure}[ht]
	\centering
	\begin{minipage}[b]{.4\textwidth}
		\centering
		\begin{tikzpicture}[scale=1, every node/.style={font=\small}, thick]
			\tikzstyle{null} = [inner sep=0pt, minimum width=0pt]
			\tikzstyle{tv} = [circle, fill=odarkblue, inner sep=0pt, minimum width=4pt]
			\tikzstyle{te} = [draw, -, odarkblue]
			\tikzstyle{root} = [circle, draw=odarkgreen, fill=white, inner sep=0pt, minimum width=5pt]
			\tikzstyle{newe} = [draw, -, odarkgreen]
			\tikzstyle{oldroot} = [circle, draw=odarkred, fill=white, inner sep=0pt, minimum width=5pt]
			
			\def\nodes{
				{(0,0)/0/root},
				{($(0)+(150:1)$)/1/oldroot},
					{($(1)+(160:1)$)/11/tv}, {($(1)+(100:1)$)/12/tv},
				{($(0)+(90:1)$)/2/oldroot},
					{($(2)+(120:1)$)/21/tv},
						{($(21)+(135:1)$)/211/tv},
						{($(21)+(75:1)$)/212/tv},
					{($(2)+(60:1)$)/22/tv},
				{($(0)+(30:1)$)/3/oldroot},
					{($(3)+(50:1)$)/31/tv}}
			\def\edges{
				0/1/newe, 1/11/te, 1/12/te,
				0/2/newe, 2/21/te, 21/211/te, 21/212/te, 2/22/te,
				0/3/newe, 3/31/te}
			
			\foreach \pos/\name/\style in \nodes
				\node[null] (\name) at \pos {};
			\foreach \source/\dest/\style in \edges
				\path[\style] (\source) -- (\dest);
			\foreach \pos/\name/\style in \nodes
				\node[\style] at \pos {};
			
			\node [circle, fill=odarkred, inner sep=0pt, minimum width=2pt] at (1) {};
			\node [circle, fill=odarkred, inner sep=0pt, minimum width=2pt] at (2) {};
			\node [circle, fill=odarkred, inner sep=0pt, minimum width=2pt] at (3) {};
			
			\node[null] at ($(1)+1*(130:1)$) {$\ttt_1$};
			\node[null] at ($(2)+1*(90:1)$) {$\ttt_2$};
			\node[null] at ($(3)+1.4*(50:1)$) {$\ttt_3$};
			
			\node[null] (oldrootslabel) at ($(1)+(-150:.75)$) {};
			\node[align=center] at ($(oldrootslabel)-1/3*(90:1)$) {old\\[-.35\baselineskip]roots};
			\path[->, >=stealth', gray, thin]
				(oldrootslabel.north) edge[bend left] ($(1)+(180:0.15)$);
			
			\node[null, label=right:\!new root, outer sep=0pt] (rootlabel) at ($(0)+(-5:1.2)$) {};
			\path[->, >=stealth', shorten >=3pt, gray, thin] (rootlabel.west) edge[bend left] (0);

			\path[draw, gray, dashed, thin]
				($(2)+(90:0.15)$) to[out=0, in=120] ($(3)+(30:0.15)$)
					to[out=-60, in=30] ($(3)+(-60:0.15)$)
					to[out=-150, in=-60] ($(3)+(-150:0.15)$)
					to[out=120, in=0] ($(2)+(-90:0.15)$)
					to[out=180, in=60] ($(1)+(-30:0.15)$)
					to[out=-120, in=-30] ($(1)+(-120:0.15)$)
					to[out=150, in=-120] ($(1)+(150:0.15)$)
					to[out=60, in=180] ($(2)+(90:0.15)$);
		\end{tikzpicture}
		\captionof{figure}{The tree $\lBrack\ttt_1,\ttt_2,\ttt_3\rBrack$}
		\label{fig:concatenation}
	\end{minipage}
	\begin{minipage}[b]{.4\textwidth}
		\centering
		\begin{tikzpicture}[scale=1, every node/.style={font=\small}, thick]
			\tikzstyle{null} = [inner sep=0pt, minimum width=0pt]
			\tikzstyle{tv} = [circle, fill=odarkblue, inner sep=0pt, minimum width=4pt]
			\tikzstyle{tr} = [circle, draw=odarkblue, fill=white, inner sep=0pt, minimum width=5pt]
			\tikzstyle{te} = [draw, -, odarkblue]
			\tikzstyle{sv} = [circle, fill=odarkgreen, inner sep=0pt, minimum width=4pt]
			\tikzstyle{se} = [draw, -, odarkgreen]
			\tikzstyle{uv} = [circle, draw=odarkred, fill=white, inner sep=0pt, minimum width=5pt]
			
			\def\nodes{
				{(0,0)/0/tr},
				{($(0)+(150:1)$)/1/tv},
					{($(1)+(170:1)$)/11/tv},
					{($(1)+(110:1)$)/12/tv},
				{($(0)+(90:1)$)/2/uv},
					{($(2)+(135:1)$)/21/tv},
					{($(2)+(75:1)$)/s1/sv},
						{($(s1)+(105:1)$)/s11/sv},
						{($(s1)+(45:1)$)/s12/sv},
					{($(2)+(15:1)$)/s2/sv},
				{($(0)+(30:1)$)/3/tv}}
			\def\edges{
				0/1/te, 1/11/te, 1/12/te,
				0/2/te, 2/21/te, 2/s1/se, s1/s11/se, s1/s12/se, 2/s2/se,
				0/3/te}
			\foreach \pos/\name/\style in \nodes
				\node[null] (\name) at \pos {};
			\foreach \source/\dest/\style in \edges
				\path[\style] (\source) -- (\dest);
			
			\path[draw, gray, dashed, thin]
				($(2)+(165:0.2)$) to[out=75, in=-90] ($(s1)+(180:0.25)$)
					to[out=90, in=-75] ($(s11)+(195:0.2)$)
					to[out=105, in=195] ($(s11)+(105:0.2)$)
					to[out=15, in=135] ($(s12)+(45:0.25)$)
					to[out=-45, in=75] ($(s2)+(15:0.4)$)
					to[out=-105, in=20] ($(s2)+(-70:0.2)$)
					to[out=200, in=10] ($(2)+(-80:0.2)$)
					to[out=190, in=-45] ($(2)+(-135:0.2)$)
					to[out=135, in=-105] ($(2)+(165:0.2)$);
			
			\foreach \pos/\name/\style in \nodes
				\node[\style] at \pos {};
			
			\node [circle, fill=odarkred, inner sep=0pt, minimum width=2pt] at (2) {};
			
			\node[inner sep=2pt] (u) at ($(2)+(-10:1.2)$) {$u$};
			\path[->, >=stealth', shorten >=2pt, thin, gray] (u.west) edge[bend left] (2.south);
			
			\node at ($1/2*(s12)+1/2*(s2)$) {$\tts$};
			\node at ($(1)+(250:.4)$) {$\ttt$};
		\end{tikzpicture}
		\captionof{figure}{The tree $\ttt\otimes(u,\tts)$}
		\label{fig:grafting}
	\end{minipage}
	
\end{figure}

For all $n\geq 0$, let $\ttb_n$ be the \emph{branch} of length $n$,
\ie the tree with $n+1$ vertices among which a single leaf.
Similarly, define the infinite branch $\ttb_\infty$
and note $(\ttv_n)_{n\geq 0}$ its vertices
where $\ttv_0$ is its root and for all $n\geq 0$,
$\ttv_n = \pr(\ttv_{n+1})$.

\paragraph{The local limit topology}
If $\ttt$ is a tree, we may endow it with the \emph{graph distance} $\upd_\gr$
where for all $u$ and $v$ in $\ttt$, $\upd_\gr(u,v)$ is defined as the number of edges
in the shortest path between $u$ and $v$.
For any non-negative integer $R$, we will note $\ttt\vert_R$ the closed ball
of radius $R$ centred at the root of $\ttt$,
that is the tree $\ttt\vert_R := \{u\in\ttt : \upd_\gr(\varnothing,u) \leq R\}$.

The local distance between two given trees $\ttt$ and $\tts$
is defined as
\[
	\upd_\loc(\ttt,\tts)
	:= \exp\big[-\inf \{ R\geq 0 : \ttt\vert_R \neq \tts\vert_R \} \big].
\]
The application $\upd_\loc$ is an ultra-metric on $\ttT$
and the resulting metric space $(\ttT,\upd_\loc)$ is Polish.
The following  well-known criterion for convergence in distribution with respect
to the local limit topology will be useful.
See for instance~\cite[Section~2.2]{abraham2014gwkesten}
for a proof (which relies on \cite[Theorem~2.3]{billingsley2013convergence}
and the fact that $\upd_\loc$ is an ultra-metric).

\begin{lemma}\label{lem:d-loc-criterion}
	Let $T_n$, $n\geq 1$ and $T$ be $\ttT$-valued random variables.
	Then, $T_n\to T$ in distribution with respect to $\upd_\loc$ \tiff
	for all $\ttt\in\ttT$ and $R\geq 0$,
	$\prob {T_n\vert_R = \ttt\vert_R} \to \prob {T\vert_R = \ttt\vert_R}$
	as $n$ tends to infinity.
\end{lemma}

\subsubsection{Partitions of integers}
As discussed in the introduction, Markov branching trees are
closely related to ``partitions of integers''.
This section thus aims to introduce a few notions on these objects
which will be useful for our forthcoming purposes.

\smallskip

Set $\calP_0 := \{\varnothing\}$, $\calP_1 := \{\varnothing,(1)\}$
and for $n\geq 2$,
let $\calP_n$ be the set of \emph{partitions} of $n$,
\ie of finite non-increasing integer sequences with sum $n$.
More precisely, set
\[
	\calP_n :=
	\Big\{\lambda = (\lambda_1,\dots,\lambda_p) \in \bbN^p
		\, : \, p\geq 1,
		\, \lambda_1 \geq \dots \geq \lambda_p > 0
		\;\text {and}\;
		\lambda_1 + \dots + \lambda_p = n\Big\}.
\]
Similarly, let $\calP_\infty$ be the set of finite non-increasing $\bbN\cup\{\infty\}$-valued
sequences with infinite sum (and therefore at least one infinite part).
In other words, define
\[
	\calP_\infty :=
	\Big\{\lambda = (\lambda_1,\dots,\lambda_p) \in \big(\bbN\cup\{\infty\}\big)^p
		\, : \, p\geq 1
		\;\text {and}\; \infty = \lambda_1 \geq \dots \geq \lambda_p > 0 \Big\}.
\]
Set $\calP_{<\infty} := \bigcup_{n\geq 0} \calP_n$ and $\calP := \calP_{<\infty}\cup\calP_\infty$.

\smallskip

Let $\lambda = (\lambda_1,\dots,\lambda_p)$ be in $\calP$.
We will use the following notations:
\begin{itemize}
	\item Let $p(\lambda) := p$ be its \emph{length}
	and $\norm\lambda = \lambda_1 + \dots + \lambda_p$ its sum
	(with the conventions $p(\varnothing) = \norm\varnothing = 0$).
	\item For $k\in\bbN\cup\{\infty\}$, let $m_k(\lambda) := \sum_i \ind_{\smash {\lambda_i} = k}$
	be the number of occurrences of $k$ in the partition $\lambda$.
	\item For a non-negative integer $K$,
	set $\lambda\wedge K := (\lambda_1\wedge K, \dots, \lambda_p\wedge K)$.
	This finite partition will be called the \emph{truncation} of $\lambda$ at level $K$.
\end{itemize}

\smallskip

We endow $\calP$ with an ultra-metric distance defined similarly to $\upd_\loc$.
For all $\lambda$ and $\mu$ in $\calP$, let
\[
	\upd_\calP(\lambda,\mu)
	:= \exp \big[-\inf \, \{ K\geq 0 \,:\, \lambda\wedge K \neq \mu\wedge K \}\big].
\]

\begin{lemma}\label{lem:topology-p}
\begin{enumerate}
	\renewcommand{\labelenumi}{$(\roman{enumi})$}
	\item The application $\upd_\calP$ is an ultra-metric distance,
	\item The metric space $(\calP,\upd_\calP)$ is Polish.
\end{enumerate}
\end{lemma}

\begin{remark}
	For all $\lambda$ and $\mu$ in $\calP$ and $K\geq 0$,
	$\lambda\wedge K = \mu\wedge K$ \tiff $\upd_\calP(\lambda,\mu) < \E^{-K}$.
	In particular, $\upd_\calP(\lambda,\mu) = 1$
	\tiff $\lambda\wedge 0 \neq \mu\wedge 0$
	in which case $p(\lambda) \neq p(\mu)$.
\end{remark}

\begin{proof}
	\noindent$(i)$\quad
	Clearly, $\upd_\calP$ is symmetrical and $\upd_\calP(\lambda,\mu)=0$ \tiff $\lambda=\mu$.
	Hence, we only need to prove that $\upd_\calP$ satisfies the ultra-metric
	triangular inequality.
	Let $\lambda$, $\mu$ and $\nu$ be in $\calP$ and assume
	that $\upd_\calP(\lambda,\nu) > \upd_\calP(\lambda,\mu) \vee \upd_\calP(\mu,\nu)$.
	Then, there exists $K\geq 0$ such that $\lambda\wedge K = \mu\wedge K = \nu\wedge K$
	and $\lambda\wedge K \neq \nu\wedge K$, which is absurd.
	Consequently, $\upd_\calP(\lambda,\nu) \leq \upd_\calP(\lambda,\mu) \vee \upd_\calP(\mu,\nu)$.
	
	\smallskip
	
	\noindent$(ii)$\quad
	Observe that $\calP\subset\bigcup_{n\geq 0} (\bbN\cup\{\infty\})^n$
	and is as a result both countable and separable.
	Therefore, it only remains to show that it is complete.
	
	Let $(\lambda_n)_n$ be a Cauchy sequence with respect to $\upd_\calP$.
	By assumption, there exists an increasing sequence $(n_K)_K$
	such that for all $K\geq 0$, $\lambda_n\wedge K = \lambda_m\wedge K$ when $n,m\geq n_K$.
	In particular, there exists a constant $p\geq 0$ such that $p(\lambda_{n_{\smash K}}) = p$ for all $K$.
	Furthermore, notice that for all $i=1,\dots,p$, the sequence $[\lambda_{n_{\smash K}}(i)\wedge K]_K$
	is non-decreasing.
	For each $i = 1, \dots, p$,
	set $\lambda(i) := \sup_K \lambda_{n_{\smash K}}(i)\wedge K \leq\infty$.
	Clearly, $\lambda := [\lambda(1),\dots,\lambda(p)]$ is in $\calP$
	and is such that $\upd_\calP(\lambda_n,\lambda) \to 0$ when $n\to\infty$.
	This proves that $(\calP,\upd_\calP)$ is indeed complete.
\end{proof}

\begin{lemma}\label{lem:d-p-criterion}
	Let $(\Lambda_n)_{n\geq 1}$ and $\Lambda$ be $\calP$-valued random variables.
	Then, $\Lambda_n$ converges to $\Lambda$ in distribution with respect to $\upd_\calP$ \tiff
	for all $\lambda$ in $\calP_{<\infty}$ and all $K\geq 0$,
	we have $\prob {\Lambda_n\wedge K = \lambda\wedge K} \to \prob {\Lambda\wedge K = \lambda\wedge K}$
	as~$n\to\infty$.
\end{lemma}

\begin{proof}
	Uses the same arguments as the proof of~Lemma~\ref{lem:d-loc-criterion}
	(recall that $\upd_\calP$ is an ultra-metric
	and use \cite[Theorem~2.3]{billingsley2013convergence}).
\end{proof}

\begin{remark}
	Elements of $\calP_{<\infty}$ are closely related to elements of $\ttT$.
	Indeed, if $\ttt$ is a finite tree which can be written as the concatenation of $p$ trees $\ttt_1,\dots,\ttt_p$,
	\ie $\ttt = \lBrack\ttt_1,\dots,\ttt_p\rBrack$,
	then the \emph{partition at the root} or \emph{first split} of $\ttt$
	defined by $\Lambda(\ttt) := (\#\ttt_1, \dots, \#\ttt_p)^{\smash\downarrow}$
	is a partition of $n$ when $\ttt$ has $n+1$ vertices (the root plus $n$ descendants).
	
	Similarly, if we consider leaves instead of vertices, then
	$\Lambda^{\smash\calL}(\ttt) := (\#_\calL\ttt_1, \dots, \#_\calL\ttt_p)^{\smash\downarrow}$
	is a partition of $n$ when $\ttt$ has $n$ leaves.
\end{remark}

In this article, we will often have to consider sequences of random partitions $\Lambda_n\in\calP_n$
that will weakly converge to a limit partition $\Lambda_\infty\in\calP_\infty$
such that, $m_\infty(\Lambda_\infty) = 1$ \as.
In this particular setting, the weak convergence can be defined as follows. 

\begin{lemma}\label{lem:p-cv-one-spine-criterion}
	For all $1\leq n\leq\infty$, let $q_n$ be a probability measure on $\calP_n$
	and assume that $q_\infty (m_\infty=1) = 1$.
	Then, $q_n\Rightarrow q_\infty$ with respect to $\upd_\calP$ \tiff
	for all $\lambda$ in $\calP_{<\infty}$ we have $q_n(n-\norm\lambda,\lambda) \to q_\infty(\infty,\lambda)$
	as $n\to\infty$.
\end{lemma}

\begin{proof}
	\noindent$\Rightarrow$\quad
	Let $\lambda=(\lambda_1,\dots,\lambda_p)$ be in $\calP_{<\infty}$
	and $K > \lambda_1$.
	In light of Lemma~\ref{lem:d-p-criterion},
	\begin{align*}
		& q_n(n-\norm\lambda,\lambda)
		= q_n\big(\mu\in\calP_n : \mu\wedge K = (K,\lambda)\wedge K\big)\\
		& \qquad\qquad \xrightarrow [n\to\infty] {}
		q_\infty\big(\mu\in\calP_\infty : \mu\wedge K = (K,\lambda)\wedge K\big)
		= q_\infty(\infty,\lambda).
	\end{align*}

	\smallskip
	
	\noindent$\Leftarrow$\quad
	For fixed $K\geq 0$ and $\lambda$ in $\calP_{<\infty}$,
	Fatou's lemma ensures that
	\begin{align*}
		& \liminf_{n\to\infty} \; q_n\big(\mu\in\calP_n : \mu\wedge K = \lambda\wedge K\big)
		= \liminf_{n\to\infty} \; {\textstyle\sum_{\nu\in\calP_{<\infty}}}
			\ind_{(\infty,\nu)\wedge K = \lambda\wedge K} \: q_n(n-\norm\nu,\nu)\\
		& \qquad\qquad\qquad\qquad
		\geq {\textstyle\sum_{\nu\in\calP_{<\infty}}}
			\ind_{(\infty,\nu)\wedge K = \lambda\wedge K} \: q_\infty(\infty,\nu)
		= q_\infty\big(\mu\in\calP_\infty : \mu\wedge K = \lambda\wedge K\big).
	\end{align*}
	Similarly,
	\[
		\liminf_{n\to\infty} \; q_n\big(\mu\in\calP_n : \mu\wedge K \neq \lambda\wedge K\big)
		\geq q_\infty\big(\mu\in\calP_\infty : \mu\wedge K \neq \lambda\wedge K\big).
	\]
	As a result and thanks to Lemma~\ref{lem:d-p-criterion}, we get that $q_n\Rightarrow q_\infty$.
\end{proof}

\subsection{The Markov-branching property}

\subsubsection{Finite Markov branching trees}\label{sec:finite-mb-trees}

We will now follow~\cite[Section 1.2]{haas2012scaling}
and define two types of family of probability measures on the set of finite unordered rooted trees,
satisfying the Markov branching property discussed in the \hyperref[sec:intro]{Introduction}.

Fix $q = (q_n)$ a sequence of probability measures respectively supported by $\calP_{n-1}$
(referred to as ``first-split distributions'' in the \hyperref[sec:intro]{Introduction}).
We will define a sequence $\MB^q = (\MB^q_n)_n^{\vphantom q}$ of probability measures
on the set of finite trees where
\begin{itemize}
	\item For all $n$, $\MB^q_n$ is supported by the set of trees with $n$ vertices,
	\item A tree $T$ with distribution $\MB^q_n$ is such that
	\begin{itemize}
		\item The decreasing rearrangement $\Lambda(T)$ of the sizes of the sub-trees
		above its root is distributed according to $q_{n-1}$,
		\item Conditionally on $\Lambda(T) = (\lambda_1,\dots,\lambda_p)$,
		the $p$ sub-trees of $T$ above its root
		are independent with respective distributions $\MB^q_{\lambda_{\smash i}}$.
	\end{itemize}
\end{itemize}
We will also define a sequence $\MB^{\smash\calL,q}$
satisfying the same Markov branching property where we count
leaves instead of vertices to measure the size of a tree.

\paragraph{Markov branching tree with $n$ vertices}\label{cond:mb-trees-vertices}
First of all, set $\calN$ an infinite subset of $\bbN$ with $1\in\calN$.
This set will index the possible number of vertices of the trees we want to generate,
which is why we need $1$ to belong to~$\calN$.
Let $q = (q_{n-1})_{n\in\calN}$ be a sequence of probability measures
such that $q_0(\varnothing) = 1$,
$q_1[(1)] = 1$ (if $2\in\calN$),
and for all $n$ in $\calN$, $n\geq 2$, $q_{n-1}$ is supported by the set
$\{\lambda\in\calP_{n-1} \,:\, \lambda_i\in\calN, \, i=1,\dots,p(\lambda)\}$.

\begin{remark}
	This last condition comes from the fact that if $T$ is distributed according to $\MB^q_n$,
	the blocks of $\Lambda(T)$ need to be in $\calN$
	because the distribution of the corresponding sub-trees 
	belong to the family~$(\MB^q_k)^{\vphantom q}_{k\in\calN}$.
\end{remark}

We now detail a recursive construction for $\MB^q$.
Let $\MB^q_1(\{\varnothing\}) = 1$
and for $n\geq 2$, proceed by a decreasing induction as follows:
\begin{itemize}
	\item Let $\Lambda$ have distribution $q_{n-1}$,
	\item Conditionally on $\Lambda = (\lambda_1,\dots,\lambda_p)\in\calP_{n-1}$,
	let $(T_1, \dots, T_p)$ be independent random trees such that $T_i$
	is distributed according to $\MB^q_{\lambda_{\smash i}}$ for each $1\leq i\leq p$,
	\item Define $\MB^q_n$ as the law of the concatenation of these trees,
	\ie that of $\lBrack T_1, \dots, T_{\smash {p(\Lambda)}} \rBrack$.
\end{itemize}

\begin{figure}[ht]
	\centering
		\begin{tikzpicture}[scale=.60, every node/.style={font=\small}, thick]
		\tikzstyle{root} = [circle, inner sep=0pt, minimum width=6pt, fill=white, draw=odarkblue]
		\tikzstyle{vertex} = [circle, inner sep=2pt, minimum width=6pt, fill=white, draw=odarkgreen]
		\tikzstyle{edge} = [draw, -, odarkblue]
		
		\def\nodes{{(0,0)/0}, {(-3.5,1.5)/1}, {(-1,1.5)/2}, {(3.5,1.5)/p}}
		\foreach \pos/\name in \nodes
			\node[inner sep=0pt] (\name) at \pos {};
		
		\node (etc) at ($0.5*(2)+0.5*(p)$) {$...$};
		\path[edge] (0) edge (1)
				edge (2)
				edge (p);
		\path[edge, densely dotted] (0)
				edge ($3/7*(0,1.8)$)
				edge ($3/7*(1,1.5)$);
		
		\node[inner sep=2pt] (lambda) at ($(0)+(-15:1.8)$) {\rlap {$\lambda$}};
		\path[->, >=stealth', gray, thin]
			(0) edge[bend right] node[align=center, black] {$q_{\mathrlap {n-1}}$\\} (lambda.west);
		
		\node[root] at (0,0) {};
		
		\foreach \name/\scale in {1/1.1, 2/.9, p/.75} {
			\node[circle, inner sep=0pt, minimum width=6pt, draw=odarkgreen, fill=odarkgreen, align=center] at (\name) {};
			\path[scale=\scale, draw=odarkgreen, fill=white] (\name) to[out=135,in=-90] ++(-1,1.5)
				to[out=90,in=180] ++(1,1)
				to[out=0,in=90] ++(1,-1)
				to[out=-90,in=45] ++(-1,-1.5);
			\node[circle, inner sep=0pt, minimum width=6pt-\the\pgflinewidth, fill=white,] at (\name) {};
			\node at ($(\name)+\scale*(0,1.5)$) {\footnotesize{$\;\MB^q_{\lambda_{\mathrlap{\smash\name}}}$}};
			};
		
		\node[inner sep=2pt] (ind) at ($0.3*(0)+0.7*(etc)$) {$\Vbar$};
		\path[->, >=stealth', dashed, thin, gray]
			(ind.west) edge[bend left] ($(1)+0.8*(45:1)$)
			(ind.north) edge[bend right] ($(2)+0.8*(45:0.8)$)
			(ind.east) edge[bend right] ($(p)+0.8*(135:0.65)$);
		
		\node at (-5,0) {};
		\node at (5,0) {};
	\end{tikzpicture}
	\caption{The construction of a tree with distribution $\MB^q_n$}
\end{figure}
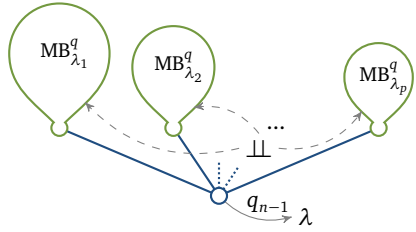

\paragraph{Markov branching tree with $n$ leaves}\label{cond:mb-trees-leaves}
Similarly, fix an infinite subset $\calN$ of $\bbN$ such that $1\in\calN$
(corresponding to the possible number of leaves of the trees we will generate)
and let $q = (q_n)_{n\in\calN}$ be such that for all~$n$ in~$\calN$,
$q_n$ is a probability measure supported by the set
$\{\lambda\in\calP_n\,:\,\lambda_i\in\calN,\, i=1,\dots,p(\lambda)\}$.

To define $\MB^{\calL,q}$, we will proceed by the same recursive method used for $\MB^q$:
first choose how the mass is shared between the children sub-trees of the root,
ans then generate the said sub-trees adequately.
However, if for some $n$ in $\calN$ we have $q_n(n) = 1$,
the recursion will be endless.
For this reason, we also require that for all $n$ in $\calN$,
$q_n(n) < 1$ (\ie with positive probability, a tree ``splits'' into smaller trees).

Let $\MB^{\smash\calL,q}_1$ be the distribution of a branch of geometric length with parameter $1-q_1(1)$,
\ie $\MB^{\smash\calL,q}_1 (\ttb_k) = q_1(1)^k [1-q_1(1)]$ for all $k\geq 0$.
For $n>1$, we do as follows:
\begin{itemize}
	\item Let $T_0$ be a branch with geometric length with parameter $1-q_n(n)$
	and call $U$ its leaf,
	\item Let $\Lambda$ have distribution $q_n$ conditioned on the event $\{m_n = 0\}$,
	\item Conditionally on $\Lambda = (\lambda_1,\dots,\lambda_p)$,
	let $(T_1, \dots, T_p)$ be independent random trees respectively distributed according to
	$\MB^q_{\lambda_{\smash i}}$ for $1\leq i\leq p$,
	\item Graft the concatenation of these trees on the leaf $U$ of $T_0$,
	\ie set $T := T_0 \otimes \big(U, \lBrack T_1, \dots, T_{\smash {p(\Lambda)}} \rBrack\big)$
	and let $\MB^{\smash\calL,q}_n$ be the distribution of $T$.
\end{itemize}

\subsubsection{Infinite Markov branching trees}\label{sec:infinite-mb-trees}
Using the same principle as before (split the mass above the root and generate independent sub-trees with corresponding sizes)
we will define a probability measure supported by the set of infinite trees
which satisfies a version of the Markov branching property.
Let $\calN$ and $q = (q_{n-1})_{n\in\calN}$
satisfy the conditions exposed in the construction of the sequence $\MB^q$.

In order to lighten notations,
for any finite decreasing sequence of integers $\lambda = (\lambda_1,\dots,\lambda_p)$,
we define $\smash {\MB^q_\lambda}$ as the distribution of the concatenation
of independent $\smash {\MB^q_{\lambda_i}}$-distributed trees.
More precisely:
\begin{itemize}
	\item Let $\smash {\MB^q_\varnothing}$ be the Dirac measure on the tree with a single vertex (its root),
	namely $\smash {\MB^q_\varnothing} = \delta_{\{\varnothing\}}$,
	\item For any $\lambda\in\calP_{<\infty}$
	with $p = p(\lambda) > 0$ and $\lambda_i\in\calN$ for $i=1,\dots,p$,
	let $(T_1,\dots,T_p)$ be independent trees with respective distributions $\smash {\MB^q_{\lambda_i}}$
	for all $i=1,\dots, p$.
	Set $\MB^q_\lambda$ as the distribution of the concatenation of these trees.
\end{itemize}
Observe that when $p(\lambda) = 1$,
a tree with distribution $\smash {\MB^q_\lambda}$
is obtained by attaching an edge ``under'' the root of a $\smash {\MB^q_{\lambda_1}}$-distributed tree.

\smallskip

Consider $q_\infty$, a probability measure on $\calP_\infty$ supported by the set
\[
	\big\{\lambda\in\calP_\infty : \lambda_i\in\calN\cup\{\infty\}, i=1,\dots,p(\lambda) \big\}
\]
and let $\Lambda$ follow $q_\infty$.
Let $T^\circ$ be a Galton-Watson tree with offspring distribution the law of $m_\infty(\Lambda)$.
Conditionally on $T^\circ$, let $(\Lambda_u,T_u)_{u\in T^{\smash\circ}}$ be independent pairs and such that:
\begin{itemize}
	\item $\Lambda_u$ has the same distribution as $\Lambda$
	conditioned on the event $m_\infty(\Lambda) = c_u (T^\circ)$,
	\item Conditionally on $\Lambda_u = (\infty,\dots,\infty,\lambda)$ with $\lambda$ in $\calP_{<\infty}$,
	$T_u$ follows $\MB^q_\lambda$.
\end{itemize}
Then, for every vertex $u$ in $T^\circ$,
graft the corresponding tree $T_u$ on $T^\circ$ at $u$.
Let $T$ be the tree hence obtained,
\ie set $T := T^\circ \bigotimes_{u\in T^{\smash\circ}} (u, T_u)$.
Finally, call $\MB^{q,q_{\smash\infty}}_\infty$ the distribution of $T$.

\begin{remark}\label{nb:infinite-spine-construction}
\begin{itemize}
	\item Suppose that $q_\infty(m_\infty=1) = 1$.
	In this case, the construction of $\MB^{q,q_{\smash\infty}}_\infty$ is much simpler:
	the tree $T^\circ$ is simply the infinite branch
	and the family $(\Lambda_{\ttv_{\smash n}}, T_{\ttv_{\smash n}})_{n\geq 0}$ is \iid.
	In particular, $T$ \as has a unique \emph{infinite spine},
	\ie a unique infinite non-backtracking path originating from the root.
	
	\item A tree $T$ with distribution $\MB^{q,q_{\smash\infty}}_\infty$
	satisfies the Markov branching property:
	conditionally on $\Lambda(T)$, the sub-trees of $T$ above its root
	are independent and their respective distributions 
	are either $\MB^{q,q_{\smash\infty}}_\infty$ or in the family $(\MB^q_n)_{n\in\calN}^{\vphantom q}$,
	depending on their sizes.
	
	\item The same exact construction can be used to define a measure
	$\MB^{\smash\calL,q,q_{\smash\infty}}_\infty$.
\end{itemize}
\end{remark}

\subsection{Local limits of Markov-branching trees}\label{sec:mb-trees-local-limits}

Let $q$ be the sequence of first-split distributions
associated to a Markov-branching family $\MB^q$ (respectively $\MB^{\calL,q}$).
Suppose $q_\infty$ is a probability measure on $\calP_\infty$ supported by the set
of sequences $\lambda$ such that for all $i=1,\dots,p(\lambda)$,
$\lambda_i$ is either infinite or in $\calN$.
The aim of this section is to expose suitable conditions on $q$ and $q_\infty$
such that $\MB^q_n$ converges weakly to $\MB^{q,q_{\smash\infty}}_\infty$
(or $\MB^{\smash\calL,q}_n \Rightarrow \MB^{\smash\calL,q,q_{\smash\infty}}_\infty$)
for the local limit topology.

\begin{theorem}\label{thm:local-limits-markov-branching}
	Suppose that when $n$ goes to infinity, $q_n$ converges weakly to $q_\infty$
	with respect to the topology induced by $\upd_\calP$.
	Then, with respect to $\upd_\loc$,
	$\MB^q_n \Rightarrow \MB^{q,q_{\smash\infty}}_\infty$
	(respectively $\MB^{\smash\calL,q}_n \Rightarrow \MB^{\smash\calL,q,q_{\smash\infty}}_\infty$).
\end{theorem}

In many cases, the infinite trees we will consider will have a unique infinite spine,
which corresponds to $q_\infty(m_\infty=1)=1$
and the particular construction mentioned in Remark~\ref{nb:infinite-spine-construction}.
In this situation, we may use Theorem~\ref{thm:local-limits-markov-branching} alongside Lemma~\ref{lem:p-cv-one-spine-criterion}
to get the following corollary.

\begin{corollary}\label{cor:local-limits-mb-one-spine}
	Assume that $q_\infty$ is such that $q_\infty(m_\infty=1)=1$
	and suppose that for any finite partition $\lambda$ in $\calP_\infty$
	we have $q_n(n-\norm\lambda,\lambda) \to q_\infty(\infty,\lambda)$.
	Then, $\MB^q_n \Rightarrow \MB^{q,q_{\smash\infty}}_\infty$
	(or $\MB^{\smash\calL,q}_n \Rightarrow \MB^{\smash\calL,q,,q_{\smash\infty}}_\infty$)
	with respect to the local limit topology.
\end{corollary}

\begin{proof}[of~Theorem~\ref{thm:local-limits-markov-branching}]
	For all $n$ in $\calN\cup\{\infty\}$, let $T_n$ follow $\MB^q_n$.
	We will use Lemma~\ref{lem:d-loc-criterion} and proceed by induction on $R$. 
	First, it clearly holds that for every tree $\ttt$,
	$\ttt\vert_0 = \{\varnothing\} = T_n\vert_0 = T_\infty\vert_0$ \as.
	
	\smallskip
	
	Let $R$ be a non-negative integer and suppose that for any $\tts\in\ttT$,
	$\prob {T_n\vert_R = \tts\vert_R} \to \prob {T_\infty\vert_R = \tts\vert_R}$ as $R\to\infty$.
	Fix $\ttt\in\ttT$ and set $d := c_\varnothing(\ttt)$, the number of children of its root.
	We may write $\ttt\vert_{R+1} = \lBrack\ttt_1,\dots,\ttt_d\rBrack$
	for some $\ttt_1,\dots,\ttt_d$ in $\ttT$ with height $R$ or less.
	
	The labelling of $\ttt_1,\dots,\ttt_d$ such that
	$\ttt\vert_{R+1} = \lBrack\ttt_1,\dots,\ttt_d\rBrack$ is arbitrary
	and creates a kind of order between the children vertices of the root of $\ttt$.
	As a result, we need to consider a subset $S$ of the set of permutations on $\{1,\dots, d\}$
	such that for every permutation $\sigma$ there is a unique $\tau\in S$
	such that for all $i=1,\dots,d$, $\ttt_{\sigma\cdot i} = \ttt_{\tau\cdot i}$
	as elements of $\ttT$.
	Then for all $n$ in $\calN\cup\{\infty\}$,
	it ensues from the Markov-branching nature of $T_n$ that
	\[
		\prob [1] {T_n\vert_{R+1} = \ttt\vert_{R+1}}
		= \int_\calP \sum_{\sigma\in S} \bigg({\textstyle\prod_{i=1}^d}
			\prob [1] {T_{\lambda_i}\vert_R = \ttt_{\sigma\cdot i}} \bigg) \, \ind_{p(\lambda)=d}
			\, q_{n-1}(\D\lambda).
	\]
	Our induction assumption ensures that for all $i=1,\dots,d$
	and $\tts$ in $\ttT$ with height $R$ or less,
	the application $\calP\to[0,1]$, $\lambda\mapsto \prob {T_{\lambda_{\smash i}}\vert_R = \tts} \,\ind_{p(\lambda)=d}$
	is continuous.
	As a result, $\prob {T_n\vert_{R+1} = \ttt\vert_{R+1}}$ may be expressed
	as the integral against $q_{n-1}$ of a finite sum of continuous functions.
	Therefore, since $q_n\Rightarrow q_\infty$,
	\[
		\prob [1] {T_n\vert_{R+1} = \ttt\vert_{R+1}}
		\xrightarrow [n\to\infty] {}
		\prob [1] {T_\infty\vert_{R+1} = \ttt\vert_{R+1}}.
	\]
	We proceed in the same way to prove the claim on $\MB^{\smash\calL,q}$ trees.
\end{proof}

In the next proposition, we prove that the condition ``$q_n\Rightarrow q_\infty$''
in Theorem~\ref{thm:local-limits-markov-branching} is optimal for $\MB^q$~trees.

\begin{proposition}
	Let $q = (q_{n-1})_{n\in\calN}$ be the sequence of first split distributions associated
	to a family $\MB^q$ of Markov branching trees with given number of vertices.
	If there exists a probability measure $q_\infty$ on $\calP_\infty$
	such that $\MB^q_n$ converges weakly to $\MB^{q,q_{\smash\infty}}_\infty$ for the local limit topology,
	then $q_{n-1}\Rightarrow q_\infty$ in the sense of the $\upd_\calP$ topology.
\end{proposition}

\begin{proof}
	Observe that for all $K\geq 0$ and $\ttt,\tts\in\ttT$,
	if $\ttt\vert_K = \tts\vert_K$ then $\Lambda (\ttt) \wedge K = \Lambda (\tts) \wedge K$.
	As a result, $\upd_\calP[\Lambda(\ttt), \Lambda(\tts)] \leq \upd_\loc(\ttt,\tts)$
	which proves in particular that the application $\Lambda:\ttT\to\calP$ is continuous.
	Consequently, since for all possibly infinite $n$, $\Lambda(T_n)$ has distribution $q_{n-1}$,
	in the sense of the $\upd_\calP$ topology we have $q_{n-1}\Rightarrow q_\infty$ when $n\to\infty$.
\end{proof}

\section{Background on scaling limits}\label{sec:ghp-topology-ssf-trees}

In this section, we will introduce the framework needed to consider the scaling limits
of both finite and infinite Markov branching trees as well as the corresponding limiting objects:
self-similar fragmentation trees with or without immigration.
Afterwards, we will also give a few useful results on point processes related
to our models of trees.

\subsection{\texorpdfstring{$\bbR$}{R}-trees and the GHP topology}\label{sec:ghp-topology}

To talk about scaling limits of discrete trees,
we need to introduce a continuous analogue.
We use the framework of $\bbR$-trees.
An \emph{$\bbR$-tree} (or \emph{real tree}) is a metric space $(T,d)$
such that for all $x$ and $y$ in $T$:
\begin{itemize}
	\item There exists a unique isometry $\varphi:[0,d(x,y)]\to T$
	such that $\varphi(0)=x$ and $\varphi[(d(x,y)]=y$,
	\item If $\gamma:[0,1]\to T$ is a continuous injection
	with $\gamma(0)=x$ and $\gamma(1)=y$,
	then $\Im\gamma = \Im\varphi =: \lBrack x,y \rBrack$.
\end{itemize}
This roughly means that any two points in an $\bbR$-tree
can be continuously joined by a single path, up to its reparametrisation,
which is akin to the acyclic nature of discrete trees.

To compare two such objects, we will use the Gromov-Hausdorff-Prokhorov distance.
More precisely, we will follow the definition from~\cite{abbgm2013minspantree}
and extend it in a way similar to that of~\cite{abraham2013ghp}.

\medskip

For any metric space $(X,\upd)$ let $\calM_f(X)$ be the set of all finite non-negative Borel measures on $X$
and $\calM(X)$ be the set of all non-negative and \emph{boundedly finite} Borel measures on $X$,
\ie non-negative Borel measures $\mu$ on $X$ such that $\mu(A) < \infty$ for all measurable bounded $A\subset X$.

A \emph{pointed} metric space is a $3$-tuple $(X,\upd,\rho)$ where $(X,\upd)$ is a metric space
and $\rho\in X$ is a fixed point, which we will call its \emph{root}.
For any $x\in X$, set $\abs x := \upd(\rho,x)$ the \emph{height} of $x$ in $(X,\upd,\rho)$,
and let $\abs X := \sup_{x\in X} \abs x$ be the height of $X$.

We will call \emph{pointed weighted metric space}
any $4$-tuple $\bfX = (X,\upd,\rho,\mu)$ where $(X,\upd)$ is a metric space,
$\rho\in X$ is its root and $\mu$ is a boundedly finite Borel measure on $X$.

\begin{remark}
	If $\bfX$ is a pointed weighted metric space,
	we will implicitly note $\bfX = (X,\upd_X,\rho_X,\mu_X)$
	unless otherwise stated.
\end{remark}

Two pointed weighted metric spaces $\bfX$ and $\bfY$
will be called \emph{GHP-isometric} if there exists a bijective isometry $\Phi:X\to Y$
such that $\Phi(\rho_X) = \rho_Y$ and $\mu_X\circ\Phi^{-1} = \mu_Y$.
Let $\bbK$ be the set of GHP-isometry classes of \emph{compact} pointed weighted metric spaces.

\subsubsection{Comparing compact metric spaces}\label{sec:compact-ghp-topology}
Let $\bfX$ and $\bfY$ be two pointed weighted compact metric spaces.
A \emph{correspondence} between $\bfX$ and $\bfY$
is a measurable subset $C$ of $X\times Y$ which contains $(\rho_X,\rho_Y)$
such that for any $x\in X$ there exists $y\in Y$ with $(x,y)\in C$
and conversely, for any $y\in Y$ there is $x\in X$ such that $(x,y)\in C$.
We will denote by $\upC(\bfX,\bfY)$ (or $\upC(X,Y)$ with a slight abuse of notation)
the set of all pointed correspondences between $\bfX$ and $\bfY$.
For any $C\in \upC(\bfX,\bfY)$, let its \emph{distortion} be defined as follows:
\[
	\dis_{\bfX,\bfY} C
	:= \sup \big\{ \abs {\upd_X(x,x') - \upd_Y(y,y')} \,:\, (x,y),(x',y')\in C \big\}.
\]
When the setting is clear, we will simply note $\dis C := \dis_{\bfX,\bfY} C$.
Observe that $\dis C \leq 2 \, \big( \abs X \vee \abs Y \big) < \infty$
and that $\dis C \geq \abs [1] {\abs X - \abs Y}$.

\smallskip

For any finite Borel measure $\pi$ on $X\times Y$,
we define its \emph{discrepancy} with respect to $\mu_X$ and $\mu_Y$ as:
\[
	\upD(\pi;\mu_X,\mu_Y)
	:= \norm {\mu_X - \pi\circ p_X^{-1}}_\TV
	+ \norm {\mu_Y - \pi\circ p_Y^{-1}}_\TV
\]
where $\norm {\,\cdot\,}_\TV$ is the total variation norm,
and $p_X:(x,y)\in X\times Y \mapsto x$, $p_Y:(x,y)\in X\times Y \mapsto y$
are the canonical projections from $X\times Y$ to $X$ and $Y$ respectively.
The definition of the total variation norm and the triangular inequality
give $\upD(\pi;\mu_X,\mu_Y) \geq \abs {\mu_X(X) - \mu_Y(Y)}$.

\smallskip

Following~\cite[Section 2.1]{abbgm2013minspantree},
we define the \emph{Gromov-Hausdorff-Prokhorov} distance (or GHP distance for short)
between two pointed weighted compact metric spaces $\bfX$ and $\bfY$ as:
\[
	\upd_\GHP(\bfX,\bfY)
	:= \inf \bigg\{ \frac 1 2 \dis C \vee \upD(\pi;\mu_X,\mu_Y) \vee \pi(C^c)
	\: : \: C\in\upC(X,Y), \pi\in\calM(X\times Y) \bigg\}
\]
where $C^c = X\times Y\setminus C$.

\begin{remark}\label{nb:height-mass-ghp-continuity}
	Observe that $\upd_\GHP(\bfX,\bfY) \leq \big(\abs X \vee \abs Y\big) \vee \big(\mu_X(X) + \mu_Y(Y)\big)$
	and is consequently finite.
	Moreover, $\upd_\GHP(\bfX,\bfY) \geq \big(1/2\cdot \abs [1] {\abs X - \abs Y}\big) \vee \abs [1] {\mu_X(X) - \mu_Y(Y)}$.
	Therefore, the applications $\bbK\to\bbR_+$,
	$\bfX\mapsto\abs\bfX$ and $\bfX\mapsto\mu_X(X)$ are both continuous
	with respect to $\upd_\GHP$.
\end{remark}

As was mentioned in~\cite[Section 2.1]{abbgm2013minspantree},
$\upd_\GHP$ is a well-defined distance on $\bbK$ which gives rise to the same topology
as the GHP distance defined in~\cite{abraham2013ghp}.
As a result and thanks to~\cite[Theorem~2.5]{abraham2013ghp},
$(\bbK,\upd_\GHP)$ is completely metrisable and separable.
It is therefore Polish.

\paragraph{Rescaling compact metric spaces}
For all $m\geq 0$, let $\mathbf 0^{(m)} := \big(\{\varnothing\},d,\varnothing,m\delta_\varnothing\big) \in \bbK$
be the degenerate metric space only made out of its root on which a mass $m$ is put.
For a pointed weighted metric space $\bfX$
and any non-negative real numbers $a$ and $b$,
we will note $(aX,b\mu_X) := (X,a\upd_X,\rho_X,b\mu_X)$.
When $\bfX$ is in $\bbK$ and $\mu_X(X)=m$,
we will use the convention $(0X,\mu_X) = \mathbf 0^{(m)}$
(which makes sense since $(\varepsilon X,\mu_X)$ converges to $\mathbf 0^{(m)}$
as $\varepsilon$ goes to $0$ with respect to $\upd_\GHP$).

\begin{lemma}\label{lem:ghp-rescaling-bounds}
	Let $\bfX$ and $\bfY$ be two elements of $\bbK$.
	For any non-negative real numbers $a$, $b$, $c$ and $d$:
	\begin{align*}
		(i) \quad & \upd_\GHP\big( (a X, b \mu_X), (c X, d \mu_X)\big)
		\leq \big( \abs {a-c} \, \abs {X} \big) \vee \big(\abs {b-d} \, \mu_X(X)\big),\\
		\mathllap{\text{and}\qquad\qquad} (ii) \quad &
		\upd_\GHP\big((a X, b \mu_X), (a Y, b \mu_Y)\big) \leq (a\vee b) \, \upd_\GHP(\bfX,\bfY).
	\end{align*}
\end{lemma}

\begin{proof}
	\noindent$(i)$\quad
	Let $C = \{(x,x) : x\in X\} \in \upC(X,X)$.
	We have
	\[
		\dis_{(aX,b\mu_X), \, (cX,d\mu_X)} C
		= \sup \big\{ \abs {a \, \upd_X(x,y) - c \, \upd_X(x,y)} : x,y\in X \big\}
		\leq 2 \abs {a-c} \, \abs X.
	\]
	Let $\pi\in\calM(X\times X)$ be defined for all measurable $A\subset X\times X$
	by $\pi(A) = \int_X \ind_A\big((x,x)\big) \, b\mu_X(\D x)$.
	Then $\upD(\pi; b\mu_X, d\mu_X) = \abs {b-d} \mu_X(X)$
	and $\pi(C^c) = 0$.
	
	\smallskip
	
	\noindent$(ii)$\quad
	For every correspondence $C\in\upC(X,Y)$, we clearly have
	$\dis_{(aX,b\mu_X), \, (aY,b\mu_Y)} C = a \dis_{\bfX,\bfY} C$.
	No less clearly, for any finite measure $\pi$ on $X\times Y$,
	$\upD(b\pi;b\mu_X,b\mu_Y) = b\upD(\pi;\mu_X,\mu_Y)$.
\end{proof}

\begin{corollary}
	The application $\bbK \times \bbR_+ \times \bbR_+ \longrightarrow \bbK$
	defined by $(\bfX,a,b) \longmapsto (aX,b\mu_X)$
	is continuous for the product topology.
\end{corollary}

\paragraph{Concatenated compact metric spaces}
Let $(\bfX_i)_{i\in\calI}$ be a countable family of pointed weighted metric spaces
with $\bfX_i = (X_i, \upd_i, \rho_i, \mu_i)$.
Let $(X,\upd,\rho,\mu)$ where:
\begin{itemize}
	\item $X = \{\rho\} \sqcup \bigsqcup_{i\in\calI} X_i$,
	\item $\upd$ is defined by:
	\begin{itemize}
		\item For all $i,j\in\calI$, $\upd(\rho, \rho_i) := \upd(\rho_i, \rho_j) = 0$,
		\item For all $i\in\calI$, and $x,y\in X_i$, $\upd(x,y) := \upd_i(x,y)$,
		\item For all $i\neq j$ and $x\in X_i$, $y\in X_j$,
		$\upd(x,y) := \upd_i(x,\rho_i) + \upd_j(y,\rho_j)$,
	\end{itemize}
	\item For any Borel subset $A$ of $X$,
	$\mu (A) = \sum_{i\in\calI} \mu_i(A\cap X_i)$.
\end{itemize}
With a slight abuse of notation, we will consider $(X,\upd)$ to be the quotient metric space
$X/\sim_\upd$ where $x\sim_\upd y$ \tiff $\upd(x,y) = 0$.
For each $i$ in $\calI$, we will also identify $X_i$ with its image in $X$ by the quotient map.
Note $\bfX =: \langle \bfX_i \,;\, i\in\calI\rangle$.

\begin{remark}
	If $(\bfT_i)_{i\in\calI}$ is a countable family of weighted  $\bbR$-trees,
	then $\langle\bfT_i \,;\, i\in\calI\rangle$ is clearly an $\bbR$-tree itself.
\end{remark}

\begin{lemma}\label{lem:concatenation-compactness}
	For all $i\geq 1$, let $\bfX_i = (X_i,\upd_i,\rho_i,\mu_i)$ be in $\bbK$.
	Their concatenation $\langle\bfX_i \,;\, i\geq 1\rangle$
	is an element of~$\bbK$ \tiff
	the height $\abs {X_i}$ of $X_i$ goes to $0$ as $i$ goes to infinity
	and $\sum_{i\geq 1} \mu_i(X_i)$ is finite.
\end{lemma}

\begin{proof}
	Set $\bfX := \langle\bfX_i \,;\, i\geq 1\rangle$
	and for all $x$ in $X$ and positive $r$, note $\upB_X(x,r) := \{y\in X : \upd_X(x,y)<r\}$
	the open ball of $X$ centred at $x$ with radius $r$.
	Similarly,  for all $i\geq 1$ and $x\in X_i$, note $\upB_i(x,r) := \{y\in X_i : \upd_i(x,y)<r\}$.
	Clearly, the measure $\mu_X$ is finite \tiff
	the sum $\sum_{i\geq 1} \mu_i(X_i)$ is.
	
	If $\abs {X_i}\to 0$, then in particular, for all positive $\varepsilon$,
	there exists a integer $n$ such that $\bigcup_{i>n} X_i \subset \upB_X(\rho_X,\varepsilon)$.
	Moreover, since $X_i$ is compact for all $i=1,\dots,n$,
	we can find a finite $\varepsilon$-cover of $X_i$,
	\ie a finite subset $A_i$ of $X_i$ such that
	$X_i \subset \bigcup_{x\in A_{\smash i}} \upB_i(x,\varepsilon)$.
	Set $A := \{\rho_X\} \cup A_1 \cup \dots \cup A_n$.
	Observe that it is finite and that $X \subset \bigcup_{x\in A} \upB_X(x,\varepsilon)$.
	Since this holds for all positive $\varepsilon$, it follows that $X$ is compact.
	
	If $\limsup \abs {X_i} > 0$, then there exists a positive $\varepsilon$
	such that $\abs {X_i} > \varepsilon$ for infinitely many indices $i$.
	As a result, $X$ cannot have a finite $\varepsilon$-cover,
	which implies that it is not compact.
\end{proof}

\begin{lemma}\label{lem:continuous-concatenation-ghp-bounds}
	Let $\bfX_i$, $\bfY_i$, $i\geq 1$ be in $\bbK$ and
	such that $\bfX := \langle\bfX_i \,;\, i\geq 1\rangle$ and $\bfY := \langle\bfY_i \,;\, i\geq 1\rangle$
	both belong to $\bbK$.
	We have
	\[
		\textstyle
		\upd_\GHP\Big(
			\langle\bfX_i \,;\, i\geq 1\rangle,
			\langle\bfY_i \,;\, i\geq 1\rangle\Big)
		\leq \sum_{i\geq 1} \upd_\GHP(\bfX_i,\bfY_i).
	\]
\end{lemma}

\begin{proof}
	Set $\bfX := \langle\bfX_i \,;\, i\geq 1\rangle$ and $\bfY := \langle\bfY_i \,;\, i\geq 1\rangle$.
	For all positive $\varepsilon$ and $i\geq 1$,
	there exists a correspondence $C_i$ in $\upC(\bfX_i,\bfY_i)$
	and a finite Borel measure $\pi_i$ on $X_i\times Y_i$
	such that
	\[
		\frac 1 2 \dis C_i \vee \upD(\pi_i;\mu_{X_i},\mu_{Y_i}) \vee \pi_i(C_i^c)
		< \upd_\GHP(\bfX_i,\bfY_i) + 2^{-i}\varepsilon.
	\]
	
	Set $C := \bigcup_{i\geq 1} C_i$,
	which is a correspondence between $\bfX$ and $\bfY$.
	Let $(x,y)$ and $(x',y')$ be in $C$.
	If both $(x,y)$ and $(x',y')$ are in $C_i$ for some $i$,
	then clearly, $\abs {\upd_X(x,x') - \upd_Y(y,y')} \leq \dis C_i$.
	Otherwise, if $(x,y)\in C_i$ and $(x',y')\in C_j$ with $i\neq j$,
	then using the definition of $\upd_X$ and $\upd_Y$ as well as the triangular inequality,
	we get $\abs {\upd_X(x,x') - \upd_Y(y,y')} \leq \dis C_i + \dis C_j$.
	Therefore, $1/2 \cdot \dis C \leq \sum_{i\geq 1} \upd_\GHP(\bfX_i,\bfY_i) + \varepsilon$.
	
	For all $n\geq 0$, define the finite Borel measure $\pi^{(n)}$ on $X\times Y$
	by $\pi^{(n)}(A) := \sum_{i=1}^n \pi_i\big[A\cap(X_i\times Y_i)\big]$
	for any Borel set $A$.
	By definition,
	\[
		\textstyle
		\pi^{(n)}(C^c)
		= \sum_{i=1}^n \pi_i\big[C^c\cap(X_i\times Y_i)\big]
		= \sum_{i=1}^n \pi_i [C_i^c]
		\leq \sum_{i\geq 1} \upd_\GHP(\bfX_i,\bfY_i) + \varepsilon.
	\]
	Moreover, the discrepancy of $\pi^{(n)}$ with respect to $\mu_X$ and $\mu_Y$ satisfies
	\begin{align*}
		\upD(\pi^{(n)};\mu_X,\mu_Y)
		& \leq \textstyle\sum_{i=1}^n
			\norm {\mu_{X_i} - \pi_i\circ p_{X_i}^{-1}}_\TV
			+ \norm {\mu_{Y_i} - \pi_i\circ p_{Y_i}^{-1}}_\TV
		+ \sum_{j>n} \big(\norm {\mu_{X_j}}_\TV + \norm {\mu_{Y_j}}_\TV\big)\\
		& \leq \textstyle\sum_{i=1}^n \upD(\pi_i;\mu_{X_i},\mu_{Y_i})
		+ \sum_{j>n} \big(\mu_{X_j}(X_j) + \mu_{Y_j}(Y_j)\big)\\
		& \leq \textstyle\sum_{i\geq 1}^n \upd_\GHP(\bfX_i,\bfY_i) + \varepsilon
		+ \sum_{j>n} \big(\mu_{X_j}(X_j) + \mu_{Y_j}(Y_j)\big).
	\end{align*}
	In light of Lemma~\ref{lem:concatenation-compactness}, there exists $n$ such that
	$\sum_{i>n} \mu_{X_{\smash i}}(X_i) + \mu_{Y_{\smash i}}(Y_i) < \varepsilon$.
	As a result, $\upd_\GHP (\bfX,\bfY) \leq \sum_{i\geq 1} \upd_\GHP(\bfX_i,\bfY_i) + 2 \varepsilon$
	which holds for all positive $\varepsilon$.
\end{proof}

\subsubsection{Extension to locally compact \texorpdfstring{$\bbR$}{R}-trees}\label{sec:ghp-extension}
Let $\bfX = (X,\upd_X,\rho_X,\mu_X)$ be a locally compact pointed weighted metric space
such that $\mu_X$ is a boundedly finite measure.
For all $r>0$, let $\bfX\vert_r := \big( X\vert_r , \upd_X, \rho_X, \mu_X\vert_r \big)$
where $X\vert_r := \{ x\in X : \abs x \leq r\}$ is the closed ball with radius $r$ centred at $\rho_X$
and $\mu_X\vert_r := \ind_{X\vert_{\smash r}} \mu_X$ is the restriction of $\mu_X$ to $X\vert_r$.
Observe that if $r\leq R$, clearly $(\bfX\vert_R)\vert_r = (\bfX\vert_r)\vert_R = \bfX\vert_r$.
We also define $\partial_r X := \{x\in X : \abs x = r\}$.

For any two locally compact pointed weighted metric spaces $\bfX$ and $\bfY$,
we define the extended Gromov-Hausdorff-Prokhorov distance between them~as:
\[
	\upD_\GHP(\bfX,\bfY) :=
	\int_0^\infty \E^{-r} \Big[ 1\wedge \upd_\GHP \big(\bfX\vert_r, \bfY\vert_r \big) \Big] \D r.
\]
This definition closely resembles that of the GHP distance on locally compact metric spaces
defined and studied in~\cite{abraham2013ghp}.

\begin{remark}\label{nb:extended-ghp-truncation}
	Let  $\bfX$ and $\bfY$ be two weighted locally compact pointed metric spaces.
	For all $R\geq 0$,
	\begin{align*}
		& \textstyle
		\abs [1] {\upD_\GHP(\bfX,\bfY) - \upD_\GHP(\bfX\vert_R,\bfY\vert_R)}
		\leq \int_R^\infty \E^{-r} \,
			\underbrace {\abs [1] {
			1\wedge \upd_\GHP (\bfX\vert_r, \bfY\vert_r )
				- 1\wedge \upd_\GHP (\bfX\vert_R, \bfY\vert_R )}}_{\leq\mathrlap 1} \, \D r
		\leq \E^{-R}.
	\end{align*}
\end{remark}

Let $\bbT$ be the set of GHP-isometry classes of locally compact rooted $\bbR$-trees
endowed with a boundedly finite Borel measure
and $\bbT_c$, be that of compact weighted and rooted $\bbR$-trees
(\ie $\bbT_c = \bbK\cap\bbT$).

\begin{proposition}\label{prop:extended-ghp-properties}
	\begin{enumerate}
		\renewcommand{\labelenumi}{$(\roman{enumi})$}
		\item $\upD_\GHP$ is a metric on $\bbT$,
		\item If $\bfT_n$, $n\geq 1$ and $\bfT$ belong to $\bbT$,
		then $\upD_\GHP(\bfT_n,\bfT) \to 0$ \tiff
		$\upd_\GHP(\bfT_n\vert_r,\bfT\vert_r) \to 0$
		for all $r\geq 0$ with $\mu_T(\partial_r T) = 0$,
		\item $(\bbT, \upD_\GHP)$ is a Polish metric space,
		\item $\upd_\GHP$ and $\upD_\GHP$ induce the same topology on $\bbT_c$.
	\end{enumerate}
\end{proposition}

\begin{proof}
	\noindent$(i)$\quad
	Since $\upd_\GHP$ is a metric, $\upD_\GHP$ is symmetric
	and clearly satisfies the triangular inequality.
	Moreover, if $\bfT$ and $\bfT'$ are two elements of $\bbT$
	such that $\upD_\GHP(\bfT,\bfT') = 0$, 
	then for almost every $r\geq 0$, $\bfT\vert_r = \bfT'\vert_r$.
	In this case, $\bfT$ and $\bfT'$ are GHP-isometric
	(see \cite[Proposition~5.3]{abraham2013ghp} for a similar proof).
	
	\smallskip
	
	\noindent$(ii)$\quad
	Suppose $\upd_\GHP(\bfT_n\vert_r,\bfT\vert_r) \to 0$
	for all $r\geq 0$ with $\mu_T(\partial_r T) = 0$.
	Since $\mu_T$ is a locally finite measure,
	the set $\{r>0 : \mu_T(\partial_r T)>0\}$ is at most countable.
	As a result, the sequence $\big( r \mapsto 1\wedge\upd_\GHP(\bfT_n\vert_r,\bfT\vert_r) \big)_{n\geq 1}$
	converges to $r\mapsto 0$ almost everywhere in $[0,\infty)$.
	Lebesgue's dominated convergence theorem then ensures
	that $\upD_\GHP(\bfT_n,\bfT) \to 0$.
	
	\smallskip
	
	Assume $\upD_\GHP (\bfT_n,\bfT) \to 0$
	and let $r>0$ be such that $\mu_T(\partial_r T) = 0$.
	For every subsequence $(n_k)_k$, there exists a sub-subsequence $(k_\ell)_\ell$
	such that $1\wedge \upd_\GHP (\bfT_{n_{\smash {k_\ell}}}\vert_t, \bfT\vert_t)\to 0$
	for almost every $t\geq 0$ as $\ell\to\infty$.
	In particular, there exists $R>r$ such that $\upd_\GHP (\bfT_{n_{\smash {k_\ell}}}\vert_R, \bfT\vert_R) \to 0$.
	
	Recall that $\upd_\GHP$ is topologically equivalent to 	the metric on $\bbK$ studied in \cite{abraham2013ghp}.
	Therefore, in light of the proof of \cite[Proposition~2.10]{abraham2013ghp},
	if $\tau_n$, $n\geq 1$ and $\tau$ are compact $\bbR$-trees 	such that $\upd_\GHP (\tau_n, \tau) \to 0$,
	then for all $r>0$ such that $\mu_\tau(\partial_r\tau) = 0$, $\upd_\GHP (\tau_n\vert_r,\tau\vert_r) \to 0$.
	
	As a result, $\upd_\GHP (\bfT_{n_{\smash {k_\ell}}}\vert_r,\bfT\vert_r) \to 0$.
	From every subsequence $(n_k)_k$ we can thus extract a sub-subsequence $(k_\ell)_\ell$
	such that $\upd_\GHP (\bfT_{n_{\smash {k_\ell}}}\vert_r,\bfT\vert_r) \to 0$,
	which is equivalent to saying that $\upd_\GHP (\bfT_n\vert_r,\bfT\vert_r) \to 0$ as $n\to\infty$.
	
	\smallskip
	
	\noindent$(iii)$\quad
	Since a criterion similar to $(ii)$ holds for the metric studied in \cite{abraham2013ghp},
	this metric is topologically equivalent to $\upD_\GHP$.
	As a result and thanks to Theorem~2.9 and Corollary~3.2 in~\cite{abraham2013ghp},
	it follows that $(\bbT,\upD_\GHP)$ is completely metrisable and separable,
	\ie it is Polish.
	
	\smallskip
	
	\noindent$(iv)$\quad
	See Proposition~2.10 in~\cite{abraham2013ghp}.
\end{proof}

\paragraph{Continuous grafting}\label{def:continuous-grafting-application}
Let $\{(u_i,\tau_i) : i\in\calI\}$ be a family of elements of $\bbR_+\times\bbT_c$
such that $\calI$ is at most countable.
We define the $\bbR$-tree $\bfG\big(\{(u_i,\tau_i) : i\in\calI\}\big)$ as
\[
	\bfG\Big(
		\big\{(u_i,\tau_i) : i\in\calI\big\}\Big)
	:= \Big( \bbR_+\sqcup
		{\textstyle\bigsqcup}_{i\in\calI} \tau_i,
		\, \upd, \, 0, \, \mu \Big)
\]
where the metric $\upd$ is defined by:
\begin{itemize}
	\item $\upd[u,v] = \abs {u-v}$ for all $u$ and $v$ in $\bbR_+$,
	\item $\upd[x,y] = \upd_{\tau_{\smash i}}(x,y)$
	for all $i\in\calI$, $x$ and $y$ in $\tau_i$,
	\item $\upd[x,v] = \upd_{\tau_{\smash i}}(x,\rho_{\tau_{\smash i}}) + \abs {u_i-v}$
	for all $i\in\calI$, $x\in\tau_i$ and $v$ in $\bbR_+$,
	\item $\upd[x,y] = \upd_{\tau_{\smash i}}(x,\rho_{\tau_{\smash i}})
	+ \upd_{\tau_{\smash j}}(y,\rho_{\tau_{\smash j}}) + \abs {u_i-u_j}$
	for all $i\neq j\in\calI$, $x\in\tau_i$ and $y\in\tau_j$,
\end{itemize}
and $\mu$ is the measure defined for all Borel set $A$ by
$\mu(A) := \sum_{i\in\calI} \mu_{\tau_i}(A\cap\tau_i)$.
The application $\bfG$ grafts the trees $\tau_i$ at height $u_i$ for each $i\in\calI$
on $\bbR_+$ which can be thought of as an infinite (continuous) branch.
It is quite obvious that the weighted pointed metric space
$\bfG\big(\{(u_i,\tau_i) : i\in\calI\}\big)$ is an $\bbR$-tree.

\begin{lemma}\label{lem:continuous-grafting-locally-compact}
	Let $(u_i)_{i\geq 1}$ be a sequence of non-negative real numbers
	and $(\tau_i,\upd_i,\rho_i,\mu_i)_{i\geq 1}$ be a sequence of compact weighted $\bbR$-trees.
	The weighted $\bbR$-tree $\bfT := \bfG\big(\{(u_i,\tau_i) : i\geq 1\}\big)$
	is an element of $\bbT$ \tiff for all $K\geq 0$ and $\varepsilon>0$ the set
	$\{i\geq 1 : u_i\leq K \text { and } \abs {\tau_i}\geq \varepsilon\}$ is finite
	and $\sum_{i\geq 1} \ind_{u_{\smash i}\leq K} \mu_{\tau_{\smash i}}(\tau_i) < \infty$.
\end{lemma}

\begin{proof}
	For all $x$ in $T$ and positive $r$, denote by $\upB_T(x,r) := \{y\in T : \upd_T(x,y)<r\}$
	the open ball of $T$ centred at $x$ with radius $r$
	and similarly for all $i\geq 1$ and $x\in\tau_i$, note $\upB_i(x,r) := \{y\in\tau_i : \upd_i(x,y)<r\}$.
	
	\smallskip
	
	\noindent$\Leftarrow$\quad	
	Assume that for all $K\geq 0$,
	$\sum_{i\geq 1} \ind_{u_{\smash i}\leq K} \mu_{\tau_{\smash i}}(\tau_i) < \infty$
	and for all positive $\varepsilon$, that the set
	$\{i\geq 1 : u_i\leq K, \abs {\tau_i}\geq\varepsilon\}$ is finite.
	Observe that for all non-negative $K$,
	$\mu_T (T\vert_K) \leq \sum_{i\geq 1} \mu_{\tau_{\smash i}}(\tau_i) \ind_{u_{\smash i}\leq K}$.
	Therefore, the measure $\mu_T$ is boundedly finite
	and we only need to prove that $T$ is locally compact.

	Fix $K\geq 0$ and let $\varepsilon$ be positive.
	For all $i\geq 1$, because $\tau_i$ is compact, there exists a finite subset $A_i$ of $\tau_i$
	such that $\tau_i \subset \bigcup_{x\in A_{\smash i}} B_i(x,\varepsilon)$.
	To build an $\varepsilon$-cover of $T\vert_K$,
	first observe that if $i$ is such that $u_i\leq K$ and $\abs {\tau_i} < \varepsilon/2$,
	then $\tau_i$ is contained in some open ball with radius $\varepsilon$
	centred at some $n\varepsilon$ for $0\leq n\leq K/\varepsilon$.
	Moreover, by assumption, there are only finitely many indices $i$ with $u_i\leq K$
	and $\abs {\tau_i}\geq \varepsilon/2$.
	Therefore, if we let $A := \{n\varepsilon ; 0\leq n\leq K/\varepsilon\}
	\cup \{x\in A_i ; i\geq 1, u_i\leq K, \abs {\tau_i}\geq \varepsilon/2\}$,
	then $A$ is finite and $T\vert_K$ is contained in $\bigcup_{x\in A} B_T(x,\varepsilon)$.
	As a result, $T\vert_K$ has a finite $\varepsilon$-cover for all positive $\varepsilon$
	which means that it is compact.
	
	\smallskip
	
	\noindent$\Rightarrow$\quad
	Suppose the set $\{i\geq 1 : u_i\leq K, \abs {\tau_i}\geq\varepsilon\}$ is infinite
	for some $K\geq 0$ and positive $\varepsilon$.
	In particular, we can find an increasing sequence $(i_n)_n$
	with $u_{i_{\smash n}}\leq K$ and $\abs {\tau_{i_{\smash n}}} \geq \varepsilon$ for all $n$.
	For each $n\geq 1$, let $x_n$ be in $\tau_{i_{\smash n}}$ and such that
	$\varepsilon/2 < \upd_{i_{\smash n}} (\rho_{i_{\smash n}},x_n) \leq \varepsilon$.
	If $n\neq m$, the definition of the metric on $T$ gives $\upd_T(x_n,x_m) > \varepsilon$.
	Therefore, $(x_n)_n$ has no Cauchy subsequence which implies that $T\vert_{K+\varepsilon}$
	isn't compact and that $\bfT\notin\bbT$.
	
	Assume that $\{i\geq 1 : u_i\leq K, \abs {\tau_i}\geq\varepsilon\}$ is finite
	for all $K\geq 0$ and $\varepsilon>0$,
	and that $\sum_{i\geq 1} \ind_{u_{\smash i}\leq K_{\smash 0}} \, \mu_{\tau_{\smash i}}(\tau_i)$
	is infinite for some finite $K_0$.
	By assumption, $\{ \abs{\tau_i} : u_i\leq K_0\}$ is bounded by a finite constant $R$.
	Therefore, $\mu_T(T\vert_{K_{\smash 0}+R})
	\geq \sum_{i\geq 1} \ind_{u_{\smash i}\leq K_{\smash 0}} \mu_{\tau_{\smash i}}(\tau_i) = \infty$.
	Consequently, $\mu_T$ isn't boundedly finite and $\bfT\notin\bbT$.
\end{proof}

\smallskip

\begin{remark}
	In the following, when we consider discrete trees,
	we will see them as $\bbR$-trees by replacing their edges by segments of length $1$.
\end{remark}

\subsection{Fragmentation trees}

In this section, we will present a few results
on certain classes of $\bbT_c$- and $\bbT$-valued random variables:
self-similar fragmentation trees (introduced in~\cite{haas2004genealogy})
and self-similar fragmentation trees with immigration (see~\cite{haas2007fragmentationinitialmass}).

\subsubsection{Self-similar fragmentation trees}\label{sec:ssf-trees}
Let $\Sdec := \big\{ \bfs = (s_n)_{n\geq 1} \in \ell_1 \,:\, s_1 \geq s_2 \geq \dots \geq 0\big\}$
and endow it with the $\ell_1$ norm, \ie for all $\bfs$ and $\bfr$ in $\Sdec$,
say that the distance between $\bfs$ and $\bfr$ is $\norm {\bfs-\bfr} = \sum_{i\geq 1} \abs {s_i-r_i}$.
Moreover, set $\mathbf 0 := (0,0,\dots)$ and $\mathbf 1 := (1,0,0,\dots)$.
We will also note $\Sdec_{\leq 1} := \big\{ \bfs\in\Sdec \,:\, \norm\bfs \leq 1\big\}$.

A \emph{self-similar fragmentation process} is an $\Sdec_{\leq 1}$-valued Markovian process
$(\bfX(t);t\geq 0)$ which is continuous in probability, and satisfies $\bfX(0) = \mathbf 1$
as well as the following so-called \emph{fragmentation property}.
There exists $\alpha\in\bbR$ such that for all $t_0\geq 0$,
conditionally to $\bfX(t_0) = \bfs$, $\big(\bfX(t_0+t), t\geq 0\big)$ has the same distribution as
\[
	\Big( \big( s_i \, \bfX^{(i)}(s_i^\alpha t) \big)^\downarrow
	\, ; \, t \geq 0 \Big)
\]
where $(\bfX^{(i)})_i$ are \iid copies of $\bfX$.
The constant $\alpha$ is called the \emph{self-similarity index} of the process~$\bfX$.

\smallskip

These processes can be seen as the evolution of the fragmentation
of an object of mass $1$ into smaller objects which will each, in turn,
split themselves apart independently from one another, at a rate proportional to their mass
to the power $\alpha$.

It was shown in~\cite{berestycki2002rankedfrag,bertoin2002self} that the distribution
of a self-similar fragmentation process
is characterised by a $3$-tuple $(\alpha,c,\nu)$
where $\alpha$ is the aforementioned self-similarity index,
$c\geq 0$ is a so-called erosion coefficient which accounts for
a continuous decay in the mass of each particle
and $\nu$ is a \emph{dislocation measure} on $\Sdec_{\smash {\leq 1}}$,
\ie a $\sigma$-finite measure such that $\int (1-s_1) \, \nu(\D\bfs) < \infty$
and $\nu(\{\mathbf 1\}) = 0$.
At any given time, each particle with mass say $x$ will, independently from the other particles,
split into smaller fragments of respective masses $xs_1, xs_2,\dots$
at rate $x^\alpha \nu(\D\bfs)$.

\smallskip

We will be interested in fragmentation processes with negative self-similarity index $-\gamma<0$
with no erosion, \ie with $c=0$.
Furthermore, we will require the dislocation measure $\nu$ to be non-trivial, \ie $\nu(\Sdec_{\leq1})>0$,
and \emph{conservative}, that is to satisfy $\nu (\norm\bfs<1) = 0$.
Therefore, the fragmentation processes we will consider will be characterised by a \emph{fragmentation pair} $(\gamma,\nu)$
and we will refer to them as $(\gamma,\nu)$-fragmentation processes.

Under these assumptions, each particle will split into smaller ones which will in turn break down faster,
thus speeding up the global fragmentation rate.
Let $\bfX$ be a $(\gamma,\nu)$-fragmentation process and
set $\tau_{\mathbf 0} := \inf\{ t\geq 0 : \bfX(t) = \mathbf 0\}$
the first time at which all the mass has been turned to dust.
It was shown in~\cite[Proposition~2]{bertoin2003asymptotic} that $\tau_{\mathbf 0}$ is \as finite
and in~\cite[Section~5.3]{haas2003lossofmass} that it has exponential moments,
\ie that there exists $a>0$ such that $\esp [1] {\exp (a \tau_{\mathbf 0})} < \infty$.

Furthermore, a $\bbT_c$-valued random variable that encodes the genealogy
of the fragmentation of the initial object was defined in~\cite{haas2004genealogy}.
This random $\bbR$-tree $(\calT,\upd,\rho,\mu)$ is such that $\mu(\calT)=1$
and if for all $t\geq 0$, $\{\calT_i(t) : i\geq 1\}$ is the (possibly empty)
set of the closures of the connected components of $\calT\setminus(\calT\vert_t)$,
then
\[
	\Big( \big(\mu[ \calT_i(t)] \,;\, i\geq 1\big)^\downarrow \; ; \; t\geq 0\Big)
\]
is a $(\gamma,\nu)$-fragmentation process.
We will note $\scrT_{\gamma,\nu}$ the distribution of $(\calT,\upd,\rho,\mu)$.

\begin{remark}
\begin{itemize}
	\item More general self-similar fragmentation trees,
	where both the assumptions ``$c=0$'' and ``$\nu$ is conservative'' are dropped,
	were defined and studied in~\cite{stephenson2013generalfragtrees}.
	
	\item Let $\calT$ be a $(\gamma,\nu)$-self-similar fragmentation tree and $m>0$.
	The tree $(m^\gamma \calT, m\,\mu_\calT)$ encodes the genealogy
	of a $(\gamma,\nu)$-self-similar fragmentation process started from a single object with mass~$m$.
\end{itemize}
\end{remark}

\paragraph{Classical examples}
It was observed in~\cite{bertoin2002self} that the Brownian tree,
which was introduced in~\cite{aldous1991crt1},
may be described as a self-similar fragmentation tree with parameters $(1/2, \nu_B)$
where $\nu_B$ is called the \emph{Brownian dislocation measure}
and is defined for all measurable $f:\Sdec_{\leq 1}\to\bbR_+$ by
\[
	\int f \, \D\nu_B = \int_{1/2}^1
		\bigg( \frac 2 {\pi \, x^3 \, (1-x)^3} \bigg)^{1/2} f(x,1-x,0,0,\dots) \, \D x.
\]

Another important example of fragmentation trees is the family of $\alpha$-stable trees
from~\cite{legall2002stabletrees}, where $\alpha$ belongs to $(1,2)$.
Indeed, a result from~\cite{miermont2003ssfragstable1} states that the $\alpha$-stable tree
is a $(1-1/\alpha, \nu_\alpha)$-self-similar fragmentation tree
with $\nu_\alpha$ defined as follows:
let $(\Sigma_t;t\geq 0)$ be a $1/\alpha$-stable subordinator with Laplace exponent
$\lambda\mapsto-\log\esp{\exp(-\lambda\Sigma_t)}=\lambda^{1/\alpha}$
and L\'evy measure $\Pi_{1/\alpha}(\D t) := [\alpha\,\Gamma(1-1/\alpha)]^{-1}
\, t^{-1-1/\alpha} \,\ind_{t>0} \, \D t$,
denote the decreasing rearrangement of its jumps on $[0,1]$ by $\Delta$
and for all measurable $f:\Sdec\to\bbR_+$, let
\[
	\int_{\Sdec} f \, \D\nu_\alpha
	= \frac {\Gamma(1-1/\alpha)} {k_\alpha}
		\esp [2] {\Sigma_1 \, f\big( \Delta / \Sigma_1 \big)}
\]
where $k_\alpha := \Gamma(2-\alpha) / [\alpha\,(\alpha-1)]$.
Observe that the random point measure $\sum_{i\geq 1} \delta_{\Delta_{\smash i}}$ on $(0,\infty)$
with atoms $(\Delta_i, i\geq 1)$ is a Poisson Point Process
with intensity measure $\Pi_{1/\alpha}$.

\paragraph{Scaling limits of Markov branching trees}
Self-similar fragmentation trees bear a close relationship with Markov branching trees.
Let $\iota:\calP_{<\infty}\to\Sdec_1$ be such that if $\lambda=(\lambda_1,\dots,\lambda_p)$ is in $\calP_n$,
then $\iota(\lambda) := (\lambda_1/n,\dots,\lambda_p/n,0,0,\dots)$.

\begin{theorem}[\cite{haas2012scaling}, Theorems~5 and~6]\label{thm:haas-scaling-mb}
	\begin{itemize}
		\item Let $(q_n)_{n\in\calN}$ be the sequence of first-split distributions
		of a Markov branching family $\MB^{\smash\calL,q}$
		and for all adequate $n\geq 1$, set $\bar q_n := q_n\circ\iota^{-1}$.
		Suppose there exists a fragmentation pair $(\gamma,\nu)$ and a slowly varying function $\ell$ such that,
		for the weak convergence of finite measures on~$\Sdec$,
		\[
			n^\gamma \ell(n) \, (1-s_1) \, \bar q_n (\D\bfs)
			\xrightarrow [n\to\infty] {}
			(1-s_1) \, \nu(\D\bfs).
		\]
		For all $n\in\calN$, let $T_n$ have distribution $\MB^{\smash\calL,q}_n$
		and set $\mu_n := \sum_{u\in\calL(T_{\smash n})} \delta_u$
		the counting measure on the leaves of $T_n$.
		
		\item Let $(q_{n-1})_{n\in\calN}$ be the sequence associated to a Markov branching family $\MB^q$.
		Assume that there exists a fragmentation pair $(\gamma,\nu)$ and a slowly varying function $\ell$
		with either $\gamma<1$ or $\gamma=1$ and $\ell(n)\to 0$
		such that $n^\gamma \ell(n) \, (1-s_1) \, \bar q_n (\D\bfs) \Rightarrow (1-s_1) \, \nu(\D\bfs)$.
		For each $n\in\calN$, let $T_n$ be a $\MB^q_n$ tree
		and endow it with its counting measure $\mu_n$.
	\end{itemize}
	Under either set of assumptions, with respect to the GHP topology on $\bbT_c$,
	\[
		\bigg( \frac 1 {n^{\gamma}\ell(n)} T_n, \frac 1 n \mu_{T_n} \bigg)
		\xrightarrow [n\to\infty] {} \scrT_{\gamma,\nu}
		\quad\text{in distribution.}
	\]
\end{theorem}

The following useful result on the heights of Markov branching
also holds.

\begin{lemma}\label{lem:mb-trees-height}
	Suppose that $(q_n)_{n\in\calN}$ satisfies the assumptions of Theorem~\ref{thm:haas-scaling-mb}
	with respect to a given fragmentation pair $(\gamma,\nu)$ and a slowly varying function $\ell$.
	Then for any $p>0$, there is a finite constant $h_p$ such~that
	\[
		\sup_{n\in\calN} \, \esp [4] {\bigg(\frac {\abs {T_n}} {n^\gamma \ell(n)}\bigg)^p} \leq h_p
		\quad\text{and}\quad
		\esp [1] {\abs\calT^p} \leq h_p
	\]
	where $\calT$ is a $(\gamma,\nu)$-fragmentation tree and,
	as in Theorem~\ref{thm:haas-scaling-mb}, $T_n$ has distribution either $\MB^q_n$ or $\MB^{\smash\calL,q}_n$.
\end{lemma}

\begin{proof}
	See \cite[Section~5.3]{haas2003lossofmass} for the continuous setting
	and \cite[Lemma~33]{haas2012scaling} plus \cite[Section~4.5]{haas2012scaling}
	for the discrete one.
\end{proof}

\paragraph{Concatenation of fragmentation trees}
Fix a fragmentation pair $(\gamma,\nu)$
and let $(\calT_i)_{i\geq 1}$ be a sequence of \iid $(\gamma,\nu)$-fragmentation trees.
For all $i\geq 1$, note $\mu_i$ the measure of $\calT_i$.
Fix $\bfs$ in $\Sdec$ and set $(\calT_{\langle\bfs\rangle}, \mu_{\langle\bfs\rangle})
:= \big\langle (s_i^\gamma \calT_i, s_i \mu_i) \,;\, i\geq 1\big\rangle$.

\begin{lemma}\label{lem:ssf-concatenation-compact}
	With these notations, $(\calT_{\langle\bfs\rangle}, \mu_{\langle\bfs\rangle})$ \as belongs to $\bbT_c$.
\end{lemma}

\begin{proof}
	Clearly $\calT_{\langle\bfs\rangle}$ is an $\bbR$-tree and its total mass is
	$\mu_{\langle\bfs\rangle}(\calT_{\langle\bfs\rangle}) = \sum_{i\geq 1} s_i \mu_i (\calT_i) = \norm\bfs$ which is finite.
	It only remains to show that it is compact
	or, in light of Lemma~\ref{lem:concatenation-compactness},
	that $s_i^\gamma \, \abs {\calT_i}$ \as converges to $0$ as $i$ grows to infinity.
	Since $\bfs$ is summable, for any positive $\varepsilon$,
	\[
		\textstyle
		\sum_{i\geq 1} \prob [1] {s_i^\gamma \abs {\calT_i} > \varepsilon}
		\leq \sum_{i\geq 1} \dfrac {s_i} {\varepsilon^{1/\gamma}} \esp [1] {\abs {\calT_1}^{1/\gamma}}
		\leq \dfrac 1 {\varepsilon^{1/\gamma}} \esp [1] {\abs {\calT_1}^{1/\gamma}} \, \norm\bfs
		< \infty
	\]
	where we have used Markov's inequality and the fact
	that $\abs {\calT_i}^{1/\gamma}\in L^1$ (see~Lemma~\ref{lem:mb-trees-height}).
	Borel-Cantelli's lemma then allows us to deduce that $s_i^\gamma \abs {\calT_i} \to 0$ \as as $i\to\infty$.
\end{proof}

\begin{lemma}\label{lem:ssf-concatenation-ghp-exp-continuity}
	For all fixed $\bfs$ in $\Sdec$,
	$\esp [1] {\upd_\GHP(\calT_{\langle\bfs\rangle},\calT_{\langle\bfr\rangle})}$
	converges to $0$ as $\bfr\to\bfs$.
\end{lemma}

\begin{proof}
	For all $n\geq 0$, in light of Lemmas~\ref{lem:ghp-rescaling-bounds} and~\ref{lem:continuous-concatenation-ghp-bounds},
	\[
		\upd_\GHP\big(\calT_{\langle\bfs\rangle},\calT_{\langle\bfr\rangle}\big)
		\leq \sum_{i=1}^n \Big[ \big( \abs {s_i^\gamma-r_i^\gamma} \, \abs {\calT_i} \big) \vee \abs {s_i-r_i} \Big]
			+ \sum_{i>n} \big( s_i + r_i \big)
			+ \sup_{i>n} \Big( s_i^\gamma \abs {\calT_i} \Big)
			+ \sup_{i>n} \Big( r_i^\gamma \abs {\calT_i} \Big).
	\]
	If $\gamma\leq 1$, $t\mapsto t^\gamma$ is concave, hence Jensen's inequality gives
	\[
		\esp [3] {\sup_{i>n} \Big( s_i^\gamma \abs {\calT_i} \Big)}
		= \esp [3] {\bigg(\sup_{i>n} \: s_i \abs {\calT_i}^{1/\gamma} \bigg)^\gamma}
		\leq \bigg(\esp [2] {\sup_{i>n} \: s_i \abs {\calT_i}^{1/\gamma}}\bigg)^\gamma
		\leq \esp [1] {\abs {\calT_1}^{1/\gamma}}^\gamma \Big({\textstyle\sum}_{i>n} s_i\Big)^\gamma,
	\]
	otherwise, if $\gamma>1$, since $(s_i)$ is non-increasing, for all $i>n$,
	$s_i^\gamma \leq s_{n+1}^{\gamma-1} s_i$ which implies
	\[
		\esp [3] {\sup_{i>n} \Big( s_i^\gamma \abs {\calT_i} \Big)}
		\leq s_{n+1}^{\gamma-1} \, \esp [3] {\sup_{i>n} \Big( s_i \abs {\calT_i} \Big)}
		\leq \esp [1] {\abs {\calT_1}} \, s_{n+1}^{\gamma-1} \, {\textstyle\sum\limits_{i>\mathrlap n}} \, s_i
		\leq \esp [1] {\abs {\calT_1}} \, \Big( {\textstyle\sum_{i>n}} \, s_i \Big)^\gamma.
	\]
	Consequently, there is a constant $C\geq 0$ independent of $\gamma$
	such that for all integer $n$ and $\bfs$ in $\Sdec$,
	$\esp [1] {\sup_{i>n} s_i^\gamma \abs {\calT_i}} \leq C \big[\sum_{i>n} s_i\big]^\gamma$.
	Hence, for all $\bfs$ and $\bfr$ in $\Sdec$ and any $n\geq 1$
	\[
		\esp [2] {\upd_\GHP(\calT_{\langle\bfs\rangle},\calT_{\langle\bfr\rangle})}
		\leq \norm {\bfs-\bfr}
			+ \esp [1] {\abs {\calT_1}} {\textstyle\sum\limits_{i=1}^n} \abs {s_i^\gamma-r_i^\gamma}
			+ {\textstyle\sum\limits_{i>\mathrlap n}} \big( s_i + r_i \big)
			+ C \bigg[ \Big({\textstyle\sum\limits_{i>\mathrlap n}} \, s_i\Big)^\gamma
			+ \Big({\textstyle\sum\limits_{i>\mathrlap n}} \, r_i\Big)^\gamma \bigg].
	\]
	As a result,
	\[
		\limsup_{\bfr\to\bfs} \: \esp [2] {\upd_\GHP(\calT_{\langle\bfs\rangle},\calT_{\langle\bfr\rangle})}
		\leq \inf_{n\geq 1} \:
			2 \, {\textstyle\sum_{i>n}} s_i
			+ 2 \, C \Big({\textstyle\sum_{i>n}} \, s_i\Big)^\gamma
		= 0.
	\]
\end{proof}

\subsubsection{Fragmentation trees with immigration}\label{sec:frag-trees-immigration}
We say that a non-negative Borel measure $I$ on $\Sdec$ is an \emph{immigration measure}
if it satisfies $\int_{\Sdec} (1\wedge\norm\bfs) \, I(\D\bfs) < \infty$.
We will say that two such measures $I$ and $J$ are equivalent
if $(1\wedge\norm\bfs) \, I(\D\bfs) = (1\wedge\norm\bfs) \, J(\D\bfs)$,
\ie if $\abs {I-J}$ is supported by $\{\mathbf 0\}$.

Fix an immigration measure $I$ such that $I(\Sdec) > 0$
and let $(\gamma,\nu)$ be a fragmentation pair.
Let $\Sigma = \sum_{n\geq 1} \delta_{(u_n,\bfs_n)}$ be a Poisson point process
on $\bbR_+\times\Sdec$ with intensity $\D u \otimes I(\D\bfs)$
independent of a family $( \bfX^{(n,k)}, \, n\geq 1, k\geq 1 )$
of \iid $(\gamma,\nu)$-fragmentation processes.
Define the $\Sdec$-valued process $\bfX$ as follows:
\[
	\bfX = \big(\bfX(t), \, t\geq 0\big)
	:= \Bigg( \bigg( s_{n,k} \bfX^{(n,k)}\big[ s_{n,k}^{-\gamma} (t-u_n) \big]
		\:;\: n\geq 1 : u_n\leq t, \, k\geq 1\bigg)^\downarrow \; ;\; t\geq 0\Bigg).
\]
We call $\bfX$ a fragmentation process with immigration with parameters $(\gamma,\nu, I)$.
It describes the evolution of the masses of a cluster of independently fragmenting objects,
where new objects of sizes $\bfs_n$ appear, or immigrate, at time $u_n$.
These processes were introduced in~\cite{haas2005equilibriumfragmentationimmigration}.

\smallskip

Similarly to pure fragmentation processes, the genealogy of these immigrations
and fragmentations can be encoded as an infinite weighted $\bbR$-tree (see~\cite{haas2007fragmentationinitialmass}),
say $(\calT^{(I)},\upd,\rho,\mu)$,
such that if for all $t\geq 0$, we note $\{\calT_i(t) : i\geq 1\}$
the set of the closures of the bounded connected components of $\calT^{(I)}\setminus(\calT^{(I)}\vert_t)$,
then
\[
	\Big( \big(\mu[\calT_i(t)] \,;\, i\geq 1\big)^\downarrow \; ; \; t\geq 0\Big)
\]
is a $(\gamma,\nu,I)$-fragmentation process with immigration.
Let $\scrT_{\gamma,\nu}^I$ be the distribution of $(\calT^{(I)},\upd,\rho,\mu)$.

\paragraph{Point process construction}
The construction of $(\gamma,\nu)$-fragmentation trees with immigration $I$
described in~\cite{haas2007fragmentationinitialmass}
can be expressed using Poisson point processes, concatenated $(\gamma,\nu)$-fragmentation trees
and the continuous grafting application $\bfG$ from the end of Section~\ref{sec:compact-ghp-topology}.
Let $\Sigma = \sum_{i\geq 1} \delta_{(u_i,\bfs_i)}$ be a Poisson point process
on $\bbR_+\times\Sdec$ with intensity $\D u \otimes I(\D\bfs)$
and $(\calT_{i,j},\mu_{i,j})_{i,j\geq 1}$ be \iid $(\gamma,\nu)$-fragmentation trees independent of $\Sigma$.
For all $i\geq 1$, set
\[
	\calT_i := \big\langle(s_{i,j}^\gamma\calT_{i,j}, s_{i,j} \mu_{i,j}) ; j\geq 1\big\rangle,
\]
the concatenation of $(\calT_{i,j} ; j\geq 1)$ with respective masses $s_{i,j}$.
Define $\calT^{(I)}$ as the tree obtained by grafting $\calT_i$ at height $u_i$
on an infinite branch for each $i\geq 1$, \ie\label{def:frag-trees-with-immigration}
\[
	\calT^{(I)} := \bfG\Big( \big\{ (u_i, \calT_i) \: : \: i\geq 1\big\} \Big).
\]
The random tree $\calT^{(I)}$ has distribution $\scrT_{\gamma,\nu}^I$.

\smallskip

Observe that for all $K\geq 0$,
we can write the total mass grafted on the infinite branch at height less than $K$
as an integral against the point-process $\Sigma$:
\[
	\textstyle
	\sum_{i\geq 1} \ind_{u_i\leq K} \mu_{\calT_i}(\calT_i)
	= \sum_{i\geq 1} \ind_{u_i\leq K} \norm{\bfs_i}
	= \int \, \ind_{u\leq K} \, \norm\bfs \, \Sigma(\D u,\D\bfs).
\]
Since $\int \, 1 \wedge\big(\ind_{u\leq K} \, \norm\bfs\big) \, \D u \, I(\D\bfs)
= K \int (1\wedge\norm\bfs) I(\D\bfs) < \infty$,
we may use Campbell's theorem (see~\cite[Section~3.2]{kingman1992poissonprocesses})
and claim that $\int \, \ind_{u\leq K} \, \norm\bfs \, \Sigma(\D u,\D\bfs) < \infty$ \as.
The second condition of Lemma~\ref{lem:continuous-grafting-locally-compact}
is thus met.
Moreover, for all $i\geq 1$,
\[
	\esp [2] {\abs {\calT_i}^{1/\gamma} \big\vert \Sigma}
	= \esp [2] {\sup\nolimits_{j\geq 1} s_{i,j} \abs {\calT_{i,j}}^{1/\gamma} \big\vert \Sigma}
	\leq {\textstyle\sum_{j\geq 1}} s_{i,j} \esp [2] {\abs {\calT_{1,1}}^{1/\gamma}}
	\leq \esp [2] {\abs {\calT_{1,1}}^{1/\gamma}} \, \norm{\bfs_i}
\]
where we have used the fact that $(\calT_{i,j})_{i,j}$ is an \iid family independent of $\Sigma$.
Markov's inequality therefore implies that
\[
	\sum_{i\geq 1} \ind_{u_i\leq K} \prob [1] {\abs {\calT_i} \geq \varepsilon \vert \Sigma}
	\leq \sum_{i\geq 1} \ind_{u_i\leq K} \varepsilon^{-1/\gamma} \, \esp [1] {\abs {\calT_i}^{1/\gamma} \vert \Sigma}
	\leq \frac {\esp [1] {\abs {\calT_{1,1}}^{1/\gamma}}} {\varepsilon^{1/\gamma}}
		\, \sum_{i\geq 1} \ind_{u_i\leq K} \norm {\bfs_i}
\]
which is, according to Campbell's formula, \as finite.
Consequently, using Borel-Cantelli's lemma, we deduce that conditionally on $\Sigma$, with probability one,
there are finitely many indices $i\geq 1$ such that $u_i\leq K$ and $\calT_i$ is higher than $\varepsilon$.
It follows from Lemma~\ref{lem:continuous-grafting-locally-compact}
that $\calT^{(I)}$ is \as $\bbT$-valued.

\begin{remark}\label{nb:ssfi-rescaling}
	Let $I$ be an immigration measure and suppose that there exists some positive $\gamma$
	such that for any measurable $F:\Sdec\to\bbR_+$ and $c>0$,
	$c \, \int F(\bfs) \, I(\D\bfs) = \int F(c^{1/\gamma}\bfs) \, I(\D\bfs)$.
	Then, for any dislocation measure $\nu$,
	a $(\gamma,\nu,I)$-fragmentation tree with immigration $(\calT,\mu)$
	satisfies the following self-similarity property:
	for any positive $m$, $(m^\gamma\calT, m\,\mu)$ has the same distribution as $(\calT,\mu)$.
	Furthermore, for any positive $c$,
	$(\calT, c\mu)$ is a $(\gamma,c^\gamma\nu,c^\gamma I)$-fragmentation tree with immigration
	and so is $(c^{-\gamma}\calT, \mu)$.
\end{remark}

\paragraph{Relationship to compact fragmentation trees}
Let $(\gamma,\nu)$ be a fragmentation pair and $I$ an immigration measure with $I(\Sdec)>0$.
Theorem~17 in \cite{haas2007fragmentationinitialmass} states that under suitable conditions,
if $(\calT,\mu_\calT)$ denotes a $(\gamma,\nu)$-self-similar fragmentation tree,
then $(m^\gamma\calT, m\mu_\calT)$ converges to $\scrT_{\gamma,\nu}^I$ in distribution as $m\to\infty$
with respect to the extended GHP topology.

\smallskip

For instance, Theorem~11~$(iii)$ in~\cite{aldous1991crt1},
states that if $(\calT,\mu_\calT)$ is a standard Brownian tree
then when $m\to\infty$, $(m^{1/2} \calT, m\, \mu_\calT)$ converges in distribution
to the ``self-similar CRT''.
This result was reformulated in terms of fragmentation trees in~\cite[Section~1.2]{haas2007fragmentationinitialmass}:
$(m^{1/2} \calT, m\, \mu_\calT)$ converges in distribution as $m\to\infty$
to a $(1/2, \nu_B, I_B)$-fragmentation tree with immigration,
where $\nu_B$ is the Brownian dislocation measure (see Section~\ref{sec:ssf-trees})
and the \emph{Brownian immigration measure} $I_B$ is defined
for all measurable $f:\Sdec\to\bbR_+$ by
\[
	\int F \, \D I_B :=
	\bigg(\frac 2 \pi\bigg)^{1/2} \int_{[0,\infty)} \frac {f(x,0,0,\dots)} {x^{3/2}} \, \D x.
\]
We will call a $(1/2, \nu_B, I_B)$-fragmentation tree with immigration
a \emph{immigration Brownian tree}.

Set $\alpha\in(1,2)$ and recall the notations used to define $\nu_\alpha$ in Section~\ref{sec:ssf-trees},
in particular, that $\Delta$ denotes the decreasing rearrangement of the jumps on $[0,1]$
of an $1/\alpha$-stable subordinator with Laplace exponent
$\lambda\mapsto-\log\esp{\exp(-\lambda\Sigma_t)}=\lambda^{1/\alpha}$
and that $k_\alpha = \Gamma(2-\alpha) / [\alpha\,(\alpha-1)]$.
Let $I^{(\alpha)}$ be the immigration measure defined for all measurable
$F:\Sdec\to\bbR_+$ by
\[
	\int_{\Sdec} F \, \D I^{(\alpha)}
	= \frac 1 {k_\alpha}
		\int_0^\infty \frac {\esp [1] {F(t^\alpha \, \Delta)}} {t^\alpha} \D t.
\]
In \cite[Section~5.1]{haas2007fragmentationinitialmass},
it was observed that if $(\calT,\mu_\calT)$ is an $\alpha$-stable tree,
then $(m^{1-1/\alpha} \calT, m\, \mu_\calT)$ converges in distribution
to a $(1-1/\alpha, \nu_\alpha, I^{(\alpha)})$-fragmentation tree with immigration
as $m\to\infty$.
These trees coincide with the $\alpha$-stable immigration L\'evy trees
introduced in~\cite[Section~1.2]{duquesne2009immigrationlevytrees}.

\subsection{Convergence of point processes}\label{sec:point-processes}

With the notations used in Section~\ref{sec:frag-trees-immigration},
let $\Pi := \sum_{i\geq 1} \smash{\delta_{(u_i,\bfs_i,\calT_i)}}$.
It is a Poisson point process on $\bbR_+\times\Sdec\times\bbT_c$
with intensity $\D u \otimes \scrI(\D\bfs,\D\tau)$
where the measure $\scrI$ on $\Sdec\times\bbT_c$ is defined as follows:
let $(\tau_i,\mu_i)_{i\geq 1}$ be a sequence of \iid $(\gamma,\nu)$-fragmentation trees
and for any $\bfs$ in $\Sdec$, similarly to Section~\ref{sec:ssf-trees}, set $\tau_{\langle\bfs\rangle} :=
\big\langle (s_i^\gamma \tau_i, s_i \mu_i) \,;\, i\geq 1\big\rangle$
and for all $G:\Sdec\times\bbT_c\to\bbR_+$, let $\int G \, \D\scrI :=
\int \esp {G(\bfs,\tau_{\langle\bfs\rangle})} \: I(\D\bfs)$.

Moreover, recall from the construction of Markov branching trees
with a unique infinite spine (see Remark~\ref{nb:infinite-spine-construction})
that a tree $T$ with distribution $\operatorname {MB}^{q,q_{\smash\infty}}_\infty$
is obtained by grafting at each height $n$ of an infinite branch a tree $T_n$,
where the sequence $(T_n)_{n\geq 0}$ is \iid,
is such that for all $n\geq 0$,
$\Lambda_n := \Lambda(T_n)$ has distribution $q_* = q_\infty(\infty,\,\cdot\,)$
and conditionally on $\Lambda_n = \lambda$ in $\calP_{<\infty}$,
$T_n$ has distribution $\operatorname {MB}^q_\lambda$.
As a result, $T$ is characterised by the
point process $\sum_{n\geq 0} \delta_{(n,\Lambda_n,T_n)}$
(or simply by $\sum_{n\geq 0} \delta_{(n,T_n)}$).

Therefore, when considering scaling limits of such trees,
it seems natural to take a step back and instead consider the convergence
of the underlying point processes on $\bbR_+\times\Sdec\times\bbT_c$.
We will follow the spirit of~\cite[Section~2.1.2]{haas2007fragmentationinitialmass}
and introduce a topology on the set of such point measures
adequate for our forthcoming purposes.

\medskip

Let $\scrR$ be the set of integer-valued Radon measures on $\bbR_+\times\Sdec\times\bbT_c$
which integrate the function $(u,\bfs,\tau)\longmapsto \ind_{u\leq K}\norm\bfs$ for all $K\geq 0$.
Two measures $\mu$ and $\nu$ in $\scrR$ will be called equivalent
when $\norm\bfs \, \mu (\D u, \D\bfs, \D\tau) = \norm\bfs \, \nu (\D u, \D\bfs, \D\tau)$,
meaning that $\abs{\mu-\nu}$ is supported by $\bbR_+\times\{\mathbf 0\}\times\bbT_c$.

Note $\scrF$ the set of continuous functions $F:\bbR_+\times\Sdec\times\bbT_c\longrightarrow\bbR_+$
such that there is $K\geq 0$ satisfying $F(u,\bfs,\tau) \leq \ind_{u\leq K} \norm\bfs$ for all $(u,\bfs,\tau)$.
If $\zeta$ is a random element of $\scrR$,
we define its \emph{Laplace transform} as the application $L_{\smash\zeta}:\scrF\to\bbR_+$,
defined by $ L_{\smash\zeta}(F) := \esp [1] {\exp \big( -\smash\int F\,\D\zeta \big)}$
for all $F$ in $\scrF$.

If $\mu_n$, $n\geq 1$ and $\mu$ are elements of $\scrR$, we will say that $\mu_n\to\mu$
\tiff for all $F\in\scrF$, $\int F \,\D\mu_n \to \int F \,\D\mu$.
Appendix~A7 of~\cite{kallenberg1983randommeasures} ensures that when endowed with
the topology induced by this convergence, $\scrR$ is a Polish space.
Moreover, Theorems~4.2 and~4.9 of~\cite{kallenberg1983randommeasures}
give the following criterion for convergence in distribution of elements of $\scrR$.

\begin{proposition}[\cite{kallenberg1983randommeasures}]\label{prop:point-processes-laplace-criterion}
	Let $\xi_n$, $n\geq 1$ and $\xi$ be $\scrR$-valued random variables.
	Then $\xi_n$ converges to $\xi$ in distribution with respect to the topology on $\scrR$
	\tiff for all $F\in\scrF$, $L_{\xi_{\smash n}}(F) \to L_\xi(F)$.
\end{proposition}

\medskip

The following extension of the Portmanteau theorem to finite measures
with any mass will be useful.

\begin{lemma}\label{lem:portmanteau-extension}
	Set $(M,\upd)$ a metric space
	and let $\mu_n$, $n\geq 1$ and $\mu$ be finite Borel measures on $M$.
	Then $\mu_n$ converges weakly to $\mu$
	\tiff for any bounded Lipschitz-continuous function $f:M\to\bbR$,
	$\int f \,\D\mu_n$ converges to $\int f\,\D\mu$
	as $n$ goes to infinity.
\end{lemma}

\begin{proof}
	Suppose $\int f \,\D\mu_n \to \int f\,\D\mu$
	for all Lipshitz-continuous functions $f:M\to\bbR$.
	Observe that since constant applications are Lipschitz-continuous,
	our assumption implies that $\mu_n(M) \to \mu(M)$.
	Therefore, if $\mu(M) = 0$, we directly get $\mu_n\Rightarrow\mu$.
	
	Otherwise, there exists $n_0$ such that $\mu_n(M) > 0$ for all $n\geq n_0$.
	For all such $n$, let $\tilde\mu_n := [\mu_n(M)]^{-1} \mu_n$
	and $\tilde\mu := [\mu(M)]^{-1} \mu$ which are probability measures.
	It ensues from the usual Portmanteau theorem and our assumption
	that $\tilde\mu_n\Rightarrow\tilde\mu$.
	As a result, for any bounded continuous function $f$, as $n$ goes to $\infty$,
	$\int f \, \D\mu_n = \mu_n(M) \, \int f \,\D\tilde\mu_n
	\to \mu(M) \, \int f \,\D\tilde\mu = \int f \,\D\mu$
	which is to say that $\mu_n\Rightarrow\mu$.
\end{proof}

\section{Scaling limits of infinite Markov-branching trees}\label{sec:infinite-mb-scaling-limits}

In this section, we will state and prove our main result
on scaling limits of infinite Markov branching trees
as well as its corollary on their volume growth.

\smallskip

Let $\calN$ be an infinite subset of $\bbN$ containing $1$
and let $q = (q_{n-1})_{n\in\calN}$ be a sequence of first-split distributions where for each $n$,
$q_{n-1}$ is supported by $\big\{\lambda\in\calP_{n-1} : \lambda_i\in\calN, i=1,\dots, p(\lambda)\big\}$.
Recall from Section~\ref{sec:finite-mb-trees} that the associated Markov branching family $\MB^q$ is well defined.
Furthermore, let $q_\infty$ be a probability measure on $\calP_\infty$ supported by the set
$\big\{(\infty,\lambda) : \lambda\in\calP_{<\infty}, \lambda_i\in\calN, i=1,\dots,p(\lambda) \big\}$.
In this way, the probability measure $\MB^{q,q_{\smash\infty}}_\infty$ on $\ttT_\infty$ is also well defined
and \as yields trees with a unique infinite spine.
To lighten notations, let $q_* := q_\infty(\infty, \,\cdot\,)$ which is a probability measure on $\calP_{<\infty}$.

\smallskip

In the remainder of this section, we will assume that:
\begin{itemize}
	\item[$(\mathtt {S})$]\phantomsection\label{assumption:scaling}
	There exist some $\gamma>0$ and a dislocation measure $\nu$ on $\Sdec$,
	such that $n^\gamma (1-s) \bar q_n(\D s) \Rightarrow (1-s) \nu(\D s)$.
	In particular, Theorem~\ref{thm:haas-scaling-mb} and Lemma~\ref{lem:mb-trees-height} hold.
	
	\item[$(\mathtt {I})$]\phantomsection\label{assumption:immigration}
	There exists an immigration measure $I$ on $\Sdec$
	such that if $\Lambda$ has distribution $q_*$,
	for any continuous $F:\Sdec\to\bbR_+$ with $F(\bfs) \leq 1\wedge\norm\bfs$,
	$R \, \esp [1] {F(\Lambda/R^{1/\gamma})} \to \int F\, \D I$ as $R\to\infty$.
\end{itemize}

\begin{remark}
	Under Assumption~\hyperref[assumption:immigration]{$(\mathtt I)$},
	the immigration measure $I$ satisfies the self-similarity condition
	exposed in Remark~\ref{nb:ssfi-rescaling}.
\end{remark}

\begin{theorem}\label{thm:scaling-limits-infinite-mb-trees}
	Let $T$ be an infinite Markov branching tree with distribution $\MB^{q,q_{\smash\infty}}_\infty$
	endowed with its counting measure $\mu_T$.
	Under Assumptions~\hyperref[assumption:scaling]{$(\mathtt S)$}
	and~\hyperref[assumption:immigration]{$(\mathtt I)$}, if $\gamma<1$,
	with respect to the extended GHP topology,
	\[
		\bigg(\frac T R, \frac {\mu_T} {R^{1/\gamma}}\bigg)
		\xrightarrow [R\to\infty] {}
		\scrT_{\gamma,\nu}^I
	\]
	in distribution,
	where $\scrT_{\gamma,\nu}^I$ denotes the distribution
	of a $(\gamma,\nu,I)$-fragmentation tree with immigration.
\end{theorem}

\smallskip

Let $\bfT$ be a fixed element of $\bbT$.
We define its \emph{volume growth function} as the application
$V_\bfT:\bbR_+\to\bbR_+$, $R\mapsto\mu_T(T\vert_R)$.
In other words, $V_\bfT(R)$ is the mass or volume
of the closed ball $T\vert_R$.
Once Theorem~\ref{thm:scaling-limits-infinite-mb-trees} is proved,
we will be interested in the volume growth processes associated to these trees.

\begin{proposition}\label{prop:volume-growth-cv}
	Suppose the assumptions of Theorem~\ref{thm:scaling-limits-infinite-mb-trees} are met.
	Let $T$ be an infinite Markov branching tree with distribution $\MB^{q,q_{\smash\infty}}_\infty$
	and $(\calT,\mu_\calT)$ be a $(\gamma,\nu,I)$-fragmentation tree with immigration.
	Then, the volume growth function of $(T/R, \mu_T/R^{1/\gamma})$ converges in distribution
	to that of $(\calT,\mu_\calT)$ with respect to the topology of uniform convergence on compacts of $\bbR_+$.
	In particular
	\[
		\frac {\mu_T(T\vert_R)} {R^{1/\gamma}}
		\xrightarrow [R\to\infty] {(\upd)}
		\mu_\calT (\calT\vert_1).
	\]
\end{proposition}

\smallskip

We may adapt the proofs of Theorem~\ref{thm:scaling-limits-infinite-mb-trees} and Proposition~\ref{prop:volume-growth-cv}
to get the following theorem.

\begin{theorem}
	Let $T$ be an infinite Markov branching tree with distribution
	$\MB^{\smash\calL,q,q_{\smash\infty}}_\infty$
	and endow it with the counting measure $\mu_T$ on the set of its leaves.
	If Assumptions~\hyperref[assumption:scaling]{$(\mathtt S)$}
	and~\hyperref[assumption:immigration]{$(\mathtt I)$} hold for $(q_n)_n$ and $q_\infty$ respectively,
	then the conclusions of both Theorem~\ref{thm:scaling-limits-infinite-mb-trees}
	and Proposition~\ref{prop:volume-growth-cv} hold.
\end{theorem}

\begin{remark}
	Instead of Assumption~\hyperref[assumption:immigration]{$(\mathtt I)$},
	we may assume that
	\begin{itemize}
		\item[$(\mathtt {I'})$]\phantomsection\label{assumption:immigration-prime}
		There exists $\alpha<1/\gamma$ and an immigration measure $I$ on $\Sdec$
		such that if $\Lambda$ is distributed according to~$q_*$,
		$R \, \esp [1] {F(\Lambda/R^\alpha)} \to \int F\, \D I$
		for any continuous $F:\Sdec\to\bbR_+$ with $F(\bfs) \leq 1\wedge\norm\bfs$.
	\end{itemize}
	If $T$ has distribution $\MB^{q,q_{\smash\infty}}_\infty$
	and is endowed with its counting measure $\mu_T$
	under \hyperref[assumption:scaling]{$(\mathtt S)$}
	and \hyperref[assumption:immigration-prime]{$(\mathtt I')$},
	we get that $(T/R, \mu_T/R^\alpha)$ converges in distribution
	to the infinite branch $\bbR_+$ endowed with the random measure
	$\mu = \sum_{i\geq 1} \norm {\bfs_i} \, \delta_{u_{\smash i}}$,
	where $\{(u_i,\bfs_i) ; i\geq 1\}$ are the atoms
	of a Poisson point process $\Sigma$ on $\bbR_+\times\Sdec$ with intensity $\D u \otimes I(\D\bfs)$.
	The tree $(\bbR_+,\mu)$ encodes the genealogy of a pure immigration process.
	Furthermore, $\mu_T(T\vert_R)/R^\alpha$ converges in distribution
	to $\mu([0,1]) = \int_{[0,1]\times\Sdec} \norm\bfs \, \Sigma(\D\bfs)$.
	
	Similarly, if $T$ is distributed according to $\MB^{\smash\calL,q,q_{\smash\infty}}_\infty$
	and is endowed with the counting measure on its leaves,
	the same results hold under \hyperref[assumption:scaling]{$(\mathtt S)$}
	and \hyperref[assumption:immigration-prime]{$(\mathtt I')$}.
\end{remark}

\smallskip

To prove Theorem~\ref{thm:scaling-limits-infinite-mb-trees},
we will first study the convergence of the underlying point processes in Section~\ref{sec:scaling-limits-pp}
which will give us more leeway to manipulate the corresponding trees and end the proof in Section~\ref{sec:scaling-limits-trees}.
Section~\ref{sec:volume-growth} will then focus on proving Proposition~\ref{prop:volume-growth-cv}.

\subsection{Convergence of the associated point processes}\label{sec:scaling-limits-pp}

Since $(\bbT_c, \upd_\GHP)$ is Polish,
in light of Assumption~\hyperref[assumption:scaling]{$(\mathtt S)$},
Theorem~\ref{thm:haas-scaling-mb} and Skorokhod's representation theorem,
we can find an \iid sequence $[(T_{i,n})_{n\in\calN}, \calT_i]_{i\geq 1}$,
where for each $i\geq 1$, the family $(T_{i,n})_{n\in\calN}, \calT_i$ of random trees
is such that:
\begin{itemize}
	\item $T_{i,n}$ has distribution $\MB^q_n$,
	\item $\calT_i$ is a $(\gamma,\nu)$ self-similar fragmentation tree,
	\item $(T_{i,n}/n^\gamma, \mu_{T_{\smash{i,n}}}/n) =: \overline T_{i,n}$
	\as converges to $\calT_i$ as $n\to\infty$.
\end{itemize}
For $\lambda = (\lambda_1,\dots,\lambda_p)\in\calP_{<\infty}$,
let $T_{[\lambda]} := \lBrack T_{i,\lambda_{\smash i}} \,;\, 1\leq i\leq p\rBrack$.
For any $\bfs\in\Sdec$, let $\calT_{\langle\bfs\rangle} :=
\langle (s_i^\gamma \calT_i, s_i \mu_{\calT_{\smash i}}) ; i\geq 1\rangle$
which is a compact $\bbR$-tree (see Lemma~\ref{lem:ssf-concatenation-compact}).

Finally, let $\Lambda$ be a random finite partition with distribution $q_*$
independent of $[(T_{i,n})_{n\in\calN}, \calT_i]_{i\geq 1}$,
and for any $R\geq 1$, set $q^{(R)}$ as the distribution of $\Lambda/R^{1/\gamma}$.
With these notations, Assumption~\hyperref[assumption:immigration]{$(\mathtt I)$}
becomes: $R \, (1\wedge\norm\bfs) \, q^{(R)}(\D\bfs) \Rightarrow (1\wedge\norm\bfs) \, I(\D\bfs)$
as finite measures on $\Sdec$.

\begin{lemma}\label{lem:Sdec-compact-bounded-tail}
	Let $K\subset\Sdec$ be compact.
	Then $\sup_{\bfs\in K} \sum_{i>n} s_i \to 0$
	as $n$ goes to infinity.
\end{lemma}

\begin{proof}
	Assume the contrary,
	\ie that there exists a sequence $(\bfs^{(n)})_{n\geq 1}$ in $K$
	and a positive constant $c$ such that $\sum_{i>n} s^{(n)}_i > c$ for all $n\geq 1$.
	Since $K$ is compact, we can find a subsequence $(\bfs^{(n_k)})_k$ and $\bfs\in K$
	such that $\norm {\bfs^{(n_k)}-\bfs} \to 0$ as $k\to\infty$.
	Consequently, $0 < c \leq \sum_{i>n_k} s^{\smash {(n_k)}}_i
	\leq \sum_{i>n_k} s_i + \norm {\bfs^{(n_k)}-\bfs} \to 0$
	as $k\to\infty$, which is a contradiction.
\end{proof}

Fix $G:\Sdec\times\bbT_c \to \bbR_+$ a $1$-Lipschitz function
satisfying $G(\bfs,\,\cdot\,) \leq 1 \wedge \norm\bfs$ for any $\bfs\in\Sdec$.
Moreover, set $g:\Sdec\to\bbR_+$ the function defined by
$g(\bfs) := \esp {G(\bfs,\calT_{\langle\bfs\rangle})}$.

\begin{lemma}\label{lem:tree-pp-intensity-lipschitz-cv}
	We have
	\[
		R \, \esp [3] { G\Big( R^{-1/\gamma} \Lambda,
			(R^{-1} T_{[\Lambda]}, R^{-1/\gamma} \mu_{T_{[\Lambda]}}) \Big)}
		\xrightarrow [R\to\infty] {}
		\int_{\Sdec} \esp [1] {G(\bfs,\calT_{\langle\bfs\rangle})} \: I(\D\bfs).
	\]
\end{lemma}

\begin{proof}
	Clearly, $g(\bfs) \leq 1\wedge\norm\bfs$.
	Moreover, for any $\bfs$ and $\bfr$ in $\Sdec$,
	\[
		\abs [1] {g(\bfs) - g(\bfr)}
		\leq \esp [2] {\abs [1] {G(\bfs,\calT_{\langle\bfs\rangle}) - G(\bfr,\calT_{\langle\bfr\rangle})}}
		\leq \norm {\bfs-\bfr}
			+ \esp [1] {\upd_\GHP\big(\calT_{\langle\bfs\rangle},\calT_{\langle\bfr\rangle}\big)}
		\xrightarrow [\bfr\to\bfs] {} 0
	\]
	where we have used Lemma~\ref{lem:ssf-concatenation-ghp-exp-continuity}.
	Therefore, $g$ is continuous and Assumption~\hyperref[assumption:immigration]{$(\mathtt I)$} ensures that
	\[
		R \, \esp [3] {G\Big( R^{-1/\gamma} \Lambda, \big(R^{-1} \calT_{\langle\Lambda\rangle},
			R^{-1/\gamma} \mu_{\calT_{\langle\Lambda\rangle}}\big) \Big)}
		= R \, \esp [1] {g ( R^{-1/\gamma} \Lambda)}
		\xrightarrow [R\to\infty] {}
		\int_{\Sdec} g(\bfs) \, I(\D\bfs).
	\]
	Consequently, it will be sufficient to prove that as $R\to\infty$,
	\begin{align*}
		& R \, \esp [3] { \abs [2] {
			G\Big( R^{-1/\gamma} \Lambda, (R^{-1} T_{[\Lambda]},
				R^{-1/\gamma} \mu_{T_{[\Lambda]}}) \Big)
			- G\Big( R^{-1/\gamma} \Lambda, (R^{-1} \calT_{\langle\Lambda\rangle},
				R^{-1/\gamma} \mu_{\calT_{\langle\Lambda\rangle}}) \Big)}}\\
		& \qquad \leq R \, \esp [2] { \big(1\wedge R^{-1/\gamma}\Lambda\big)
			\:\wedge\: \upd_\GHP\Big( (R^{-1} T_{[\Lambda]}, R^{-1/\gamma} \mu_{T_{[\Lambda]}}),
				(R^{-1} \calT_{\langle\Lambda\rangle}, R^{-1/\gamma} \mu_{\calT_{\langle\Lambda\rangle}})\Big) }
		=: \Delta_R \longrightarrow 0.
	\end{align*}
	
	For all $n\geq 1$, thanks to Lemma~\ref{lem:continuous-concatenation-ghp-bounds}
	we get
	\begin{align*}
		& \upd_\GHP\Big( (R^{-1} T_{[\Lambda]}, R^{-1/\gamma} \mu_{T_{[\Lambda]}}),
				(R^{-1} \calT_{\langle\Lambda\rangle}, R^{-1/\gamma} \mu_{\calT_{\langle\Lambda\rangle}})\Big)\\
		& \qquad\qquad \leq \vphantom\int
		\smash {\sum_{i=1}^n} \upd_\GHP\Big( (R^{-1} T_{i,\Lambda_i}, R^{-1/\gamma} \mu_{T_{i,\Lambda_i}}),
				(R^{-1}\Lambda_i^\gamma \calT_i, R^{-1/\gamma}\Lambda_i \mu_{\calT_i})\Big)\\
		& \qquad\qquad\qquad\qquad \vphantom\int
			\smash {+ \sup_{i>n} \bigg( \frac {\Lambda_i^\gamma} R \abs {\overline T_{i,\Lambda_i}} \bigg)
			+ \sup_{i>n} \bigg( \frac {\Lambda_i^\gamma} R \abs {\calT_i} \bigg)
			+ 2 \sum_{i>n} \frac {\Lambda_i} {R^{1/\gamma}}},
	\end{align*}
	and for each $i\geq 1$, Lemma~\ref{lem:ghp-rescaling-bounds} gives
	\begin{align*}
		& \upd_\GHP\Big( (R^{-1} T_{i,\Lambda_i}, R^{-1/\gamma} \mu_{T_{i,\Lambda_i}}),
				(R^{-1}\Lambda_i^\gamma \calT_i, R^{-1/\gamma}\Lambda_i \mu_{\calT_i})\Big)
		\leq \bigg( \frac {\Lambda_i^\gamma} R \,\vee\, \frac {\Lambda_i} {R^{1/\gamma}} \bigg) \,
			\upd_\GHP \big( \overline T_{i,\Lambda_i}, \calT_i\big).
	\end{align*}
	
	Let $\varepsilon>0$ be fixed.
	As a result of Assumption~\hyperref[assumption:immigration]{$(\mathtt I)$},
	the sequence $R \, (1\wedge\norm\bfs) \, q^{(R)}(\D\bfs)$, $R\geq 1$ is tight
	and so there exists a compact subset $K$ of $\Sdec$ such that
	$\sup_{R\geq 1} R \int (1\wedge\norm\bfs) \,\big(1-\ind_K (\bfs)\big)\, q^{(R)}(\D\bfs) < \varepsilon$.
	Moreover, as a compact subset, $K$ is bounded,
	\ie $\sup_{\bfs\in K} \norm\bfs = C < \infty$.
	
	For all $n\geq 1$, recall that
	$\upd_\GHP ( \overline T_{1,n}, \calT_1 )
	\leq 2 \vee \abs {\overline T_{1,n}} \vee \abs {\calT_1}$.
	As a result, thanks to Lemma~\ref{lem:mb-trees-height},
	\begin{align*}
		\sup\nolimits_n \esp [2] {\big( \upd_\GHP ( \overline T_{1,n}, \calT_1 ) \big)^2}
		\leq 3 \Big( 2^2 + \sup\nolimits_n \esp [1] {\abs {\overline T_{1,n}}^2} + \esp [1] {\abs {\calT_1}^2}\Big)
		\leq 12 + 6 \, h_2
		< \infty,
	\end{align*}
	so the sequence $\big[\upd_\GHP( \overline T_{1,n}, \calT_1)\big]_n$
	is bounded in $L^2$.
	Since by assumption, it converges to $0$ \as, it also does in $L^1$.
	Furthermore, $\sup_n \esp {\upd_\GHP( \overline T_{1,n}, \calT_1)} =: D$ is finite.
	Consequently, and because the sequence of families $\big\{(T_{i,n})_n, \calT_i\big\}_{\smash {i\geq 1}}$ is \iid,
	for any $\eta>0$, there exists $N$ such that for all $i\geq 1$ and $n\geq N$,
	$\esp [1] {\upd_\GHP \big( \overline T_{i,n}, \calT_i \big)} < \eta$.
	This gives the rather crude following bound
	\[
		\esp [1] {\upd_\GHP( \overline T_{i,n}, \calT_i)}
		\leq D\,\ind_{n<N} + \eta.
	\]
	
	For all $\delta>0$, in light of~Lemma~\ref{lem:Sdec-compact-bounded-tail},
	there exists an integer $m_{K,\delta}$ which depends only on $K$ and $\delta$
	such that $\sup_{\bfs\in K} \sum_{i>m_{\smash {K,\delta}}} s_i < \delta$.
	Then for all $R\geq 1$ and $\lambda\in\calP_{<\infty}$ with $\lambda/R^{1/\gamma}\in K$,
	if $\gamma\leq 1$, Jensen's inequality gives
	\[
		\esp [4] {\sup_{i>m_{K,\delta}} \bigg( \frac {\lambda_i^\gamma} R \abs {\overline T_{i,\lambda_i}} \bigg)}
		\leq \Bigg( \esp [3] {\sup_{i>m_{K,\delta}} \frac {\lambda_i} {R^{1/\gamma}}
			\abs {\overline T_{i,\lambda_i}}^{1/\gamma}} \Bigg)^\gamma
		\leq \Bigg(\sum_{i>m_{K,\delta}} \frac {\lambda_i} {R^{1/\gamma}}
			\esp [1] {\abs {\overline T_{i,\lambda_i}}^{1/\gamma}} \Bigg)^\gamma
		\leq (h_{1/\gamma})^\gamma \, \delta^\gamma
	\]
	where $h_{1/\gamma}$ is the constant from Lemma~\ref{lem:mb-trees-height}.
	Otherwise, if $\gamma>1$, since $(\lambda_i)_{i\geq 1}$ is a non-increasing sequence,
	\[
		\esp [4] {\sup_{i>m_{K,\delta}} \bigg( \frac {\lambda_i^\gamma} R \abs {\overline T_{i,\lambda_i}} \bigg)}
		\leq \bigg(\frac {\lambda_{m_{K,\delta}+1}} {R^{1/\gamma}} \bigg)^{\gamma-1}
			\esp [4] {\sup_{i>m_{K,\delta}} \frac {\lambda_i} {R^{1/\gamma}} \abs {\overline T_{i,\lambda_i}}}
		\leq \delta^{\gamma-1} \sum_{i>m_{K,\delta}} \frac {\lambda_i} {R^{1/\gamma}}
			\esp [1] {\abs {\overline T_{i,\lambda_i}}}
		\leq h_1 \, \delta^\gamma
	\]
	where $h_1$ is defined as in Lemma~\ref{lem:mb-trees-height}.
	Similarly,
	\[
		\esp [4] {\sup_{i>m_{K,\delta}} \bigg( \frac {\lambda_i^\gamma} R \abs {\calT_i} \bigg)}
		\leq \begin{cases}
			(h_{1/\gamma})^\gamma \, \delta^\gamma
			& \text {if $\gamma\leq 1$,}\\
			h_1 \, \delta^\gamma
			& \text{if $\gamma>1$.}
		\end{cases}
	\]
	In summary, for all $\lambda$ in $\calP_{<\infty}$
	such that $\lambda/R^{1/\gamma}$ belongs to $K$,
	we get that
	\[
		\esp [4] {\sum_{i>m_{K,\delta}} \frac {\lambda_i} {R^{1/\gamma}}} \leq \delta
		\qquad\text{and}\qquad
		\esp [4] {\sup_{i>m_{K,\delta}} \bigg( \frac {\lambda_i^\gamma} R \abs {\overline T_{i,\lambda_i}} \bigg)
			+ \sup_{i>m_{K,\delta}} \bigg( \frac {\lambda_i^\gamma} R \abs {\calT_i} \bigg)}
		\leq B \, \delta^\gamma
	\]
	for some finite constant $B$
	independent of $\varepsilon$, $\eta$, $\delta$ and $K$.
	
	Therefore, for all positive $\varepsilon$, $\delta$ and $\eta$,
	\begin{align*}
		\Delta_R
		& \leq \varepsilon + R \, \esp [4] {\ind_K\bigg(\frac\Lambda{R^{1/\gamma}}\bigg) \:
			\bigg(1\wedge \frac{\norm\Lambda}{R^{1/\gamma}}\bigg)
		\,\wedge\, \Bigg(
		\sum_{i=1}^{m_{K,\delta}} \bigg( \frac {\Lambda_i^\gamma} R \,\vee\, \frac {\Lambda_i} {R^{1/\gamma}} \bigg)
			\esp [3] {\upd_\GHP \big( \overline T_{i,\Lambda_i}, \calT_i\big)\,\Big\vert\,\Lambda\,}\\
		& \qquad\qquad\qquad\qquad\qquad\qquad\qquad
		+ \esp [3] {\sup_{i>m_{K,\delta}} \frac {\Lambda_i^\gamma} R \, \abs {\overline T_{i,\Lambda_i}}
				+ \sup_{i>m_{K,\delta}} \frac {\Lambda_i^\gamma} R \, \abs {\calT_i}
				+ 2 \sum_{i>m_{K,\delta}} \frac {\Lambda_i} {R^{1/\gamma}} \,\Big\vert\,\Lambda\,} 
			\,\Bigg)}\\
		& \leq \varepsilon + R \, \esp [4] {
			\bigg(1\wedge\frac{\norm\Lambda}{R^{1/\gamma}}\bigg)
		\, \wedge\, \bigg(
			(C + C^\gamma) \,m_{K,\delta} \eta
			+ \Big( \frac {N^\gamma} R + \frac N {R^{1/\gamma}} \Big) \, m_{K,\delta} D
			+ 2\delta + B \delta^\gamma \bigg)}.
	\end{align*}
	Let $\delta$ be such that $(\delta + \delta^\gamma) B < \varepsilon$
	and set $\eta<\varepsilon/[(C+C^\gamma) m_{K,\delta}]$.
	Because of Assumption~\hyperref[assumption:immigration]{$(\mathtt I)$},
	we therefore get that $\limsup_{R\to\infty} \, \Delta_R \leq O(\varepsilon)$
	from which it follows that $\Delta_R \to 0$.
\end{proof}

Since the conclusion of Lemma~\ref{lem:tree-pp-intensity-lipschitz-cv} is met
for any Lipschitz continuous function $G:\Sdec\times\bbT_c\to\bbR_+$
with $G(\bfs,\,\cdot\,)\leq 1\wedge\norm\bfs$,
Lemma~\ref{lem:tree-pp-intensity-lipschitz-cv} gives the following corollary:

\begin{corollary}\label{cor:point-process-cv-prelim}
	The convergence of Lemma~\ref{lem:tree-pp-intensity-lipschitz-cv} holds
	for any continuous $G$ with $G(\bfs,\,\cdot\,) \leq 1 \wedge \norm\bfs$.
\end{corollary}

We will now prove that the point processes associated to
adequately rescaled Markov branching trees with a unique infinite spine
converge in distribution to the point process associated to fragmentation trees with immigration.
Let $\Pi$ be a Poisson point process on $\bbR_+\times\Sdec\times\bbT_c$
with intensity $\D u \otimes \scrI (\D\bfs, \D\tau)$,
where $\scrI$ is the measure defined at the beginning of Section~\ref{sec:point-processes}.
Observe that for all $K\geq 0$,
\[
	\textstyle
	\int \ind_{u\leq K} \big(1\wedge\norm\bfs\big) \, \D u \otimes \scrI(\D\bfs,\D\tau)
	= K \int_{\Sdec} \big(1\wedge\norm\bfs\big) \, I(\D\bfs) < \infty.
\]
Campbell's theorem (see~\cite[Section~3.2]{kingman1992poissonprocesses})
therefore ensures that $\Pi$ \as satisfies the integrability conditions
necessary to belong to the set $\scrR$ of point measures on $\bbR_+\times\Sdec\times\bbT_c$
defined in Section~\ref{sec:point-processes}.

Let $T$ have distribution $\MB^{q,q_{\smash\infty}}_\infty$.
By construction of Markov branching trees with a unique infinite spine (see Remark~\ref{nb:infinite-spine-construction}),
there exists a sequence $(\Lambda_n,T_n)_{n\geq 0}$ of \iid random variables
such that $T = \ttb_\infty \bigotimes_{n\geq 0} (\ttv_n, T_n)$,
where $\Lambda_n$ is distributed according to $q_*$ and conditionally on $\Lambda_n = \lambda$,
$T_n$ has distribution $\MB^q_\lambda$.
For all $R\geq 1$, let $\Pi_R$ be the point process associated to $(T/R,\mu_T/R^{1/\gamma})$,
\ie the $\scrR$-valued random variable defined for all measurable
$f:\bbR_+\times\Sdec\times\bbT_c\longrightarrow\bbR_+$ by
\[
	\textstyle
	\int f \, \D \Pi_R :=
	\sum_{n\geq 0} f\big[n/R, \,\Lambda_n/R^{1/\gamma}, \, (T_n/R, \mu_{T_n}/R^{1/\gamma})\big].
\]

\begin{lemma}\label{lem:point-processes-cv}
	With respect to the topology on $\scrR$ introduced in Section~\ref{sec:point-processes},
	$\Pi_R$ converges to $\Pi$ in distribution as $R$ goes to infinity.
\end{lemma}

\begin{proof}
	In light of Proposition~\ref{prop:point-processes-laplace-criterion}, it will be enough
	to prove that for any function $F$ in the set $\scrF$,
	the Laplace transform of $\Pi_R$ evaluated in $F$ converges to that of $\Pi$.
	Fix such $F$ in $\scrF$ and recall that it is continuous
	and that there exists $K\geq 0$ such that $0\leq F(u,\bfs,\tau)\leq \norm\bfs \, \ind_{u\leq K}$
	for all $(u,\bfs,\tau)$.
	Campbell's theorem for Poisson point processes gives
	\[
		\textstyle
		L_\Pi(F)
		= \exp \Big( - \int \big[1 - \E^{-F(u,\bfs,\tau)}\big]
			\, \D u \otimes \scrI (\D\bfs, \D\tau) \Big).
	\]
	For all $R\geq 1$ and $u\geq 0$, set
	\begin{align*}
		\varphi_R(u) & :=
		R \, \esp [2] {1 - \exp \Big(- F\big[u,\Lambda_0/R^{1/\gamma},
			(T_0/R, \mu_{T_0}/R^{1/\gamma}) \big]\Big)},\\
		\mathllap{\text{and} \qquad}
		\varphi(u) & :=
		{\textstyle\int} \esp [1] {1 - \exp \big(- F
			[u, \bfs, \calT_{\langle\bfs\rangle} ] \big)} \, I(\D\bfs).
	\end{align*}
	Using these notations,
	we may write $\log \, L_\Pi(F) = - \int_0^K \varphi(u) \, \D u$
	and thanks to the \iid nature of the sequence $(\Lambda_n,T_n)_{n\geq 0}$,
	for all $R\geq 1$,
	\begin{align*}
		\log \, L_{\Pi_R}(F)
		& = - {\textstyle\sum_{n=0}^{\lfloor KR\rfloor}} \log \, \esp [2] {\exp\Big(
			- F\big[ n/R, \Lambda_0/R^{1/\gamma}, (T_0/R, \mu_{T_0}/R^{1/\gamma})\big] \Big)}\\
		& = - {\textstyle\sum_{n=0}^{\lfloor KR\rfloor}} \log \, \Big( 1 - 1/R \cdot \varphi_R(n/r) \Big).
	\end{align*}
	
	The functions $\varphi_R$, $R\geq 1$ and $\varphi$ all have support in $[0,K]$
	and are continuous (in light of the dominated convergence theorem).
	Observe that $0\leq 1-\E^{-F(u,\bfs,\tau)} \leq 1\wedge\norm\bfs$.
	From Corollary~\ref{cor:point-process-cv-prelim}, we know that for all fixed $u\geq 0$,
	$\varphi_R(u) \to \varphi(u)$ as $R\to\infty$
	and that furthermore
	\[
		\textstyle
		\sup_{R\geq 1} \, \sup_{u\geq 0} \: \varphi_R(u)
		\leq \sup_{R\geq 1} \: R \, \esp [1] {1\wedge (\norm{\Lambda_0}/R^{1/\gamma})}
		< \infty,
	\]
	\ie that the sequence $(\varphi_R)_{R\geq 1}$ is uniformly bounded by a finite constant, say $C$.
	Let $\varepsilon$ be positive.
	It also follows from Corollary~\ref{cor:point-process-cv-prelim} that
	there exists a compact subset $A$ of $\Sdec\times\bbT_c$ with
	\[
		\textstyle
		\sup_{R\geq 1} \: R \, \esp [2] {\big(1\wedge (\norm{\Lambda_0}/R^{1/\gamma})\big)
			\cdot \ind_{A^c} \big(\Lambda_0/R^{1/\gamma}, (T_0/R,\mu_{T_0}/R^{1/\gamma})\big)}
		< \varepsilon.
	\]
	Recall that $F$ is continuous, hence there exists $\delta>0$
	such that for any $(u,\bfs,\tau)$ and $(u',\bfs',\tau')$ in the compact set $[0,K] \times A$,
	if $\abs {u-u'} + \norm {\bfs-\bfs'} + \upd_\GHP(\tau,\tau') < \delta$,
	then $\abs {F(u,\bfs,\tau) - F(u',\bfs',\tau')} < \varepsilon$.
	As a result, and because $x\mapsto\E^{-x}$ is $1$-Lipschitz continuous on $\bbR_+$,
	for all $R\geq 1$ and $u,v$ in $[0,K]$ with $\abs {u-v} < \delta$,
	\begin{align*}
		\abs [1] {\varphi_R(u) - \varphi_R(v)}
		& \leq R \, \esp [3] {1\wedge \abs [1] {
			F\big[u, \Lambda_0/R^{1/\gamma}, (T_0/R, \mu_{T_0}/R^{1/\gamma}) \big]
			- F\big[v, \Lambda_0/R^{1/\gamma}, (T_0/R, \mu_{T_0}/R^{1/\gamma}) \big]}}\\
		& \leq \varepsilon + R \, \esp [3] {\big(\varepsilon \wedge (\norm{\Lambda_0}/R^{1/\gamma})\big)
			\cdot \ind_A \big(\Lambda_0/R^{1/\gamma}, (T_0/R,\mu_{T_0}/R^{1/\gamma})\big)}
		= O(\varepsilon).
	\end{align*}
	This ensures that the sequence $(\varphi_R)_{R\geq 1}$ is equicontinuous on $[0,K]$.
	It follows from the Arzel\`a-Ascoli theorem that $\varphi_R$ converges uniformly to $\varphi$.
	In turn, we deduce that
	\[
		\abs [2] {\frac 1 R {\textstyle\sum_{n=0}^{\lfloor KR\rfloor}} \varphi_R(n/R)
			- \frac 1 R {\textstyle\sum_{n=0}^{\lfloor KR\rfloor}} \varphi(n/R)}
		\leq \frac {KR + 1} R \, \sup_{0\leq u\leq K} \abs [1] {\varphi_R(u)-\varphi(u)}
		\xrightarrow [R\to\infty] {} 0.
	\]
	Moreover, because $\sup_{R\geq 1, u\geq 0} \varphi_R(u) \leq C$,
	\begin{align*}
		\abs [2] {\log L_{\Pi_R}(F)
			- 1/R \cdot {\textstyle\sum_{n=0}^{\lfloor KR\rfloor}} \varphi_R(n/R)}
		& = {\textstyle\sum_{n=0}^{\lfloor KR\rfloor}}
		\abs [2] {1/R \cdot \varphi_R(n/R)
			- \log \big[1-1/R \cdot \varphi_R(n/R)\big]}\\
		& \leq (KR + 1) \, \abs [1] {C/R - \log(1-C/R)}
		= O(1/R)
		\xrightarrow [R\to\infty] {} 0
	\end{align*}
	where we have used the fact that the application $[0,1)\to\bbR_+$, $x\mapsto x-\log(1-x)$
	increases with $x$.
	Finally, as Riemann sums of the continuous function $\varphi$,
	\[
		\frac 1 R {\textstyle\sum_{n=0}^{\lfloor KR\rfloor}} \varphi(n/R)
		\xrightarrow [R\to\infty] {}
		{\textstyle\int_0^K} \varphi(u) \, \D u
		= \log \, L_\Pi (F).
	\]
	In summary, $\log \, L_{\Pi_{\smash R}} (F) \to \log \, L_\Pi (F)$ when $R\to\infty$.
\end{proof}

\subsection{Proof of \texorpdfstring{Theorem~\ref{thm:scaling-limits-infinite-mb-trees}}%
	{Theorem~\ref*{thm:scaling-limits-infinite-mb-trees}}}\label{sec:scaling-limits-trees}

Now that we know that the underlying point processes converge,
we can prove convergence of the trees themselves.

\smallskip

Recall that the topology we defined on $\scrR$ in Section~\ref{sec:point-processes}
makes it a Polish topological space.
As such, Skorokhod's representation theorem holds for $\scrR$-valued random variables.
In particular, because of Lemma~\ref{lem:point-processes-cv},
there exist:
\begin{itemize}
	\item A Poisson point process $\Pi$ with intensity $\D u \otimes \scrI(\D\bfs,\D\tau)$,
	\item A family $\smash {\big\{ (\Lambda^{(R)}_n, \tau^{(R)}_n)_{n\geq 0} ; R\in\bbN \big\}}$
	such that for all fixed $R\geq 1$, $(\Lambda^{(R)}_n,\tau^{(R)}_n)_{n\geq 0}$ is an \iid sequence,
	$\smash {\Lambda^{(R)}_n}$ follows $q_*$ and conditionally on $\smash {\Lambda^{(R)}_n} = \lambda$,
	$\tau^{(R)}_n$ has distribution $\MB^q_\lambda$
	and is endowed with the measure $\mu_{\smash {\tau^{(R)}_n}} :=
	\sum_{u\in\smash {\tau^{(R)}_n}} \delta_u$,
\end{itemize}
such that if for any $R$ we let $\Pi_R$ be the random element of $\scrR$ defined
for all measurable $f:\bbR_+\times\Sdec\times\bbT_c\longrightarrow\bbR_+$
by $\int f\,\D\Pi_R := \sum_{n\geq 0} f\big[n/R, \Lambda^{(R)}_n/R^{1/\gamma},
(\tau^{(R)}_n/R,\mu_{\smash {\tau^{(R)}_n}}/R^{1/\gamma})\big]$,
then $\Pi_R$ \as converges to $\Pi$ when $R\to\infty$.

Let $\{(u_i,\bfs_i,\calT_i) ; i\geq 1\}$ be the atoms of $\Pi$
and set $\Sigma := \sum_{i\geq 1} \delta_{(u_{\smash i}, \bfs_{\smash i})}$. 
By definition of the intensity measure of $\Pi$,
there exists a family $\{ \calT_{i,\smash j} \, ; \, i,j\geq 1 \}$
of \iid $(\gamma,\nu)$-fragmentation trees independent of $\Sigma$
such that for all $i\geq 1$, $\calT_i := \langle (s_{i,j}^\gamma\calT_{i,j},s_{i,j}\mu_{\smash {\calT_{i,j}}}) ; j\geq 1\rangle$.
Set $\calT^{(I)} := \bfG(\{(u_i,\calT_i) ; i\geq 1\})$
where $\bfG$ is the continuous grafting application defined in Section~\ref{sec:ghp-extension}
and recall that it is a $(\gamma,\nu)$-fragmentation tree with immigration $I$
(see Section~\ref{sec:frag-trees-immigration}).
For all $\varepsilon>0$, let
\[
	\calT^{(I)}_\varepsilon
	:= \bfG\big(\{(u_i, \calT_i) ; i\geq 1,\norm{\bfs_i}\geq \varepsilon\}\big).
\]
This tree can be thought of as $\calT^{(I)}$ on which
all sub-trees grafted on the spine with mass less than $\varepsilon$
have been cut away.
Observe that because of the definition of the application $\bfG$,
the measure on $\calT^{(I)}_\varepsilon$ is simply the restriction
of $\mu_{\calT^{\smash {(I)}}}$ to $\calT^{(I)}_\varepsilon$.

For all $R$, set $\tau^{(R)} := \ttb_\infty \bigotimes_{n\geq 0} (\ttv_n, \tau^{(R)}_n)$
and note $\mu_{\tau^{\smash{(R)}}}$ its counting measure.
Observe that $\tau^{(R)}$ is distributed according to $\MB^{q,q_{\smash\infty}}_\infty$.
Let $T^{(R)} := (R^{-1}\tau^{(R)}, R^{-1/\gamma}\mu_{\smash{\tau^{(R)}}})$
be the rescaled infinite Markov branching tree associated to $\Pi_R$.
Moreover, for all positive $\varepsilon$, let $\smash {T^{(R)}_\varepsilon}$ be
the tree obtained by removing from $T^{(R)}$ all the sub-trees grafted on its spine
with mass less than $\varepsilon$, \ie set
\[
	T^{\smash {(R)}}_\varepsilon :=
	\bfG\Big(\Big\{\big[n/R, (R^{-1}\tau^{(R)}_n, R^{-1/\gamma}\mu_{\smash{\tau^{(R)}_n}})\big]
		\;\big\vert\;
		n\geq 0: \norm{\Lambda^{\smash {(R)}}_n}\geq R^{\smash{1/\gamma}} \varepsilon\Big\}\Big).
\]
The tree $T^{\smash {(R)}}_\varepsilon$ is clearly a subset of $T^{(R)}$
and it is endowed with the restriction of $\mu_{T^{\smash {(R)}}}$.

\medskip

In this section we will endeavour to prove Theorem~\ref{thm:scaling-limits-infinite-mb-trees}.
In order to do so, we will use the following criterion for convergence in distribution.

\begin{theorem}[\cite{billingsley2013convergence}, Theorem~3.2]\label{thm:billingsley-weak-cv-criterion}
	Let $(M,\upd)$ be a metric space.
	If $X_n^{\vphantom{(k)}}$, $X_n^{(k)}$, $X^{(k)}_{\vphantom n}$, $n\geq 1$, $k\geq 1$
	and $X$ are $M$-valued random variables satisfying:
	\begin{enumerate}
		\renewcommand{\labelenumi}{$(\roman{enumi})$}
		\item For all $k\geq 1$, $X_n^{(k)} \Rightarrow X^{(k)}_{\vphantom n}$ as $n\to\infty$,
		\item $X^{(k)}_{\vphantom n} \Rightarrow X$ as $k\to\infty$,
		\item For any positive $\eta$,
		$\lim_{k\to\infty} \limsup_{n\to\infty}
		\prob [1] {\upd(X_n^{(k)},X_n^{\vphantom{(k)}}) > \eta} = 0$,
	\end{enumerate}
	Then $X_n^{\vphantom {(k)}}$ converges to $X$ in distribution.
\end{theorem}

\begin{remark}
	Condition $(i)$ is akin to finite-dimensional convergence of $X_n$ to $X$
	and Conditions~$(ii)$ and~$(iii)$ to tightness of $(X_n)_n$.
\end{remark}

In our setting, the sequence $(T^{(R)} ; R\in\bbN)$ of rescaled $\MB^{q,q_{\smash\infty}}_\infty$ trees will play the role of $(X_n)_n$
and the limit variable $X$ will be $\calT^{(I)}$, a $(\gamma,\nu)$-fragmentation tree with immigration $I$.
The intermediate family \smash {$(X_n^{(k)})_{n,k}$} will be replaced by $(\smash {T^{(R)}_\varepsilon} ; R\geq 1)$
with $\varepsilon\to 0$ along some countable subset of $(0,\infty)$.
Similarly, we'll consider $\calT^{(I)}_\varepsilon$ trees instead of \smash {$(X^{(k)})_k$}.

\medskip

\begin{lemma}\label{lem:ssfi-pruning-cv}
	With these notations, $\calT^{(I)}_\varepsilon$ \as converges to $\calT^{(I)}$
	as $\varepsilon\to 0$ with respect to $\upD_\GHP$.
\end{lemma}

\begin{proof}
	For all $\varepsilon>0$, let $C_\varepsilon$ be the correspondence between $\calT^{(I)}$ and $\calT^{(I)}_\varepsilon$
	defined by $C_\varepsilon := \big\{(x,x) : x\in\calT^{(I)}_\varepsilon\big\}
	\cup \bigcup_{i\geq 1 : \norm{\bfs_i}<\varepsilon} \calT_i \times\{u_i\}$
	and set $\pi_\varepsilon$, the boundedly finite Borel measure on $\calT^{(I)}\times \calT^{(I)}_\varepsilon$,
	such that for all Borel $A$, $\pi_\varepsilon(A) :=
	\int_{\calT^{(I)}_\varepsilon} \ind_A (x,x) \, \mu_{\calT^{(I)}_\varepsilon}(\D x)$.
	Let $K\geq 0$ be fixed.
	Note $\pi_\varepsilon\vert_K$ the restriction of $\pi_\varepsilon$ to $\calT^{(I)}\vert_K\times\calT^{(I)}_\varepsilon\vert_K$.
	The monotone convergence theorem yields
	\[
		\upD\big( \pi_\varepsilon\vert_K ; \mu_{\calT^{(I)}}\vert_K, \mu_{\calT^{(I)}_\varepsilon}\vert_K\big)
		= \pi_\varepsilon\vert_K(C_\varepsilon^c)
		\leq \int \norm\bfs \ind_{\norm\bfs<\varepsilon} \ind_{u\leq K} \Sigma (\D u, \D\bfs)
		\xrightarrow [\varepsilon\to 0] {\text{\as}} 0.
	\]
	Let $C_\varepsilon\vert_K := C_\varepsilon\cap\big(\calT^{(I)}\vert_K\times\calT^{(I)}_\varepsilon\vert_K\big)$
	and observe that it is a correspondence between $\calT^{(I)}\vert_K$ and $\calT^{(I)}_\varepsilon\vert_K$.
	Its distortion satisfies
	\[
		\dis C_\varepsilon\vert_K
		\leq 2 \, \sup \, \Big\{ \abs{\calT_i} \,:\, i\geq 1, u_i\leq K,\norm{\bfs_i}<\varepsilon \Big\}
		\xrightarrow [\varepsilon\to 0] {\text {\as}} 0.
	\]
	As a result, $\upd_\GHP\big(\calT^{(I)}\vert_K, \calT^{(I)}_\varepsilon\vert_K\big)
	\to 0$ \as as $\varepsilon\to 0$.
	Since this holds for all $K\geq 0$, Proposition~\ref{prop:extended-ghp-properties}~$(ii)$
	ensures that $\upD_\GHP\big(\calT^{(I)}, \calT^{(I)}_\varepsilon\big)$
	\as converges to $0$ when $\varepsilon\to 0$.
\end{proof}

\begin{lemma}\label{lem:infinite-mb-pruning-unif-cv}
	For all positive $\eta$,
	\[
		\adjustlimits\lim_{\varepsilon\to 0} \limsup_{R\to\infty} \:
			\prob [2] {\upD_\GHP \big(T^{(R)}_{}, T^{(R)}_\varepsilon\big) > \eta}
		= 0.
	\]
\end{lemma}

\begin{proof}
	We will proceed in a way similar to the proof of~Lemma~\ref{lem:ssfi-pruning-cv}.
	For all $R\geq 1$ and $\varepsilon>0$,
	define the correspondence $C^{\smash {(R)}}_\varepsilon$ between $T^{(R)}$ and $T^{\smash {(R)}}_\varepsilon$
	as $C^{\smash {(R)}}_\varepsilon := \big\{(u,u) : u\in T^{\smash {(R)}}_\varepsilon\big\}
	\cup \big\{(u,n/R) : n\geq 1, \norm{\Lambda_n}<R^{1/\gamma}\varepsilon, u\in\tau^{\smash {(R)}}_n\big\}$
	and let $\pi^{\smash {(R)}}_\varepsilon$ be the boundedly finite measure
	$T^{(R)}\times T^{\smash {(R)}}_\varepsilon$ defined for all Borel set $A$
	by $\pi^{\smash {(R)}}_\varepsilon (A) := \int_{T^{(R)}_\varepsilon} \ind_A (x,x) \, \mu_{T^{(R)}_\varepsilon} (\D x)$.
	
	For all $K\geq 0$, set $C^{\smash {(R)}}_\varepsilon\big\vert_K := 
	C^{\smash {(R)}}_\varepsilon \cap \big(T^{(R)}\vert_K\times T^{\smash {(R)}}_\varepsilon\vert_K\big)$,
	which is a correspondence between $T^{(R)}\vert_K$ and $T^{\smash {(R)}}_\varepsilon\vert_K$,
	and let $\pi^{\smash {(R)}}_\varepsilon\big\vert_K$ be the restriction of $\pi^{\smash {(R)}}_\varepsilon$
	to $T^{(R)}\vert_K \times T^{\smash {(R)}}_\varepsilon\vert_K$.
	Then, for any non-negative~$K$,
	\[
		\dis_{T^{(R)}\vert_K,T^{(R)}_\varepsilon\vert_K}
			C^{(R)}_\varepsilon\big\vert_K
		\leq \frac 2 R \, \sup \, \Big\{
			\abs{\tau^{(R)}_n} \,:\, 0\leq n\leq RK,
			\, \norm{\Lambda^{(R)}_n} < R^{1/\gamma}\varepsilon\Big\}.
	\]
	For all $n\geq 0$ and $R\geq 1$,
	$\abs {\tau^{(R)}_n} = 1 + \sup \{ \abs {\tau^{(R)}_{n,i}} : 1\leq i\leq p(\Lambda^{(R)}_n)\}$.
	Further observe that thanks to Lemma~\ref{lem:mb-trees-height}, we can find a finite constant $h$
	such that for all $n\geq 0$, $R\geq 1$ and $i=1,\dots,p(\Lambda_n^{(R)})$,
	$\esp {(1+\abs{\tau^{(R)}_{\smash {n,i}}})^{1/\gamma} \vert \Lambda^{(R)}_n}
	\leq h \, \Lambda^{(R)}_n(i)$.
	Therefore, since the sequence $(\Lambda^{(R)}_n, \tau^{(R)}_n)_{n\geq 1}$ is \iid,
	\begin{align*}
		\esp [3] {\Big(\dis_{T^{(R)}\vert_K,T^{(R)}_\varepsilon\vert_K}
			C^{(R)}_\varepsilon\big\vert_K \Big)^{1/\gamma}}
		& \leq (KR+1) \, \frac {2^{1/\gamma}} {R^{1/\gamma}} \, \esp [2] {
			{\textstyle\sum_{i=1}^{p(\Lambda^{(R)}_0)}} (1+\abs{\tau^{(R)}_{0,i}})^{1/\gamma}
			\ind_{\norm{\Lambda_0^{(R)}}<R^{1/\gamma}}}\\
		& \leq (KR+1) \, \frac {2^{1/\gamma} h} {R^{1/\gamma}} \, \esp [2] {\norm{\Lambda^{(R)}_0}
			\, \ind_{\norm {\Lambda^{(R)}_0}< R^{1/\gamma} \varepsilon}}.
	\end{align*}
	Similarly,
	\begin{align*}
		\esp [2] {\upD\Big(\pi^{(R)}_\varepsilon\big\vert_K ;
			\mu_{T^{(R)}}\big\vert_K, \mu_{T^{(R)}_\varepsilon}\big\vert_K\Big)}
		& = \esp [2] {\pi^{(R)}_\varepsilon\big\vert_K \big[ (C^{(R)}_\varepsilon)^c \big]}
		= (KR+1) \, \frac 1 {R^{1/\gamma}} \, \esp [2] {\norm{\Lambda^{(R)}_0}
			\, \ind_{\norm {\Lambda^{(R)}_n}< R^{1/\gamma} \varepsilon}}.
	\end{align*}
	In light of Assumption~\hyperref[assumption:immigration]{$(\mathtt I)$},
	\[
		(KR+1) \, \frac 1 {R^{1/\gamma}} \, \esp [2] {\norm{\Lambda^{(R)}_0}
			\, \ind_{\norm {\Lambda^{(R)}_0}< R^{1/\gamma} \varepsilon}}
		\leq (KR+1) \, \esp [3] {\varepsilon\wedge\frac {\norm{\Lambda^{(R)}_0}} {R^{1/\gamma}}}
		\xrightarrow [R\to\infty] {} K \int (\varepsilon\wedge\norm\bfs) \, I(\D\bfs)
	\]
	
	Finally, for any positive $\eta$, if $K > -2\,\log\eta$,
	using Markov's inequality and the monotone convergence theorem,
	\begin{align*}
		& \limsup_{R\to\infty} \:
			\prob [2] {\upD_\GHP \big(T^{(R)}_{}, T^{(R)}_\varepsilon\big) > \eta}
		\leq \limsup_{R\to\infty}
			\prob [2] {\upD_\GHP \big(T^{(R)}_{}\vert_K, T^{(R)}_\varepsilon\vert_K\big)
			> \eta - 2 \E^{-K}}\\
		& \qquad\qquad\qquad
		\leq \limsup_{R\to\infty} \: \left(
			\frac {\esp [2] {\big(\dis_{T^{(R)}\vert_K,T^{(R)}_\varepsilon\vert_K}
				C^{(R)}_\varepsilon\big\vert_K\big)^{1/\gamma}}} {(\eta-2\E^{-K})^{1/\gamma}}
			+ \frac {\esp [2] {\upD\big(\pi^{(R)}_\varepsilon\big\vert_K ;
				\mu_{T^{(R)}}\big\vert_K, \mu_{T^{(R)}_\varepsilon}\big\vert_K\big)}} {\eta-2\E^{-K}}
			\right)\\
		& \qquad\qquad\qquad
		\leq \bigg( \frac {2^{1/\gamma} Kh} {(\eta-\E^{-K})^{1/\gamma}} + \frac K {\eta-\E^{-K}} \bigg)
			\int (\varepsilon\wedge\norm\bfs) \, I(\D\bfs)
		\xrightarrow [\varepsilon\to 0] {} 0.
	\end{align*}
\end{proof}

The next result is both intuitive and easy to prove.
Its proof will therefore be left to the reader.

\begin{lemma}\label{lem:finite-grafting-continuity}
	Fix $n$ a positive integer and let $\bfG_n$ be the restriction
	of $\bfG$ to $(\bbR_+\times\bbT_c)^n$.
	The application $\bfG_n$ is continuous for the product topology.
\end{lemma}

\begin{lemma}\label{lem:point-processes-large-immigration-weak-cv}
	Let $K\geq 0$ and $\varepsilon>0$ be fixed.
	Almost surely, for any continuous $F:\bbR_+\times\Sdec\times\bbT_c\to\bbR_+$
	bounded by $1$,
	\begin{align*}
		\limsup_{R\to\infty} \: {\textstyle\int} F(u,\bfs,\tau) \,
			\ind_{u\leq K, \, \norm\bfs\geq\varepsilon}\, \D\Pi_R(u,\bfs,\tau)
		& \leq {\textstyle\int} F(u,\bfs,\tau) \,
			\ind_{u\leq K, \, \norm\bfs\geq\varepsilon}\, \D\Pi(u,\bfs,\tau),\\
		\mathllap{\text{and}\qquad}
		\liminf_{R\to\infty} \: {\textstyle\int} F(u,\bfs,\tau) \,
			\ind_{u<K, \, \norm\bfs>\varepsilon}\, \D\Pi_R(u,\bfs,\tau)
		& \geq {\textstyle\int} F(u,\bfs,\tau) \,
			\ind_{u<K, \, \norm\bfs>\varepsilon}\, \D\Pi(u,\bfs,\tau).
	\end{align*}
\end{lemma}

\begin{proof}
	Let $\varphi$ and $\varphi_n$, $n\geq 1$
	be the applications from $\bbR_+\times\Sdec\times\bbT_c$ to $\bbR_+$
	defined for all $(u,\bfs,\tau)$ by $\varphi(u,\bfs,\tau) := \ind_{u\leq K} \, \ind_{\norm\bfs\geq\varepsilon}$
	and $\varphi_n(u,\bfs,\tau) := [1-n(u-K)_+]_+ \times [1-n(\varepsilon-\norm\bfs)_+]_+$ respectively
	(where $x_+ = x\vee 0$ for any real number $x$).
	Observe that for all $n\geq 1$, $\varphi_n$ is continuous and that for $n$ large enough,
	$\varepsilon\varphi_n \, F$ is an element of $\scrF$.
	Therefore, everywhere on the event $\{\Pi_R\to\Pi\}$,
	$\int \varphi_n \, F \,\D\Pi_R \to \int \varphi_n \, F \,\D\Pi$
	for any fixed $n\geq 1$.
	Furthermore, $\varphi_n\downarrow_n \varphi$ so the monotone convergence theorem yields
	$\inf_{n\geq 1} \int \varphi_n \, F \,\D\Pi = \int \varphi \, F \, \D\Pi$
	and for all $R\geq 1$, $\inf_{n\geq 1} \int \varphi_n \, F \,\D\Pi_R = \int \varphi \, F \, \D\Pi_R$.
	As a result, on $\{\Pi_R\to\Pi\}$,
	\[
		\limsup_{R\to\infty} \: {\textstyle\int} \varphi \, F \, \D\Pi_R
		\leq \inf_{n\geq 1} \bigg[ \limsup_{R\to\infty} \: {\textstyle\int} \varphi_n \, F \, \D\Pi_R\bigg]
		= {\textstyle\int} \varphi \, F \, \D\Pi.
	\]
	Similarly, if we let $\psi(u,\bfs,\tau) := \ind_{u<K} \, \ind_{\norm\bfs>\varepsilon}$,
	there exists a sequence $(\psi_n)_n$ of continuous applications
	such that $\psi_n\uparrow_n\psi$ and for $n$ large enough, $\varepsilon\psi_n\, F$ is in $\scrF$.
	The same kind of arguments lead~to
	\[
		\liminf_{R\to\infty} \: {\textstyle\int} \psi \, F \, \D\Pi_R
		\geq \sup_{n\geq 1} \bigg[ \liminf_{R\to\infty} \: {\textstyle\int} \psi_n \, F \, \D\Pi_R\bigg]
		= {\textstyle\int} \psi \, F \, \D\Pi
	\]
	everywhere on $\{\Pi_R\to\Pi\}$.
\end{proof}

\begin{lemma}\label{lem:infinite-mb-to-ssfi-pruning-cv}
	Let $\varepsilon$ be positive and such that $\Pi\big((u,\bfs,\tau):\norm\bfs=\varepsilon\big) = 0$ \as.
	Then $T^{\smash {(R)}}_\varepsilon$ \as converges to $\calT^{\smash {(I)}}_\varepsilon$ as $R\to\infty$.
\end{lemma}

\begin{proof}
	Observe that for any $K\geq 0$, $\Pi\big((u,\bfs,\tau):u=K\big) = 0$ \as
	which implies that with probability $1$, for any continuous bounded $F:\bbR_+\times\Sdec\times\bbT_c\to\bbR_+$,
	\[
		\textstyle
		\int F(u,\bfs,\tau) \,
			\ind_{u\leq K, \, \norm\bfs\geq\varepsilon}\, \D\Pi(u,\bfs,\tau)
		= \int F(u,\bfs,\tau) \,
			\ind_{u<K, \, \norm\bfs>\varepsilon}\, \D\Pi(u,\bfs,\tau).
	\]
	Consequently, in light of Lemma~\ref{lem:point-processes-large-immigration-weak-cv},
	\[
		\ind_{u\leq K, \, \norm\bfs\geq\varepsilon}\, \Pi_R(\D u,\D\bfs,\D\tau)
		\xRightarrow [R\to\infty] {\text {\as}}
		\ind_{u\leq K, \, \norm\bfs\geq\varepsilon}\, \Pi(\D u,\D\bfs,\D\tau).
	\]
	Furthermore, the measures $\ind_{u\leq K, \, \norm\bfs\geq\varepsilon}\, \Pi_R(\D u,\D\bfs,\D\tau)$,
	$R\geq 1$ and $\ind_{u\leq K, \, \norm\bfs\geq\varepsilon}\, \Pi(\D u,\D\bfs,\D\tau)$
	may be written as finite sums of Dirac measures.
	As a result, almost surely, the atoms of $\ind_{u\leq K, \, \norm\bfs\geq\varepsilon}\, \Pi_R(\D u,\D\bfs,\D\tau)$
	converge to those of $\ind_{u\leq K, \, \norm\bfs\geq\varepsilon}\, \Pi(\D u,\D\bfs,\D\tau)$
	when $R\to\infty$.
	Lemma~\ref{lem:finite-grafting-continuity} then ensures that $T^{\smash {(R)}}_\varepsilon\vert_K$
	\as converges to $\calT^{\smash {(I)}}_\varepsilon\vert_K$.
	Since this holds for any $K\geq 0$, Proposition~\ref{prop:extended-ghp-properties}
	allows us to conclude.
\end{proof}

\begin{proof}[of~Theorem~\ref{thm:scaling-limits-infinite-mb-trees}]
	Observe that the set of positive $\varepsilon$ such that
	$\smash {\prob [1] {\Pi\big((u,\bfs,\tau):\norm\bfs=\varepsilon\big)=0}} < 1$
	is at most countable.
	As a result, we may consider a sequence $(\varepsilon_k)_{k\geq 1}$
	of positive real numbers which converges to $0$ and such that for all $k$,
	$\Pi\big((u,\bfs,\tau):\norm\bfs=\varepsilon_k\big)=0$ \as.
	Lemmas~\ref{lem:ssfi-pruning-cv}, \ref{lem:infinite-mb-pruning-unif-cv} and~\ref{lem:infinite-mb-to-ssfi-pruning-cv}
	then respectively prove that conditions $(ii)$, $(iii)$ and $(i)$ of Theorem~\ref{thm:billingsley-weak-cv-criterion}
	are met for $T^{(R)}$, $T^{\smash {(R)}}_{\varepsilon_{\smash k}}$,
	$\calT^{\smash {(I)}}_{\varepsilon_{\smash k}}$, $R\geq 1$, $k\geq 1$ and $\calT^{(I)}$.
	Therefore, $T^{\smash {(R)}} \Rightarrow \calT^{\smash {(I)}}$ with respect to $\upD_\GHP$.
\end{proof}

\subsection{Volume growth of infinite Markov branching trees}\label{sec:volume-growth}

We now turn to the proof of Proposition~\ref{prop:volume-growth-cv}.
Recall that if $\bfT\in\bbT$ is fixed,
then $V_\bfT$, the volume growth function of $\bfT$,
is given by
\[
	V_\bfT:\bbR_+\longrightarrow\bbR_+, \;
	R\longmapsto\mu_T(T\vert_R).
\]
Notice that $V_\bfT$ is a non-negative, non-decreasing c\`adl\`ag function.

\begin{proof}[of~Proposition~\ref{prop:volume-growth-cv}]
	Proposition~\ref{prop:extended-ghp-properties} ensures that $(\bbT,\upD_\GHP)$ is a Polish metric space.
	In light of Skorokhod's representation theorem and
	since the assumptions of Theorem~\ref{thm:scaling-limits-infinite-mb-trees} are met,
	there exist a sequence $(\tau_R)_{R\geq 1}$ of $\MB^{q,q_{\smash\infty}}_\infty$ trees
	as well as a $(\gamma,\nu,I)$-fragmentation tree with immigration $\calT^{(I)}$
	such that $(R^{-1}\tau_R, R^{-1/\gamma}\mu_{\tau_{\smash R}}) =: T^{(R)}$ \as converges to $\calT^{(I)}$.
	
	Proposition~\ref{prop:extended-ghp-properties} and Remark~\ref{nb:height-mass-ghp-continuity} ensure
	that \as, for all $t\geq 0$ such that $\mu_{\calT^{\smash {(I)}}}[\partial_t\calT^{(I)}] = 0$,
	$V_{T^{\smash {(R)}}}(t)$ converges to $V_{\calT^{\smash {(I)}}}(t)$.
	Now observe that $\mu_{\calT^{\smash {(I)}}} [\partial_t\calT^{(I)}] = 0$
	\tiff $V_{\calT^{\smash {(I)}}}$ is continuous at $t$.
	Therefore, if we prove that $V_{\calT^{\smash {(I)}}}$ is \as continuous on $\bbR_+$,
	since volume growth functions are monotone, we may use the following classical result to conclude this proof:
	\begin{adjustwidth}{\parindent}{\parindent}
		\emph{
		If $(f_n)_n$ is a sequence of monotone functions from a compact interval $I$ to $\bbR$
		such that $f_n\to f$ point-wise for some continuous function $f$,
		then $f_n\to f$ uniformly on $I$.}
	\end{adjustwidth}
	
	\smallskip
	
	Following the construction of fragmentation trees with immigration detailed in Section~\ref{sec:frag-trees-immigration},
	there exist a Poisson point process $\Sigma = \sum_{i\geq 1} \delta_{(u_i,\bfs_i)}$
	on $\bbR_+\times\Sdec$ with intensity $\D u \otimes I(\D\bfs)$
	and a family $\big[\calT_{i,j} ; i,j\geq 1\big]$ of \iid $(\gamma,\nu)$-fragmentation trees independent of $\Sigma$
	such that
	\[
		\calT^{(I)} = \bfG \, \bigg( \Big\{ \Big(u_i,
			\big\langle \big(s_{i,j}^\gamma \calT_{i,j}, s_{i,j} \mu_{\calT_{i,j}}\big) ; j\geq 1\big\rangle
		\Big) \: : \: i\geq 1\Big\} \bigg).
	\]
	With these notations, we may write $V_{\calT^{\smash {(I)}}} = \sum_{i\geq 1} \sum_{j\geq 1}
	s_{i,j} V_{\calT_{\smash {i,j}}} \big[(\,\cdot\,-u_i)_+/s_{i,j}^{\smash\gamma}\big]$.
	Furthermore, for any non-negative $K$,
	since $V_{\calT_{\smash {i,j}}} \leq 1$ for all $i,j\geq 1$,
	\[
		\textstyle
		\sum_{i\geq 1} \sum_{j\geq 1} s_{i,j} \ind_{u_i\leq K}
		= \int \ind_{u\leq K} \, \norm\bfs \, \Sigma(\D u, \D\bfs)
	\]
	which is \as finite, as already noticed.
	As a result and in light of the Weierstrass $M$-test,
	the restriction of $V_{\calT^{\smash {(I)}}}$ to the compact interval $[0,K]$
	is a series which \as converges uniformly on $[0,K]$.
	
	Proposition~1.9 in~\cite{bertoin2006fragmentationcoagulation} implies that
	the volume growth function of $(\gamma,\nu)$-fragmentation trees is \as continuous.
	In particular, with probability one, $V_{\calT_{\smash {i,j}}}$ is continuous for all $i$ and $j$.
	As a uniformly converging series of continuous functions,
	$V_{\calT^{\smash {(I)}}}\vert_{[0,K]}$ is \as continuous on $[0,K]$.
	Since this holds for any $K\geq 0$, $V_{\calT^{\smash {(I)}}}$ is \as continuous on $\bbR_+$,
	which concludes this proof.
\end{proof}

\subsection{Unary immigration measures}\label{sec:unary-immigration}

Before concluding this section, we will state a useful criterion
to prove Assumption~\hyperref[assumption:immigration]{$(\mathtt I)$}
when the limit immigration measure is unary,
\ie supported by the set $\{(s,0,0,\dots): s\geq 0\}$.

\begin{lemma}\label{lem:unary-immigration-prelim}
	Let $X$ be an integer valued random variable
	such that there exist $\gamma\in(0,1)$ and a positive constant $c$
	satisfying $n^{1+\gamma} \prob {X=n} \to c$.
	In this case, for all continuous $f:\bbR_+\to\bbR_+$ with $f(x)\leq 1\wedge x$,
	$R\, \esp {f(X/R^{1/\gamma})} \to \int_0^\infty c \, f(x) \, x^{-1-\gamma} \, \D x$
	as $R$ goes to infinity.
\end{lemma}

\begin{proof}
	By assumption, for all $\varepsilon>0$,
	there exists an integer $N$ such that for all $n\geq N$,
	$\abs [1] {n^{1+\gamma} \prob {X=n} - c} < \varepsilon$.
	As a result
	\[
		R \sum_{n>N} (c-\varepsilon) \frac 1 {n^{1+\gamma}} f\bigg(\frac n {R^{1/\gamma}}\bigg)
		\leq R \, \esp [3] {f\bigg(\frac X {R^{1/\gamma}}\bigg)}
		\leq R \smash {\sum_{n=1}^N} \frac n {R^{1/\gamma}}
			+ R \sum_{n>N} (c+\varepsilon) \frac 1 {n^{1+\gamma}} f\bigg(\frac n {R^{1/\gamma}}\bigg).
	\]
	As a Riemann sum, $R \, \sum_{n>N} n^{-1-\gamma} f(n/R^{1/\gamma})$
	converges toward $\int_0^\infty f(x) \, x^{-1-\gamma} \, \D x$
	as $R$ goes to infinity.
	The desired result then follows.
\end{proof}

For each $\gamma\in(0,1)$, note $I^{\operatorname {un}}_{\smash\gamma}$
the measure defined by
\[
	\textstyle
	\int_{\Sdec} f \, \D I^{\operatorname {un}}_\gamma
	= \int_0^\infty f(x,0,0,\dots) \, x^{-1-\gamma} \, \D x
\]
for any measurable $f:\Sdec\to\bbR_+$.
We have $\smash\int 1\wedge\norm\bfs \, I^{\operatorname {un}}_{\smash\gamma}(\D\bfs) < \infty$
therefore $I^{\operatorname {un}}_{\smash\gamma}$ is an immigration measure.
Observe that $I_B = (2/\pi)^{1/2} I^{\operatorname {un}}_{1/2}$
where $I_B$ denotes the Brownian immigration measure from Section~\ref{sec:frag-trees-immigration}.

\begin{proposition}\label{prop:unary-immigration-criterion}
	Let $\Lambda$ be a random finite partition such that as $n\to\infty$,
	$n^{1+\gamma} \prob {\norm\Lambda=n} \to c$ for some $\gamma\in(0,1)$, $c>0$
	and $n^\gamma \prob {\Lambda_1\geq n}$ converges to $c/\gamma$.
	For all $R\geq 1$, let $q^{(R)}$ be the distribution of $\Lambda/R^{1/\gamma}$.
	Then, $R \, (1\wedge\norm\bfs) \, q^{(R)}(\D\bfs)$ converges weakly
	to $(1\wedge\norm\bfs) \, c \, I^{\operatorname {un}}_\gamma(\D\bfs)$ as $R\to\infty$
	in the sense of finite measures on $\Sdec$.
\end{proposition}

\begin{proof}
	The main idea for this proof is to show that the tail of $\Lambda$
	is asymptotically negligible when its first component is large, or more precisely,
	that $R \, \esp [1] {1\wedge \big([\norm\Lambda-\Lambda_1]\big/R^{1/\gamma}\big)}$
	converges to $0$ when $R$ goes to infinity.
	Since $\norm\Lambda$ fulfils the assumptions of Lemma~\ref{lem:unary-immigration-prelim},
	\[
		\textstyle
		R \, \esp [1] {1\wedge (\norm\Lambda/R^{1/\gamma})}
		\xrightarrow [R\to\infty] {}
		c\,\int 1\wedge\norm\bfs \, I^{\operatorname {un}}_{\smash\gamma}(\D\bfs)
		= c / [\gamma \,(1-\gamma)]
		=: C_\gamma
	\]
	Furthermore, $\Lambda_1\leq\norm\Lambda$, so we get that
	$\limsup_{R\to\infty} R \, \esp [1] {1\wedge (\Lambda_1/R^{1/\gamma})} \leq C_\gamma$.
	In light of Fatou's lemma and the assumption on the probability tail of $\Lambda_1$,
	\begin{align*}
		\liminf_{R\to\infty} \:
			R \, \esp [3] {1\wedge \frac {\Lambda_1} {R^{1/\gamma}}}
		& = \liminf_{R\to\infty} \:
			{\textstyle\int_0^1} R \, \prob [1] {\Lambda_1 \geq R^{1/\gamma} t} \, \D t
		\geq {\textstyle\int_0^1} c\gamma^{-1} t^{-\gamma} \, \D t
		= C_\gamma.
	\end{align*}
	In summary, when $R\to\infty$,
	$R \, \esp [1] {1\wedge (\Lambda_1/R^{1/\gamma})} \to C_\gamma$.
	
	\smallskip
	
	Now observe that if $a$, $b$, $x$ and $y$ are four real numbers,
	then $a\wedge x + b\wedge y \leq (a+b)\wedge(x+y)$.
	In particular, for all $\varepsilon\in(0,1)$,
	$1\wedge (\norm\Lambda/R^{1/\gamma}) \geq (1-\varepsilon)\wedge (\Lambda_1/R^{1/\gamma})
	+ \varepsilon\wedge\big([\norm\Lambda-\Lambda_1]\big/R^{1/\gamma}\big)$.
	Moreover,
	\begin{align*}
		\lim_{R\to\infty} \: R \, \esp [3] {(1-\varepsilon)\wedge
			\frac {\Lambda_1} {R^{1/\gamma}}}
		& = \lim_{R\to\infty} \: (1-\varepsilon) \, R \, \esp [3] {1\wedge
			\frac {\Lambda_1} {[(1-\varepsilon)^\gamma R]^{1/\gamma}}}\\
		& = (1-\varepsilon)^{1-\gamma} \bigg( \lim_{S\to\infty} \: S \, \esp [3] {1\wedge
			\frac {\Lambda_1} {S^{1/\gamma}}}\bigg)
		= (1-\varepsilon)^{1-\gamma} \, C_\gamma
	\end{align*}
	where we have taken $S = (1-\varepsilon)^\gamma \, R$.
	Similarly,
	\begin{align*}
		& \limsup_{R\to\infty} \: R \, \esp [3] {\varepsilon\wedge
			\frac {\norm\Lambda - \Lambda_1} {R^{1/\gamma}}}
		= \varepsilon^{1-\gamma} \: \bigg(\limsup_{S\to\infty} \: S \,
			\esp [3] {1 \wedge \frac {\norm\Lambda - \Lambda_1} {S^{1/\gamma}}}\bigg).
	\end{align*}
	Therefore,
	\begin{align*}
		\limsup_{R\to\infty} \: R \, \esp [3] {1\wedge
			\frac {\norm\Lambda - \Lambda_1} {R^{1/\gamma}}}
		& \leq \inf_{\varepsilon\in(0,1)}
			\frac {C_\gamma - (1-\varepsilon)^{1-\gamma} \, C_\gamma} {\varepsilon^{1-\gamma}}
		= 0.
	\end{align*}
	
	\smallskip
	
	Let $f:\Sdec\to\bbR_+$ be a Lipschitz-continuous function bounded by $1$
	and set $g(x) := f(x,0,0,\dots)$ for all $x\geq 0$.
	There exists a constant $K\geq 0$ such that for all $\bfx$ and $\bfy$ in $\Sdec$,
	$\abs {f(\bfx) - f(\bfy)} \leq 1 \wedge ( K\, \norm {\bfx-\bfy} )$.
	Therefore
	\[
		\abs [4] {
		R \, \esp [3] {\bigg(1\wedge \frac {\norm\Lambda} {R^{1/\gamma}}\bigg)
			\, f\bigg(\frac \Lambda {R^{1/\gamma}}\bigg)
		- \bigg(1\wedge \frac {\norm\Lambda} {R^{1/\gamma}}\bigg)
			\, g\bigg(\frac {\norm\Lambda} {R^{1/\gamma}}\bigg)} }
		\leq R \, \esp [3] { 1 \wedge \frac {2K\,(\norm\Lambda - \Lambda_1)} {R^{1/\gamma}}}
		\xrightarrow [R\to\infty] {} 0.
	\]
	Used conjointly with our assumption on $\norm\Lambda$ and Lemma~\ref{lem:unary-immigration-prelim},
	this ensures that $R\, \esp [1] {(1\wedge \norm\Lambda/R^{1/\gamma}) \, f(\Lambda)}$
	converges to $\int f(\bfs) \, I^{\operatorname {un}}_{\smash\gamma}(\D\bfs)$ as $R\to\infty$.
	Lemma~\ref{lem:portmanteau-extension} concludes this proof.
\end{proof}

\section{Applications}\label{sec:applications}

In this section, we will develop applications of our three main results
(Theorems \ref{thm:local-limits-markov-branching}, \ref{thm:scaling-limits-infinite-mb-trees}
and Proposition~\ref{prop:volume-growth-cv})
to various models of random trees which satisfy the Markov branching property.
With our unified approach, we will recover known results and get new ones.

\subsection{Galton-Watson trees}

Let $\xi$ be a probability measure on $\bbZ_+$ with mean $1$ and $\xi(1) < 1$ (\emph{critical regime}).
We will be interested in \emph{unordered} Galton-Watson trees with offspring ditribution $\xi$,
the law of which we will note $\GW_{\smash\xi}$.
For any finite tree $\ttt$,
\[
	\GW_{\smash\xi}(\ttt)
	:= \sum_{\ttt'\in\ttT^\ord \, : \, \ttt'\sim\ttt} \: \prod_{u\in\ttt'} \xi\big[c_u(\ttt')\big].
\]

\smallskip

For each positive integer $n$ such that $\GW_{\smash\xi}(\ttT_n) > 0$,
let $\GW_{\smash\xi}^n$ be the measure $\GW_{\smash\xi}$ conditioned on the set of trees with $n$ vertices.
Similarly, if $n$ satisfies $\GW_{\smash\xi}(\ttT_{\smash\calL,n}) > 0$,
define $\GW_{\smash\xi}^{\smash\calL,n}$ as $\GW_{\smash\xi}$ conditioned on the
set of trees with $n$ leaves.
Moreover, let $d := \gcd\,\{n-1;\GW_{\smash\xi}(\ttT_n) > 0\}$
and $d_\calL := \gcd\,\{n-1;\GW_{\smash\xi}(\ttT_{\calL,n}) > 0\}$.

\paragraph{Kesten's tree}
Let $\smash{\hat\xi}$ be the \emph{size-biased distribution} of $\xi$,
that is $\smash{\hat\xi}(k) = k\xi(k)$ for all $k\geq 0$.
By assumption, the mean of $\xi$ is $1$, so $\smash{\hat\xi}$ is a probability measure.
We define $\GW_{\smash\xi}^\infty$ as the distribution of \emph{Kesten's tree}
which is obtained as follows:
\begin{itemize}
	\item Let $(X_n)_{n\geq 0}$ be a sequence of \iid random variables
	such that  $X_n+1$ follows $\hat\xi$,
	\item Independently of this sequence, let
	$(T_{n,k} ; n\geq 0, k\geq 1)$ be \iid $\GW_{\smash\xi}$ trees,
	\item For each $n\geq 0$, let $T_n := \lBrack T_{n,1}, \dots, T_{n,\smash {X_n}} \rBrack$,
	\item For all $n\geq 0$, graft $T_n$ on an infinite branch at height $n$ respectively,
	\ie set $T := \ttb_\infty \bigotimes_{n\geq 0} (\ttv_n, T_n)$
	and denote its distribution by $\GW_{\smash\xi}^\infty$.
\end{itemize}

\begin{remark}
	These infinite trees were first indirectly introduced in~\cite{kesten1986subdiffusive} by Kesten
	who studied the genealogy of Galton-Watson processes conditioned to hit $0$ after a large time.
	This result entails that if $T$ is a $\GW_{\smash\xi}$ tree, conditionally on $\abs T \geq n$,
	$T$ converges in distribution to $\GW_{\smash\xi}^\infty$ as $n\to\infty$.
	Kesten's tree can thus be, in a way, considered as a $\GW_{\smash\xi}$
	tree conditioned to have infinite height.
	
	This tree also appears as the local limit of conditioned critical Galton-Watson trees
	under various types of conditionings, see~\cite{abraham2014gwkesten}.
	In particular, it was first proved in~\cite{kennedy1975gwprocess} (in terms of Galton-Watson processes)
	and in~\cite{aldous1998treevaluedmarkovchains} (in terms of trees)
	that if $\xi$ is critical and has finite variance, then $\GW_{\smash\xi}^n \Rightarrow \GW_{\smash\xi}^\infty$.
	In~\cite{curien2014noncrossingplane},
	it was shown that under the same assumptions, $\GW_{\smash\xi}^{\smash\calL,n} \Rightarrow \GW_{\smash\xi}^\infty$.
	In both cases, the finite variance assumption may be dropped,
	see~\cite{janson2012conditionedgw}
	and~\cite{abraham2014gwkesten}.
	
	The local limits of Galton-Watson trees conditioned on their size
	with offspring distribution with means less than $1$
	were studied in~\cite{jonsson2011condensation},~\cite{janson2012conditionedgw} and~\cite{abraham2014condensation}.
	See also~\cite{stephenson2014localmultitypegw} for the study of the local limits of multi-type
	critical Galton-Watson trees.
\end{remark}

\smallskip

Using Theorem~\ref{thm:local-limits-markov-branching},
we will recover the following proposition in Section~\ref{sec:gw-local-limits}.

\begin{proposition}\label{prop:gw-local-limit}
	In the sense of the $\upd_\loc$ topology,
	$\GW_{\smash\xi}^n$ and $\GW_{\smash\xi}^{\smash\calL,n}$
	both converge weakly towards~$\GW_{\smash\xi}^\infty$.
\end{proposition}

Afterwards, we will study scaling limits of Kesten's tree
in the spirit of Theorem~\ref{thm:scaling-limits-infinite-mb-trees}.
Recall the descriptions of the immigration Brownian tree
and $\alpha$-stable immigration L\'evy trees
from Section~\ref{sec:frag-trees-immigration}.

\begin{proposition}\label{prop:kesten-scaling-limit}
	Let $T$ be a tree with distribution $\GW_{\smash\xi}^\infty$
	and define $\mu_T := \sum_{u\in T} \delta_u$
	and $\mu_T^{\smash\calL} := \sum_{u\in \smash\calL(T)} \delta_u$
	the counting measures on the set of its vertices and leaves respectively.
	
	\smallskip
	
	\noindent$(i)$\;\;\emph{Finite variance:}\;\;
	Suppose $\xi$ has finite variance $\sigma^2$ and that $d=1$.
	Then, with respect to the $\upD_\GHP$ topology,
	\[
		\bigg( \frac T R, \frac {\mu_T} {R^2}\bigg)
		\xrightarrow [R\to\infty] {(\upd)}
		\bigg( \calT_B, \frac {\sigma^2} 4 \mu_B \bigg)
	\]
	where $(\calT_B,\mu_B)$ is the immigration Brownian tree.
	
	\smallskip
	
	\noindent$(i')$\;\;
	If $\xi$ has finite variance $\sigma^2$ and if $d_\calL = 1$,
	then
	\[
		\bigg( \frac T R, \frac {\mu_T^\calL} {R^2}\bigg)
		\xrightarrow [R\to\infty] {(\upd)}
		\bigg( \calT_B, \frac {\sigma^2 \, \xi(0)} 4 \mu_B \bigg).
	\]
	
	\smallskip
	
	\noindent$(ii)$\;\;\emph{Stable case:}\;\;
	Suppose that $\xi(n)\sim c \, n^{-1-\alpha}$ as $n\to\infty$
	for some positive constant $c$ and $\alpha\in(1,2)$.
	Then,
	\[
		\bigg( \frac T R, \frac {\mu_T} {R^{\alpha/(\alpha-1)}}\bigg)
		\xrightarrow [R\to\infty] {(\upd)}
		\big( \calT_\alpha, (ck_\alpha)^{1/(\alpha-1)} \mu_\alpha \big)
	\]
	where $(\calT_\alpha,\mu_\alpha)$ is the $\alpha$-stable immigration L\'evy tree
	and $k_\alpha = \Gamma(2-\alpha)/[\alpha\,(\alpha-1)]$.
\end{proposition}

\begin{remark}\label{nb:kesten-scaling-limits-leaves}
	Both $(i)$ and $(ii)$ were proved in~\cite{duquesne2009immigrationlevytrees}
	and $(i')$ seems to be a new, if predictable, result.
	
	We also mention that under the assumptions of $(ii)$,
	$(T/R, \mu_T^\calL/R^{\alpha/(\alpha-1)})$ should converge
	in distribution to $\big(\calT_\alpha, (ck_\alpha)^{1/(\alpha-1)}\,\xi(0) \mu_\alpha\big)$.
	We won't prove this statement as Assumption~\hyperref[assumption:scaling]{$(\mathtt S)$}
	hasn't been proved in this case and to do so would require quite a bit of computation.
	The scaling limits of Galton-Watson trees with such an offspring distribution
	conditioned on their number of leaves were however studied in~\cite{kortchemski2012invariance}.
\end{remark}

Section~\ref{sec:kesten-scaling-variance} will focus on the finite variance case,
first on $(i)$ and then on $(i')$.
We will prove Proposition~\ref{prop:kesten-scaling-limit}
in the stable case $(ii)$ in Section~\ref{sec:kesten-scaling-stable}.

\subsubsection{Markov branching property and local limits}\label{sec:gw-local-limits}
Let $\calN := \{n\geq 1 \,:\, \GW_{\smash\xi}(\ttT_n) > 0\}$.
Proposition~37 in~\cite{haas2012scaling} states that
the sequence of probability measures $(\GW_{\smash\xi}^n)_{n\in\calN}^{}$
satisfies the Markov branching property, \ie we have $\GW_{\smash\xi}^n = \MB^q_n$
for all adequate $n$ with $q_{n-1}$ defined for all $\lambda = (\lambda_1,\dots,\lambda_p)$
in $\calP_{n-1}$ by
\[
	q_{n-1}(\lambda)
	= \frac {p! \xi(p)} {\prod_{j\geq 1} m_j(\lambda)!}
	\frac {\prod_{i=1}^p \prob {\# T = \lambda_i}} {\prob {\# T = n}}
\]
where $T$ is a $\GW_{\smash\xi}$ tree.

Similarly, if we let $\calN_\calL := \{n\geq 1 \,:\, \GW_{\smash\xi}(\ttT_{\calL,n}) > 0\}$,
then in light of \cite[Lemma~8]{rizzolo2015scaling}, the family
$(\GW_{\smash\xi}^{\smash\calL,n})_{n\in\calN_{\smash\calL}}^{}$ of probability measures
satisfies the Markov branching property and the associated sequence~$q^{\smash\calL}$
of first-split distributions such that $\GW_{\smash\xi}^{\smash\calL,n} = \MB^{\smash{\calL,q^\calL}}_n$
is given for all $n$ in $\calN_\calL$ and $\lambda = (\lambda_1,\dots,\lambda_p)$ in $\calP_n$~by
\[
	q^\calL_n(\lambda)
	= \frac {p! \xi(p)} {\prod_{j\geq 1} m_j(\lambda)!}
	\frac {\prod_{i=1}^p \prob {\#_\calL T = \lambda_i}} {\prob {\# T_\calL = n}}
\]
where $T$ still denotes a $\GW_{\smash\xi}$ tree.

\smallskip

A Kesten tree with distribution $\GW_{\smash\xi}^\infty$ can be seen
as an infinite Markov branching tree with distribution $\MB^{q,q_{\smash\infty}}_\infty$
where $q_\infty$ is defined for any $\lambda = (\lambda_2,\dots,\lambda_p)$ in $\calP_{<\infty}$ by
\[
	q_\infty(\infty, \lambda)
	:= \hat\xi(p) \, \frac {(p-1)!} {\prod_{j\geq 1} m_j(\lambda)!} \,
	{\textstyle\prod_{i=2}^p} \prob {\# T = \lambda_i}.
\]
The distribution of Kesten's tree may also
be rewritten as $\GW_{\smash\xi}^\infty = \MB^{\smash{\calL,q^\calL,q^\calL_\infty}}_\infty$
where $q^{\smash\calL}_\infty$ is given for all $\lambda\in\calP_{<\infty}$ by
\[
	q^\calL_\infty (\infty, \lambda)
	= \hat\xi(p) \, \frac {(p-1)!} {\prod_{j\geq 1} m_j(\lambda)!} \,
	{\textstyle\prod_{i=2}^p} \prob {\#_\calL T = \lambda_i}.
\]

\smallskip

Proposition~\ref{prop:gw-local-limit} is a direct consequence
of the following results from Sections~4.3 and~4.4 in~\cite{abraham2014gwkesten}
used alongside Theorem~\ref{thm:local-limits-markov-branching}.

\begin{lemma}\label{lem:gw-prob-size-ratio}
	If $T$ is a $\GW_{\smash\xi}$ tree, then
	\[
		\frac {\prob {\# T = (n+1)d +1}} {\prob {\# T = nd +1}}
		\xrightarrow [n\to\infty] {} 1
		\quad\text{and}\quad
		\frac {\prob {\#_\calL T = (n+1)d_\calL +1}} {\prob {\#_\calL T = nd_\calL +1}}
		\xrightarrow [n\to\infty] {} 1.
	\]
\end{lemma}

\begin{proof}[of~Proposition~\ref{prop:gw-local-limit}]
	Let $\lambda = (\lambda_2,\dots,\lambda_p)$ be an element of $\calP_{<\infty}$.
	If there exists $2\leq i\leq p$ such that $\lambda_i-1$ isn't divisible by $d$,
	then for all $n\in\calN$, $q_{n-1}(n-1-\norm\lambda,\lambda) = 0 = q_\infty(\infty,\lambda)$.
	Otherwise, for $n\in\calN$ large enough, in light of Lemma~\ref{lem:gw-prob-size-ratio}
	\begin{align*}
		q_{n-1}\big(n-1-\norm\lambda,\lambda\big)
		& = \frac {p! \xi(p)} {\prod_{j\geq 1} m_j(\lambda)!}
			\frac {\prob {\# T = n-\norm\lambda}} {\prob {\# T = n}}
			\prod_{i=1}^p \prob {\# T = \lambda_i}\\
		& \xrightarrow [n\to\infty] {}
			\hat\xi(p) \, \frac {(p-1)!} {\prod_{j\geq 1} m_j(\lambda)!} \,
			\prod_{i=2}^p \prob {\# T = \lambda_i}
		= q_\infty(\infty,\lambda).
	\end{align*}
	Similarly, as $n$ goes to infinity,
	$q^\calL_n(n-\norm\lambda,\lambda) \to q^\calL_\infty(\infty,\lambda)$.
	Since these hold for any $\lambda$ in $\calP_{<\infty}$,
	we end this proof by using Corollary~\ref{cor:local-limits-mb-one-spine}.
\end{proof}

\subsubsection{Scaling limits, finite variance}\label{sec:kesten-scaling-variance}
In the remainder of this section,
$(T_i)_{i\geq 1}$ will denote \iid Galton-Watson trees with offspring distribution $\xi$,
$(Y_n)_{n\geq 1}$, \iid $\xi$ distributed random variables
and for all $n\geq 1$, $S_n := Y_1 + \dots + Y_n - n$.
We will also consider $N$, a random variable independent of both $(T_i)_i$ and $(Y_n)_n$
and such that $N+1$ follows $\hat\xi$.

The following so called \emph{Otter-Dwass' formula}
or \emph{cyclic lemma} (see \cite[Chapter~6]{pitman2006combinatorial} for instance)
will be the cornerstone of many forthcoming computations.

\begin{lemma}[Otter-Dwass' formula]\label{lem:otter-dwass-formula}
	With these notations, for all $k\geq 1$ and $n\geq 1$,
	\[
		\prob [1] {\#T_1 + \dots + \#T_k = n}
		= \frac k n \prob [1] {S_n = -k}.
	\]
\end{lemma}

Let $q_*$ be the probability distribution on $\calP_{<\infty}$
defined by $q_* = q_\infty(\infty,\,\cdot\,)$.
Let $\Lambda$ follow $q_*$ and recall that it has
the same distribution as $(\#T_1, \dots, \#T_N)^{\smash\downarrow}$.

\smallskip

In this paragraph, we'll assume that the variance $\sigma^2$ of $\xi$ is finite
and that $d=1$.
Recall that the immigration Brownian tree is a $(1/2,\nu_B,I_B)$-fragmentation tree with immigration.
It was proved in \cite[Section~5.1]{haas2012scaling}
that Assumption~\hyperref[assumption:scaling]{$(\mathtt S)$} of Theorem~\ref{thm:scaling-limits-infinite-mb-trees}
is fulfilled for $\gamma=1/2$ and $\nu = \sigma/2 \cdot \nu_B$.
To prove Proposition~\ref{prop:kesten-scaling-limit}, it will therefore be sufficient to
show that Assumption~\hyperref[assumption:immigration]{$(\mathtt I)$} is satisfied
for $\gamma = 1/2$ and $I = \sigma/2 \cdot I_B$.
For all $R\geq 1$, note $q^{(R)}$ the distribution of $\Lambda/R^2$.

\begin{proposition}\label{prop:kesten-finite-variance-immmigration}
	In the sense of weak convergence of finite measures on $\Sdec$,
	$R \, ( 1\wedge\norm\bfs ) \, q^{(R)}(\D\bfs)$ converges as $R$ goes to infinity
	toward $( 1\wedge\norm\bfs ) \, \sigma/2 \cdot I_B(\D\bfs)$.
\end{proposition}

Since $I_B$ is unary, in order to prove Proposition~\ref{prop:kesten-finite-variance-immmigration},
it will be enough to show that $\Lambda$ satisfies
the assumptions of Proposition~\ref{prop:unary-immigration-criterion}.
The next two lemmas will prove that both are met.

\begin{lemma}\label{lem:kesten-finite-variance-norm}
	When $n$ goes to infinity,
	$n^{3/2} \prob {\norm\Lambda = n} \to (\sigma^2/2\pi)^{1/2}$.
\end{lemma}

\begin{proof}
	In light of Otter-Dwass' formula, for all $n\geq 1$,
	\begin{align*}
		n^{3/2} \prob [1] {\norm\Lambda=n}
		& = n^{3/2} {\textstyle\sum_{k\geq 1}} \prob [1] {\# T_1 + \dots + \# T_k = n \,\vert\, N=k}
			\, \prob [1] {N=k}\\
		& = {\textstyle\sum_{k\geq 1}} k \hat\xi(k) \, n^{1/2} \, \prob [1] {S_n=-k}.
	\end{align*}
	Recall the local limit theorem in the finite variance case:
	\[
		\sup\nolimits _{k\in\bbZ} \abs [1] {n^{1/2} \, \prob {S_n=k} - (2\pi\sigma^2)^{-1/2} \, \E^{-k^2/2n\sigma^2}}
		\xrightarrow [n\to\infty] {} 0.
	\]
	As a result, there exists a finite constant $C$ such that $n^{1/2} \prob {S_n=-k} \leq C$
	for all $n\geq 1$ and $k\geq 1$ and if $k\geq 1$ is fixed, $n^{1/2} \prob {S_n=-k} \to (2\pi\sigma^2)^{-1/2}$.
	Furthermore, $\sum_{k\geq 1} k \smash{\hat\xi(k)} = \sigma^2$
	so Lebesgue's dominated convergence theorem yields
	\[
		\lim_{n\to\infty} n^{3/2} \prob {\norm\Lambda=n}
		= \sum_{k\geq 1} k \hat\xi(k) \Big( \lim\nolimits_{n\to\infty} n^{1/2} \prob{S_n=-k} \Big)
		= \big(\sigma^2/2\pi\big)^{1/2}.
	\]
\end{proof}

\begin{lemma}\label{lem:kesten-finite-variance-1st-component}
	When $n\to\infty$, $n^{1/2} \prob {\Lambda_1\geq n}$
	converges to $(2\sigma^2/\pi)^{1/2}$.
\end{lemma}

\begin{proof}
	Observe that for all $n\geq 0$, the event $\{\Lambda_1\geq n\}$ has the same probability
	as $\{N\geq 1, \exists i\leq N : \#T_i \geq n\}$.
	Therefore $\prob {\Lambda_1 \geq n} = \sum_{k\geq 1} \hat\xi(k+1) \,
	\big( 1 - \prob {\#T_1 < n}^k \big)$.
	Let $G$ be the moment generating function of $\xi$,
	\ie $G(s) = \sum_{k\geq 0} \xi(k) \, s^k$ for all $s\in [0,1]$.
	This function is twice-differentiable on $[0,1]$
	and we may write $\prob {\Lambda_1 \geq n}
	= G'(1) - G'\big(1 - \prob {\# T_1 \geq n}\big)$.
	
	For all $n\geq 1$, Otter-Dwass' formula gives
	$n^{1/2} \prob {\# T_1\geq n} = n^{1/2} \sum_{m\geq n} m^{-1} \, \prob {S_m=-1}$.
	The local limit theorem ensures that $m^{1/2} \prob {S_m=-1} \to (2\pi\sigma^2)^{-1/2}$ as $m\to\infty$.
	Therefore, for all positive $\varepsilon$ and $n$ large enough,
	\[
		n^{1/2} \, \abs [2] {\prob {\#T\geq n}
			-{\textstyle\sum}_{m\geq n} m^{-3/2} (2\pi\sigma^2)^{-1/2}}
		\leq n^{1/2} \, \sum_{m\geq n} m^{-3/2} \, \varepsilon
			\xrightarrow [n\to\infty] {} 2\varepsilon.
	\]
	Incidentally, $n^{1/2} \prob {\# T_1\geq n}$ and $n^{1/2} \sum_{m\geq n} m^{-3/2} (2\pi\sigma^2)^{-1/2}$
	have the same limit when $n\to\infty$ which is to say that
	$n^{1/2} \prob {\# T_1\geq n} \to (2/\pi\sigma^2)^{1/2}$ as $n\to\infty$.
	As a result,
	\[
		n^{1/2} \prob {\Lambda_1 \geq n}
		= n^{1/2} \Big[ G'(1) - G'\big(1 - \prob {\# T_1 \geq n}\big) \Big]
		\xrightarrow [n\to\infty] {}
		\bigg(\frac 2 {\pi\sigma^2}\bigg)^{1/2} \, G''(1)
		= \bigg(\frac {2 \sigma^2} \pi\bigg)^{1/2}.
	\]
\end{proof}

Lemmas~\ref{lem:kesten-finite-variance-norm} and~\ref{lem:kesten-finite-variance-1st-component}
and Proposition~\ref{prop:unary-immigration-criterion} prove Proposition~\ref{prop:kesten-finite-variance-immmigration}.
Theorem~\ref{thm:scaling-limits-infinite-mb-trees} therefore implies that $(T/R,\mu_T/R^2)$
converges in distribution to a $(1/2, \sigma/2\cdot\nu_B, \sigma/2\cdot I_B)$ fragmentation tree with immigration.
Using Remark~\ref{nb:ssfi-rescaling}, we may restate this last result as Proposition~\ref{prop:kesten-scaling-limit}~$(i)$.
Furthermore, as a result of Proposition~\ref{prop:volume-growth-cv}, we get that in particular,
$\mu_T(T\vert_R)/R^2$ converges in distribution to $(\sigma^2/4) \, \mu_{\calT_B}(\calT_B\vert_1)$
or equivalently to $\mu_{\calT_B}(\calT_B\vert_{\sigma/2})$.

\medskip

We will now prove Proposition~\ref{prop:kesten-scaling-limit} $(i')$.
Assume that $d_\calL=1$.
Theorem~7 in \cite{rizzolo2015scaling} proves that the family $(q^{\smash\calL}_n)_n$
of first split distributions associated to Galton-Watson trees conditioned
on their number of leaves satisfies Assumption~\hyperref[assumption:scaling]{$(\mathtt S)$}:
$n^{1/2} \, (1-s_1) \, \bar q^{\smash\calL}_n \Rightarrow
\sigma \, \xi(0)^{1/2} / 2  \cdot (1-s_1) \, \nu_B(\D\bfs)$.
As a result, we only need to prove Assumption~\hyperref[assumption:immigration]{$(\mathtt I)$}
for $\gamma=1/2$ and $I = \sigma \, \xi(0)^{1/2} / 2  \cdot I_B$.

\begin{proof}[of Proposition~\ref{prop:kesten-scaling-limit} $(i')$]
	Theorem~6 in~\cite{rizzolo2015scaling} states that there exists
	a critical probability distribution $\zeta$ on $\bbZ_+$
	such that $\#_\calL T_1$, the number of leaves of $T_1$, has the same distribution as $\#\tau$,
	where $\tau$ follows $\GW_{\smash\zeta}$.
	Lemma~6 further states that if $\xi$ has finite variance $\sigma^2$,
	then $\zeta$ has variance $\sigma^2/\xi(0)$.

	Let $\Lambda^{\smash\calL}$ be such that $(\infty,\Lambda)$
	is distributed according to $q_\infty^{\smash\calL}$.
	The random partition $\Lambda^{\smash\calL}$ is distributed
	like $(\#_\calL T_1, \dots, \#_\calL T_N)^{\smash\downarrow}$,
	or equivalently, like $(\#\tau_1, \dots, \#\tau_N)^{\smash\downarrow}$,
	where $(\tau_n)_{n\geq 1}$ are \iid $\GW_{\smash\zeta}$ trees independent of $N$.
	Therefore, if $(V_n)_{n\geq 1}$ is a sequence of \iid $\zeta$-distributed random variables
	and if $Z_n := V_1 + \dots + V_n - n$,
	proceeding as in the proof of Lemma~\ref{lem:kesten-finite-variance-norm} gives:
	\[
		n^{3/2} \prob {\norm {\Lambda^\calL} = n}
		= {\textstyle\sum_{k\geq 0}} k \, \hat\xi(k+1)
			n^{1/2} \prob {Z_n = -k}
		\xrightarrow [n\to\infty] {}
		\big[\sigma^2 \xi(0) / (2\pi) \big]^{1/2}.
	\]
	Similarly, the same kind of computations as in Lemma~\ref{lem:kesten-finite-variance-1st-component} yields
	\[
		n^{1/2} \prob {\Lambda^\calL_1 \geq n}
		= n^{1/2} \Big[ G'(1) - G'\big(1 - \prob {\#\tau_1 \geq n}\big) \Big]
		\xrightarrow [n\to\infty] {}
		\big[ 2\sigma^2 \xi(0) / \pi \big]^{1/2}
	\]
	where $G$ still denotes the moment generating function of $\xi$.
	As a result, because of Theorem~\ref{thm:scaling-limits-infinite-mb-trees}
	and Proposition~\ref{prop:unary-immigration-criterion},
	when $R\to\infty$, $(T/R, \mu^{\smash\calL}_T/R^2)$ converges in distribution
	to a $(1/2, \sigma \xi(0)^{1/2} / 2 \cdot \nu_B, \sigma \xi(0)^{1/2} / 2 \cdot I_B)$ fragmentaion
	tree with immigration.
	Remark~\ref{nb:ssfi-rescaling} then allows us to conlude.
\end{proof}

\subsubsection{Scaling limits, stable case}\label{sec:kesten-scaling-stable}
In this paragraph, we'll suppose that there exist $\alpha\in(1,2)$ and a positive constant $c$
such that $n^{1+\alpha} \xi(n) \to c$ when $n\to\infty$.

Recall that $\Lambda$ denotes a $q_*$-distributed variable
and has the same distribution as $(\#T_1,\dots,\#T_N)^{\smash\downarrow}$
where $N+1$ is distributed according to $\smash{\hat\xi}$ and
is independent of the sequence $(T_n)_{n\geq 1}$ of \iid $\GW_{\smash\xi}$ trees.
Moreover, we will use the notations introduced to define $\nu_\alpha$ and $I^{(\alpha)}$
in Sections~\ref{sec:ssf-trees} and~\ref{sec:frag-trees-immigration}:
$(\Sigma_t;t\geq 0)$ will denote a $1/\alpha$-stable subordinator with Laplace exponent
$\lambda\mapsto-\log\esp{\exp(-\lambda\Sigma_t)}=\lambda^{1/\alpha}$
and $\Delta$ will be the decreasing rearrangement of its jumps on $[0,1]$.

\smallskip

It was proved in \cite[Section~5.2]{haas2012scaling}
that the family $q = (q_n)_{n\in\calN}$
of first-split distributions associated to $(\GW_{\smash\xi}^n)_{n\in\calN}^{}$ satisfies
Assumption~\hyperref[assumption:scaling]{$(\mathtt S)$} of Theorem~\ref{thm:scaling-limits-infinite-mb-trees}
for $\gamma=1-1/\alpha$ and $\nu = (c \, k_\alpha)^{1/\alpha} \cdot \nu_\alpha$.
Proposition~\ref{prop:kesten-scaling-limit}~$(ii)$ will therefore be a consequence of the next proposition.
For all $R\geq 1$, note $q^{(R)}$ the distribution of $R^{-\alpha/(\alpha-1)} \Lambda$.

\begin{proposition}\label{prop:kesten-stable-immigration}
	When $R\to\infty$, $R \, (1\wedge\norm\bfs) \, q^{(R)}(\D\bfs)$
	converges weakly to $(c\,k_\alpha)^{1/\alpha} (1\wedge\norm\bfs) \, I^{(\alpha)}(\D\bfs)$.
\end{proposition}

\begin{proof}
	As shown in \cite[Section~5.2]{haas2012scaling},
	$n^{1+1/\alpha} \, \prob {\#T_1 = n}$ converges to $[(c\, k_\alpha)^{1/\alpha}
	\alpha \, \Gamma(1-1/\alpha)]^{-1}$.
	Therefore, $(\#T_n)_{n\geq 1}$ lies in the domain of attraction of a $1/\alpha$-stable distribution.
	More accurately, in the Skorokhod topology,
	\[
		\Bigg( \frac {\#T_1+\dots+\#T_{\lfloor nt\rfloor}} {n^\alpha}
			\, ; \, t\geq 0 \Bigg)
		\xrightarrow [n\to\infty] {(\upd)}
		\frac 1 {c \, k_\alpha}
			\Big(\Sigma_t \, ; \, t\geq 0\Big).
	\]
	This, in conjunction with Skorokhod's representation theorem,
	implies that there exists a sequence $(X_n)_{n\geq 0}$, where  for all $n\geq 1$,
	\[
		X_n \overset {(\upd)} =
		\frac {c \, k_\alpha} {n^\alpha}
			\big(\#T_1,\dots,\#T_n,0,0,\dots\big)^\downarrow
	\]
	which \as converges to (a version of) $\Delta$.

	\smallskip

	Let $F:\Sdec\to\bbR_+$ be a Lipschitz continuous function such that $F(\bfs) \leq 1\wedge\norm\bfs$
	and set $f:\bbR_+\to\bbR_+$, $t\mapsto \esp {F(t^\alpha/(c\,k_\alpha)\cdot\Delta)}$.
	The dominated convergence theorem ensures that the function $f$ is continuous.
	It is clearly bounded by $1$ and
	\[
		f(t)
		\leq \esp [2] {1 \wedge \big(t^\alpha/ (c\,k_\alpha) \cdot \norm\Delta\big)}
		= \esp [2] {1 \wedge \Sigma_{(c\,k_\alpha)^{-1/\alpha} t}}
		\leq \frac t {(c\,k_\alpha)^{1/\alpha}} \,
			\smash {\int_{\mathrlap {\bbR_+}}} (1\wedge x) \, \Pi_{1/\alpha}(\D x).
	\]
	Since $n^\alpha \prob {N=n} \to c$, Lemma~\ref{lem:unary-immigration-prelim}
	ensures that when $R$ goes to infinity, $R\,\esp [1] {f(N/R^{1/(\alpha-1)})}$
	converges to $c \, \int_0^\infty t^{-\alpha} \, f(t) \, \D t
	= (c\,k_\alpha)^{1/\alpha} \, \int F\,\D I^{(\alpha)}$.
	Furthermore, because $\Lambda$ is distributed like $(c\,k_\alpha)^{-1} \, N^\alpha \, X_N$,
	\[
		\abs [4] {R \, \esp [3] {
			F\bigg(\frac \Lambda {R^{\alpha/(\alpha-1)}}\bigg)
			- f\bigg(\frac N {R^{1/(\alpha-1)}}\bigg)}}
		\leq R \, \esp [4] {1 \wedge \Bigg(K\, \bigg(\frac N {R^{1/(\alpha-1)}}\bigg)^\alpha \, 
			\norm [1] {X_N - \Delta}\Bigg)}
	\]
	where $K\cdot(c\,k_\alpha)$ is bigger than the Lipschitz constant of $F$.
	We will now endeavour to prove that this last quantity goes to $0$ when $R\to\infty$.
	For all $\bfs$ in $\Sdec$, let $\bfs\wedge 1$ be the sequence $(s_i\wedge 1)_{i\geq 1}$.
	Then for any $\bfx$ and $\bfy$ in $\Sdec$, we may write $\norm {\bfx-\bfy}
	= \norm {\bfx\wedge 1 - \bfy\wedge 1} + \norm {(\bfx-\bfx\wedge 1) - (\bfy-\bfy\wedge 1)}$.

	\medskip

	In light of Lemma~\ref{lem:unary-immigration-prelim},
	$n \, \esp [1] {1\wedge (\#T_1/n^\alpha)}$
	converges to $[(c\,k_\alpha)^{1/\alpha}\,\Gamma(2-1/\alpha)]^{-1}$.
	It ensues from the \iid nature of the sequence $(\#T_i)_{i\geq 1}$ that
	\[
		\sup_{n\geq 1} \, \esp [1] {\norm {X_n\wedge 1}^2}
		= \sup_{n\geq 1} \, \Bigg( n \, \esp [4] {\bigg(\frac {\#T_1} {n^\alpha}\wedge 1\bigg)^2}
		+ n(n-1) \, \esp [3] {\frac {\#T_1} {n^\alpha}\wedge 1}^2\Bigg)
		< \infty.
	\]
	Fatou's lemma (or classical results on Poisson Point Process,
	see \cite[Section~3.2]{kingman1992poissonprocesses})
	ensures that $\esp {\norm {\Delta\wedge 1}^2}$ is also finite.
	As a result, the sequence $(\norm {X_n\wedge 1 - \Delta\wedge 1})_{n\geq 1}$ is bounded in $L^2$.
	Since $\norm {X_n\wedge 1 - \Delta\wedge 1} \to 0$ \as,
	we also have $\esp [1] {\norm {X_n\wedge 1 - \Delta\wedge 1}} \to 0$.
	
	\smallskip
	
	If $\beta<1/\alpha$, then $\esp {\norm {\Delta - \Delta\wedge 1}^\beta}
	\leq \esp {\norm\Delta^\beta} = \esp {\Sigma_1^\beta} < \infty$.
	Moreover, since it converges, the sequence $\big(m^{1+1/\alpha} \prob {\#T_1 = m} \big)_m$
	is bounded by a finite constant, say $Q$.
	Consequently,
	\[
		\esp [1] {\norm {X_n - X_n\wedge 1}^\beta}
		= n \, \esp [4] {\bigg(\frac {\#T_1} {n^\alpha} - 1\bigg)_+^\beta}
		\leq Q \, n \!\! \sum_{k>n^\alpha}
			\frac {k^\beta} {n^{\alpha\beta}} \, \frac 1 {k^{1+1/\alpha}}
		\xrightarrow [n\to\infty] {}
		Q \int_1^{\mathrlap\infty} \, \frac {\D t} {t^{1+1/\alpha-\beta}}
		= \frac {\alpha \, Q} {1-\alpha\beta}
	\]
	which proves that the sequence $\big( \esp {\norm {X_n - X_n\wedge 1}^\beta} \big)_{n\geq 1}$ is bounded.
	Since this holds for all $\beta<1/\alpha$, if $\varepsilon$ is positive
	and such that $(1+\varepsilon)\beta =: \beta' < 1/\alpha$, then
	\[
		\sup_{n\geq 1} \, \esp [2] {\big(\norm {(X_n - X_n\wedge 1)
			- (\Delta - \Delta\wedge 1)}^\beta\big)^{1+\varepsilon}}
		\leq \sup_{n\geq 1} \, \esp [2] {\norm {X_n - X_n\wedge 1}^{\beta'}
			+ \norm {\Delta - \Delta\wedge 1}^{\beta'}}
		< \infty.
	\]
	Hence, the sequence
	$\big( \norm {(X_n - X_n\wedge 1) - (\Delta - \Delta\wedge 1)}^\beta \big)_{n\geq 1}$
	is bounded in $L^{1+\varepsilon}$.
	Because it converges to $0$ almost surely,
	its expectancy also goes to $0$ as $n$ tends to infinity.

	\smallskip

	For all $\beta<1/\alpha$ and $\varepsilon>0$,
	there exist a finite constant $C$ and a finite integer $n_\varepsilon$
	such that for all~$n\geq 1$
	\[
		\esp [2] {\norm {X_n\wedge1 - \Delta\wedge 1}}
		\, \vee \,
		\esp [2] {\norm {(X_n-X_n\wedge 1) - (\Delta-\Delta\wedge 1)}^\beta}
			\leq \varepsilon + C \ind_{n< n_\varepsilon}.
	\]
	Using the same arguments as in the proof of Lemma~\ref{lem:unary-immigration-prelim}
	it is easy to prove that for any $\kappa>\alpha-1$,
	\[
		R\, \esp [1] {1\wedge(N/R^{1/(\alpha-1)})^\kappa}
		\xrightarrow [R\to\infty] {}
		c \int_0^\infty \frac {1\wedge t^\kappa} {t^\alpha} \, \D t
		= \frac c {\kappa - (\alpha-1)} + \frac c {\alpha-1}.
	\]
	Consequently, if $\beta\in(1-1/\alpha,1/\alpha)$, we get
	{\allowdisplaybreaks
	\begin{align*}
		& \limsup_{R\to\infty} \:
		R \, \esp [4] {1 \wedge \Bigg( K
			\bigg(\frac N {R^{1/(\alpha-1)}}\bigg)^\alpha \, 
			\norm [1] {X_N - \Delta}\Bigg)}\\
		& \qquad \leq
		\begin{aligned}[t]
			\limsup_{R\to\infty} \;\;
			& R \, \esp [4] {1 \wedge \Bigg( K
				\bigg(\frac N {R^{1/(\alpha-1)}}\bigg)^\alpha \, 
				\esp [2] {\norm [1] {X_N\wedge 1 - \Delta\wedge 1} \bigm\vert \, N}\Bigg)}\\
			{}+{} & R \, \esp [4] {1 \wedge \Bigg( K^\beta
				\bigg(\frac N {R^{1/(\alpha-1)}}\bigg)^{\alpha\beta} \, 
				\esp [2] {\norm [1] {(X_N-X_N\wedge 1) - (\Delta-\Delta\wedge 1)}^\beta \bigm\vert \,N}\Bigg)}
		\end{aligned}\\
		& \qquad \leq
		\limsup_{R\to\infty} \;\;
		R \, \esp [4] {1 \wedge \Bigg( K \,
			\frac {N^\alpha} {R^{\alpha/(\alpha-1)}}\, 
			\Big(\varepsilon + C \ind_{N<n_\varepsilon}\Big)\Bigg)}
		+ R \, \esp [4] {1 \wedge \Bigg( K^\beta
			\frac {N^{\alpha\beta}} {R^{\alpha\beta/(\alpha-1)}}\, 
			\Big(\varepsilon + C \ind_{N<n_\varepsilon}\Big)\Bigg)}\\
		& \qquad \leq
		\limsup_{R\to\infty} \;\;
			\frac {K\,C\,n_\varepsilon^\alpha} {R^{\alpha/(\alpha-1)-1}}
			+ \frac {K^\beta C \, n_\varepsilon^{\alpha\beta}} {R^{\alpha\beta/(\alpha-1)-1}}
		\begin{aligned}[t]
			{}+{} & K^{\alpha/(\alpha-1)} \varepsilon^{\alpha/(\alpha-1)} \, R \,
				\esp [4] {1 \wedge \frac {N^\alpha} {R^{\alpha/(\alpha-1)}}}\\
			{}+{} & K^{\alpha/(\alpha-1)} \varepsilon^{[\alpha/(\alpha-1)]/\beta} \, R \,
				\esp [4] {1 \wedge \frac {N^{\alpha\beta}} {R^{\alpha\beta/(\alpha-1)}}}
		\end{aligned}\\
		& \qquad = O(\varepsilon^{\alpha/(\alpha-1)}).
	\end{align*}}
	Since this holds for any positive $\varepsilon$,
	it follows that
	\[
		R \, \esp [4] {1 \wedge \Bigg( K\, \bigg(\frac N {R^{1/(\alpha-1)}}\bigg)^\alpha \, 
			\norm [1] {X_N - \Delta}\Bigg)}
		\xrightarrow [R\to\infty] {} 0,
	\]
	which in turn proves that $R \, \esp [1] {F(\Lambda / R^{\alpha/(\alpha-1)})}$
	indeed converges to $(c\,k_\alpha)^{1/\alpha} \int_{\Sdec} F \, \D I^{(\alpha)}$.
	We conclude with Lemma~\ref{lem:portmanteau-extension}.
\end{proof}

\subsection{Cut-trees}

Let $\tau$ be a finite labelled tree.
If $\tau$ is made out of a single vertex, let its cut-tree $\operatorname {Cut}\,(\tau)$
be the tree with a single vertex.
Otherwise, define  the cut-tree of $\tau$
as the (unordered) binary tree $\operatorname {Cut}\,(\tau)$
obtained by the following recursive process:
\begin{itemize}
	\item Pick $a\to b$ uniformly at random among the edges of $\tau$
	and remove that edge,
	\item Let $\tau_1$ and $\tau_2$ be the two sub-trees of $\tau$ formerly connected by $a\to b$,
	\item Define the cut-tree of $\tau$ as the concatenation
	of the cut-trees of $\tau_1$ and $\tau_2$,
	\ie set $\operatorname {Cut}\,(\tau) :=
	\lBrack \operatorname {Cut}\,(\tau_1), \operatorname {Cut}\,(\tau_2) \rBrack$.
\end{itemize}
With this definition, if $\tau$ has $n$ vertices,
then $\operatorname {Cut}\,(\tau)$ has $n$ leaves.
The cut-tree of $\tau$ represents the genealogy of its dismantling
when we remove edge after edge, until all have been deleted.
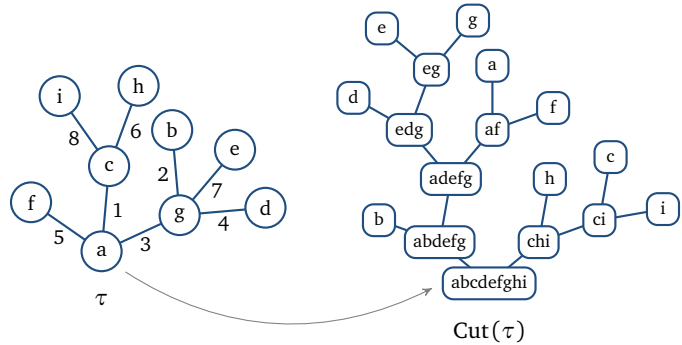
\begin{figure}[ht]
	\captionsetup{justification=centering}
	\centering
	\begin{tikzpicture}[scale=.85, every node/.style={font=\small}, thick, baseline={(-90:1)}]
		\tikzstyle{phantomvertex} = [inner sep=0pt, outer sep=0pt, minimum width=0pt]
		\tikzstyle{edge} = [draw, -, odarkblue, shorten >=-0.5pt, shorten <=-0.5pt]
		\tikzstyle{roundlabel} = [circle, inner sep=2pt, minimum width=15pt, fill=white, draw=odarkblue]
		\tikzstyle{labelled} = [rounded corners=4pt, inner sep=2pt, minimum height=12pt, minimum width=12pt, fill=white, draw=odarkblue]
		
		\def\vertices{
				{(0,0)/a/a},
					{($(a)+(25:1)$)/g/g},
						{($(g)+(5:1)$)/d/d},
						{($(g)+(50:1)$)/e/e},
						{($(g)+(95:1)$)/b/b},
					{($(a)+(85:1)$)/c/c},
						{($(c)+(70:1)$)/h/h},
						{($(c)+(125:1)$)/i/i},
					{($(a)+(145:1)$)/f/f}}	
		\def\edges{
				{a/g/3},
					{g/d/4},
					{g/e/7},
					{b/g/2},
				{a/c/1},
					{c/h/6},
					{i/c/8},
				{f/a/5}}
		
		\begin{scope}[scale=4/3]
			\foreach \pos/\name/\label in \vertices
				\node[phantomvertex] (\name) at \pos {};
		\end{scope}
		\foreach \source/\dest/\label in \edges
			\path[edge] (\source) -- (\dest);
		\foreach \source/\dest/\label in \edges
			\path[bend right] (\source) to node[midway, black] {\footnotesize\label} (\dest);
		\foreach \pos/\name/\label in \vertices
			\node[roundlabel] (\name) at (\name) {\footnotesize\label};
		
		\def\cutvertices{
				{(6,-.5)/0/\,abcdefghi\,},
					{($(0)+(40:1)$)/1/\,chi\,},
						{($(1)+(20:1)$)/11/\,ci\,},
							{($(11)+(10:1)$)/111/i},
							{($(11)+(80:1)$)/112/c},
						{($(1)+(80:1)$)/12/h},
					{($(0)+(140:1)$)/2/\,abdefg\,},
						{($(2)+(80:1)$)/21/\,adefg\,},
							{($(21)+(50:1)$)/211/\,af\,},
								{($(211)+(20:1)$)/2111/f},
								{($(211)+(90:1)$)/2112/a},
							{($(21)+(130:1)$)/212/\,edg\,},
								{($(212)+(70:1)$)/2121/\,eg\,},
									{($(2121)+(50:1)$)/21211/g},
									{($(2121)+(140:1)$)/21212/e},
								{($(212)+(150:1)$)/2122/d},
						{($(2)+(160:1)$)/22/b}}	
		\def\cutedges{0/1, 1/11, 11/111, 11/112, 1/12,
				0/2, 2/21, 21/211, 211/2111, 211/2112,
				21/212, 212/2121, 2121/21211, 2121/21212,
				212/2122, 2/22}
		
		\begin{scope}[scale=1]
			\foreach \pos/\name/\label in \cutvertices
				\node[phantomvertex] (\name) at \pos {};
		\end{scope}
		\foreach \source/\dest in \cutedges
			\path[edge] (\source) -- (\dest);
		\foreach \pos/\name/\label in \cutvertices
			\node[labelled] (\name) at (\name) {\scriptsize\label\vphantom{fg}};
		
		\path[->, >=stealth', shorten >=4pt, shorten <=4pt, thin, gray]
			(a.south east) edge[bend right] 
			(0.west);
		\node at ($(a)+(0,-.75)$) {\small $\tau$};
		\node at ($(0)+(0,-.75)$) {\small $\operatorname{Cut}\,(\tau)$};
	\end{tikzpicture}
	\caption{A labelled tree $\tau$ and its cut-tree\protect\\
	(the edges of $\tau$ are labelled in the order they are removed)}
\end{figure}

Cut-trees were introduced in~\cite{bertoin2012fires} as a means of generalising
the study of the number of cuts necessary to isolate a marked vertex or a finite number of marked vertices.
In this section, we will study the local and scaling limits
of two models of cut-trees,
studied in~\cite{bertoin2012fires} and~\cite{bertoin2015cutrecursivetrees},
which both satisfy the Markov branching property.
Also see~\cite{bertoin2013cutgw} and~\cite{dieuleveut2015stablegwcuttree}
for the study of the cut-trees of conditioned Galton-Watson trees

\subsubsection{Cut-trees of Cayley trees}

A \emph{Cayley tree} of size $n\geq 1$ is a labelled tree $\tau_n$ chosen uniformly at random
in the set of trees with $n$ labelled vertices (for convenience, with labels $1$ through $n$).
It is well-known that, viewed as an unlabelled tree,
$\tau_n$ has the same distribution as an unordered Galton-Watson tree with offspring law $\operatorname {Poisson}\,(1)$
conditioned to have $n$ vertices.
For all $n\geq 1$, let $T_n := \operatorname {Cut}\,(\tau_n)$
be the cut-tree of a Cayley tree with size $n$.

\smallskip

Let $(\vartheta_n)_{n\geq 0}$ be a sequence of \iid
unconditioned $\GW_{\operatorname {Poisson}\,(1)}$ trees.
Let $T_\infty$ be the tree obtained by attaching for each $n\geq0$ the cut-tree of $\vartheta_n$
to the vertex of an infinite branch at height $n$ by an edge.
In other words, set $T_\infty := \ttb_\infty \bigotimes_{n\geq 0} \big(\ttv_n,
\lBrack\operatorname {Cut} \, (\vartheta_n)\rBrack\big)$.

\smallskip

The aim of this section will be to prove the next two results.

\begin{proposition}\label{prop:cut-caley-local-limits}
	When $n\to\infty$,
	$T_n$ converges to $T_\infty$ in distribution
	with respect to the local limit topology.
\end{proposition}

\begin{proposition}\label{prop:cut-caley-scaling-limits}
	Endow $T_\infty$ with counting measure on its leaves $\mu_\infty$.
	Then $(T_\infty/R, \mu_\infty/R^2)$ converges as $R$ goes to infinity
	to $(\calT_B, 1/2 \cdot \mu_B)$ in distribution with respect to the $\upD_\GHP$ topology,
	where $(\calT_B, \mu_B)$ denotes the immigration Brownian tree.
\end{proposition}

\paragraph{Markov branching property}
It was stated in~\cite{bertoin2012fires} that $(T_n)$ satisfies the Markov branching property
and more specifically, that the distribution of $T_n$ is $\MB^{\calL,q}_n$
where the associated first-split distributions are given by $q_1(1) = 1$,
for all $n\geq 2$, $q_n(p\neq 2) = 0$ and if $1\leq k < n/2$,
\[
	q_n (n-k, k)
	= \frac {(n-k)^{n-k-1}} {(n-k)!}
	\, \frac {k^{k-1}} {k!}
	\, \frac {(n-2)!} {n^{n-3}}.
\]

The tree $T_\infty$ can be described as an infinite Markov branching tree
with distribution $\MB^{\smash\calL,q,q_{\smash\infty}}_\infty$
where the probability measure $q_\infty$ is defined
by $q_\infty(p\neq 2) = q_\infty(m_\infty\neq 1) = 0$
and for all positive $k$, $q_\infty(\infty,k) = \prob {\#\vartheta=k}$
where $\vartheta$ is a $\GW_{\operatorname {Poisson}\,(1)}$ tree.
Recall that the size of $\vartheta$ has Borel distribution with parameter $1$,
therefore, for any positive $k$, $q_\infty(\infty,k) = k^{k-1} \, \E^{-k} / k!$.

\paragraph{Local limits}
For any $k\geq 1$, when $n\to\infty$, Stirling's approximation gives
\[
	q_n (n-k, k)
	\sim \frac {k^{k-1}\, \E^{2-k}} {k!} ( 1 -2/n )^n
	\xrightarrow [n\to\infty] {} \; \frac {k^{k-1} \, \E^{-k}} {k!}
	= q_\infty(\infty,k).
\]
We may then use Corollary~\ref{cor:local-limits-mb-one-spine}
and thus prove Proposition~\ref{prop:cut-caley-scaling-limits}.

\paragraph{Scaling limits}
Section~2.1 in~\cite{bertoin2012fires} proves that $n^{1/2} \, (1-s_1) \, \bar q_n(\D\bfs)$
converges weakly to $(1-s_1) \, 1/2 \cdot \nu_B (\D\bfs)$ in the sense of finite measures on $\Sdec_{\leq 1}$.

Moreover, $q_\infty$ is \as binary,
and Stirling's approximation ensures that $n^{3/2} q_\infty(\infty,n) \to (2\pi)^{-1/2}$.
Therefore, if $\Lambda$ is such that $(\infty,\Lambda)$ follows $q_\infty$
and if $q^{(R)}$ is the distribution of $\Lambda/R^2$,
then Proposition~\ref{prop:unary-immigration-criterion} implies that $R \, (1\wedge\norm\bfs) \, q^{(R)}(\D\bfs)$
weakly converges to $(1\wedge\norm\bfs) \, 1/2 \cdot I_B(\D\bfs)$ as $R\to\infty$.
In other words, Assumption~\hyperref[assumption:immigration]{$(\mathtt I)$}
is also satisfied.

Consequently, Theorem~\ref{thm:scaling-limits-infinite-mb-trees}
ensures that when $R\to\infty$, $(T_\infty/R, \mu_\infty/R^2)$ converges
in distribution to a $(1/2, 1/2 \cdot \nu_B, 1/2 \cdot I_B)$ fragmentation tree
with immigration with respect to the topology induced by $\upD_\GHP$.
Remark~\ref{nb:ssfi-rescaling} then concludes the proof of Proposition~\ref{prop:cut-caley-scaling-limits}.

\subsubsection{Cut-trees of uniform recursive trees}

A \emph{recursive tree} with $n$ vertices is a labelled tree (with labels $1$ through $n$)
such that the labels on the shortest path from $1$ to any given leaf are increasing.
For all $n\geq 1$, let $\tau_n$ denote a labelled tree chosen uniformly at random among
the set of recursive trees with $n$ vertices
and call $T_n$ its cut-tree.

\smallskip

Define a probability measure $\pi$ on $\bbN$ by $\pi(n) = 1/[n(n+1)]$
and let $(X_n, \vartheta_n)_{n\geq 0}$ be a sequence of \iid variables,
where for each $n$, $X_n$ follows $\pi$ and conditionally on $X_n = \ell$,
$\vartheta_n$ is a recursive tree with $\ell$ vertices.
Define $T_\infty$ as the tree obtained by attaching the cut-tree of $\vartheta_n$
by an edge to an infinite branch at height $n$,
\ie set $T_\infty := \ttb_\infty \bigotimes_{n\geq 0} \big(\ttv_n,
\lBrack\operatorname {Cut} \, (\vartheta_n)\rBrack\big)$.

\begin{proposition}\label{prop:cut-recursive-local-limits}
	In the sense of the local limit topology,
	$T_n$ converges in distribution to $T_\infty$
	when $n\to\infty$.
\end{proposition}

It was observed in~\cite{bertoin2014percolationrecursivetrees} and~\cite{bertoin2015cutrecursivetrees}
that the sequence $(T_n)_{n\geq 1}$ is Markov branching.
Moreover, we may deduce from \cite[Section~2]{bertoin2014percolationrecursivetrees}
the expression of the respective distributions $q_n$ of $\Lambda^{\smash\calL}(T_n)$.
Clearly, $q_1(1) = 1$, and for $n\geq 2$, if $X$ denotes a random variable with distribution $\pi$,
then for all $k\leq n/2$, $q_n(n-k,k) = \prob {X=k \vert X<n} + \prob {X=n-k \vert X<n} \, \ind_{k\neq n/2}$.
In particular,
\[
	q_n(n-k,k) =
	\begin{dcases}
		\frac n {n-1} \bigg( \frac 1 {k(k+1)} + \frac 1 {(n-k) (n-k+1)} \bigg)
		& \text{if $k<n/2$},\\
		\frac 4 {(n-1) (n+2)}
		& \text{if $k=n/2$}.
	\end{dcases}
\]

The tree $T_\infty$ may also be described as an infinite Markov branching tree
with distribution $\MB^{\smash\calL,q,q_{\smash\infty}}_\infty$
where the measure $q_\infty$ is given by $q_\infty(p\neq 2) = q_\infty(m_\infty\neq 1) = 0$
and for all $k\geq 1$, $q_\infty(\infty, k) = \pi(k)$.

\smallskip

If $k$ is a fixed integer, then $q_n(n-k,k)$ clearly converges to $q_\infty(\infty,k)$.
We conclude the proof of Proposition~\ref{prop:cut-recursive-local-limits}
with Corollary~\ref{cor:local-limits-mb-one-spine}.

\begin{remark}
	It was shown in~\cite{bertoin2015cutrecursivetrees}
	that $(n/\log n)^{-1} T_n$ converges to the real interval $[0,1]$
	rooted at $0$ and endowed with the Lebesgue measure.
	However, Assumption~\hyperref[assumption:scaling]{$(\mathtt S)$} doesn't hold.
\end{remark}

\subsection[The \texorpdfstring{$\alpha$-$\gamma$}{alpha-gamma} model]{The $\boldsymbol\alpha$-$\boldsymbol\gamma$ model}
\label{sec:alpha-gamma}

In this section, we will study trees generated according
to the algorithm of the $\alpha$-$\gamma$ model described in~\cite{ford2009alphagamma}.
This algorithm was introduced as an interpolation between various models
of sequentially growing trees such as R\'emy's algorithm~\cite{remy1985binarygw},
used to generate uniform binary trees with any number of leaves,
Marchal's~\cite{marchal2008stabletree}, which gives the $n$-dimensional marginal of Duquesne-Le~Gall's stable trees
(the discrete tree spanned by $n$ leaves chosen uniformly at random in a stable tree),
and Ford's $\alpha$-model~\cite{ford2005alpha}, used for instance in phylogeny.

\smallskip

Let $0 \leq \gamma \leq \alpha \leq 1$.
Start with $T_1 := \{\varnothing\}$, the trivial tree,
and $T_2 := \{\varnothing,(1),(2)\}$, a tree with two leaves attached to its root.
Then for $n\geq 3$, conditionally on the tree $T_{n-1}$:
\begin{itemize}
	\item Assign to each edge of $T_{n-1}$ (considered as a planted tree,
	\ie a tree in which a phantom edge has been attached under the root)
	the weight $1-\alpha$ if the edge ends with a leaf or $\gamma$ otherwise,
	\item Also assign to each non-leaf vertex $u$ the weight $[c_u(T_{n-1})-1]\alpha - \gamma$,
	\item Pick an edge or a vertex in $T_{n-1}$ with probability proportional to these weights,
	\begin{itemize}
		\item If an edge was picked, place a new vertex at its middle
		and attach a new leaf to it,
		\item If a vertex was selected, attach a new leaf to it,
	\end{itemize}
\end{itemize}
and call $T_n$ the tree thus obtained.
We will also note $\AG_{\alpha,\gamma}^n$ its distribution
for all $n\geq 1$ and $0\leq\gamma\leq\alpha\leq 1$.

\begin{remark}
	As mentioned at the beginning of this section,
	some particular choices of parameters give previously studied algorithms:
	\begin{itemize}
		\item When $\alpha = \gamma = 1/2$,
		we get R\'emy' algorithm~\cite{remy1985binarygw},
		\item If $\beta\in(1,2)$, taking $\alpha=1/\beta$ and $\gamma=1-\alpha$
		gives Marchal's algorithm~\cite{marchal2008stabletree},
		\item When $\alpha=\gamma$, this algorithm coincides
		with that of Ford's $\alpha$-model~\cite{ford2005alpha}.
	\end{itemize}
\end{remark}

\paragraph{The Beta geometric distribution}
Fix $\theta$ in $(0,1)$.
Let $\Pi$ be a Beta random variable with parameters $(1-\theta,\theta)$,
and conditionally on $\Pi$, let $X$ have geometric distribution with parameter $1-\Pi$,
meaning that $\prob {X=n \,\vert\, \Pi} = \Pi^n (1-\Pi)$ for every integer $n\geq 0$.
We say that $X$ is a \emph{beta geometric} variable of parameters $(\theta,1-\theta)$.
For all integers $n\geq 0$,
\begin{align*}
	\prob {X = n}
	& = \esp [1] {\Pi^n (1-\Pi)}
	= \frac 1 {\Beta(1-\theta, \theta)} \int_0^1 x^{n-\theta} (1-x)^\theta \D x
	= \frac {\theta\,\Gamma(n+1-\theta)} {\Gamma(1-\theta) \, (n+1)!}.
\end{align*}
We will also use the convention $X=0$ \as if $\theta=1$
and $X=\infty$ \as if $\theta=0$.

\paragraph{Infinite $\alpha$-$\gamma$ tree}
Assume that $0<\gamma\leq\alpha\leq 1$.
Let $(X_n)_{n\geq 0}$ be a sequence of \iid beta geometric random variables with parameters $(\gamma/\alpha,1-\gamma/\alpha)$.
Let $(Y_{n,k}, \tau_{n,k})$ be a sequence of \iid variables independent of $(X_n)_n$
such that $Y_{n,k}$ is a $(\alpha,1-\alpha)$ beta geometric variable
and conditionally on $Y_{n,k} = \ell$, $\tau_{n,k}$ is an $\alpha$-$\gamma$ tree with $\ell+1$ leaves,
\ie $\tau_{n,k}$ follows $\AG_{\alpha,\gamma}^{\ell+1}$.

Finally, conditionally on $(X_n,Y_{n,k},\tau_{n,k} ; n\geq 0, k\geq 0)$,
define $T_\infty$ as the tree obtained by grafting for each $n\geq 0$
the concatenation of $\tau_{n,i}$, $0\leq i\leq X_n$
at height $n$ on an infinite branch.
In other words,
\[
	\textstyle
	T_\infty := \ttb_\infty \bigotimes_{n\geq 0} \big(\ttv_n,
		\lBrack \tau_{n,0}, \dots \tau_{n,X_n} \rBrack\big)
\]
and denote by $\AG_{\alpha,\gamma}^\infty$ its distribution.

\begin{remark}
	In Ford's $\alpha$-model, \ie when $\alpha=\gamma>0$,
	$X_n=0$ \as for all $n$, so a single tree is grafted at each height.
	Similarly, when $\alpha=1$ and $0<\gamma\leq\alpha$, $Y_{n,k}=0$ \as.
\end{remark}

\smallskip

We will start our study of the $\alpha$-$\gamma$ model by proving
this next proposition with the help of Theorem~\ref{thm:local-limits-markov-branching}.
Similar results for $\alpha=\gamma$ were already proved in~\cite{stefansson2009ford}
and in~\cite[Lemma~3.8]{broutin2015andortrees} for any $0<\gamma\leq\alpha\leq 1$.

\begin{proposition}\label{prop:alpha-gamma-local-limit}
	For any $0<\gamma\leq\alpha\leq 1$,
	the probability measure $\AG_{\alpha,\gamma}^n$
	converges weakly to $\AG_{\alpha,\gamma}^\infty$
	as $n$ grows to $\infty$
	in the sense of the local limit topology.
\end{proposition}

We will then study the scaling limits of these infinite trees:
Section~\ref{sec:alpha-gamma-scaling-limits} will focus on the case $0<\gamma<\alpha<1$
and Section~\ref{sec:ford-alpha-model}, on $\alpha=\gamma$.

\subsubsection{Markov branching property and local limits}\label{sec:alpha-gamma-local-limits}
Proposition~1 in~\cite{ford2009alphagamma} states
that the sequence $(\AG_{\alpha,\gamma}^n)_n$ satisfies the Markov branching property.
Moreover, the sequence $q = (q_n)_n$
associated to the first split distributions of $T_n$,
\ie such that $q_n$ is the law of $\Lambda^\calL(T_n)$ for all $n\geq 1$,
is given by $q_1 (\varnothing) = 1$,
and for any $n\geq 2$, for all $\lambda=(\lambda_1,\dots,\lambda_p)\in\calP_n$,
\[
	q_n (\lambda)
	= \frac 1 {\prod\limits_{j\geq 1} m_j(\lambda)} \,
	\bigg(\gamma + \frac {1-\alpha-\gamma} {n(n-1)} \sum_{i\neq j} \lambda_i \lambda_j\bigg)
	\frac {\Gamma(1-\alpha) \, n!} {\Gamma(n-\alpha)} \,
	\frac {\alpha^{p-2} \, \Gamma(p-1-\gamma/\alpha)} {\Gamma(1-\gamma/\alpha)} \,
	\prod_{i=1}^p \frac {\Gamma(\lambda_i-\alpha)} {\Gamma(1-\alpha) \, \lambda_i!},
\]
with the conventions $\Gamma(0)=\infty$ and $\Gamma(0)/\Gamma(0) = 1$
(which will be used throughout this section).

We can also write $\AG_{\alpha,\gamma}^\infty = \MB^{\smash\calL,q,q_{\smash\infty}}_\infty$
where $q_\infty$ is the measure on $\calP_\infty$ given by
\[
	q_\infty (\infty, \lambda)
	= \frac {\gamma/\alpha \, \Gamma(p-\gamma/\alpha)} {\Gamma(1-\gamma/\alpha) \, p!} \,
	\frac {p!} {\prod_{j\geq 1} m_j(\lambda)!} \,
	\prod_{i=1}^p \frac {\alpha \, \Gamma(\lambda_i-\alpha)} {\Gamma(1-\alpha) \, \lambda_i!}
\]
for all $\lambda = (\lambda_1,\dots,\lambda_p)$ in $\calP_{<\infty}$
and $q_\infty(\mu) = 0$ for all $\mu$ in $\calP_\infty$
with either $p(\mu)=1$ or $m_\infty(\mu)>1$.

If $X$ has beta geometric distribution with parameters $(\gamma/\alpha, 1-\gamma/\alpha)$
and is independent of the \iid sequence $(Y_i)_{i\geq 0}$ of beta geometric
variables with  parameters $(\alpha, 1-\alpha)$,
for any $\lambda = (\lambda_1,\dots,\lambda_p)$ in $\calP_{<\infty}$,
we get~that
\[
	q_\infty (\infty, \lambda)
	= \prob [1] {X=p-1, \, (Y_1+1,\dots,Y_{p(\lambda)}+1)^\downarrow = \lambda}
\]
which ensures that $q_\infty$ is a probability measure on $\calP_\infty$.

\begin{proof}[of~Proposition~\ref{prop:alpha-gamma-local-limit}]
	Let $\lambda=(\lambda_1,\dots,\lambda_p)$ be in $\calP_{<\infty}$.
	Then, for $n$ large enough, in light of Stirling's approximation,
	\begin{align*}
		q_n (n-\norm\lambda,\lambda)
		& = \frac 1 {\prod_{j\geq 1} m_j(\lambda)!} \,
		\bigg(\gamma + \overbrace {\frac {1-\alpha-\gamma} {n(n-1)} \sum_{i\neq j} \lambda_i \lambda_j
			}^{\substack {0\\ {\displaystyle\uparrow} \mathrlap {\; n\to\infty}}} \bigg) \,
		\overbrace {\frac {\Gamma(n-\norm\lambda-\alpha) \, n!} {\Gamma(n-\alpha) \, (n-\norm\lambda)!}
			}^{\substack {1\\ {\displaystyle\uparrow} \mathrlap {\; n\to\infty}}}\\
		& \qquad\qquad\qquad\qquad \times
			\frac {\alpha^{p-1} \, \Gamma(p-\gamma/\alpha)} {\Gamma(1-\gamma/\alpha)} \,
			\prod_{i=1}^p \frac {\Gamma(\lambda_i-\alpha)} {\Gamma(1-\alpha) \, \lambda_i!}\\
		& \qquad \xrightarrow [n\to\infty] {}
			\frac {\gamma/\alpha \, \Gamma(p-\gamma/\alpha)} {\Gamma(1-\gamma/\alpha) \, p!} \,
			\frac {p!} {\prod_{j\geq 1} m_j(\lambda)!} \,
			\prod_{i12}^p \frac {\alpha \, \Gamma(\lambda_i-\alpha)} {\Gamma(1-\alpha) \, \lambda_i!}
		= q_\infty(\infty,\lambda).
	\end{align*}
	We conclude with Corollary~\ref{cor:local-limits-mb-one-spine}.
\end{proof}

\subsubsection{Scaling limits}\label{sec:alpha-gamma-scaling-limits}
In this paragraph, we will assume that $0<\gamma<\alpha<1$.
Let $\Sigma$ be an $\alpha$-stable subordinator with Laplace exponent
$\lambda\mapsto\lambda^\alpha$ and L\'evy measure $\Pi_\alpha(\D t)
= \alpha/\Gamma(1-\alpha) \, t^{-1-\alpha} \, \ind_{t>0} \, \D t$.
Define $\Delta$ as the decreasing rearrangement of its jumps on $[0,1]$.
We define the dislocation measure $\nu_{\alpha,\gamma}$ for all measurable functions $f:\Sdec_{\leq 1}\to\bbR_+$ by
\[
	\int_{\Sdec_{\leq 1}} f \, \D\nu_{\alpha,\gamma}
	:= \frac {\Gamma(1-\alpha)} {\alpha \, \Gamma(1-\gamma/\alpha)}
	\esp [2] {\Sigma_1^{\alpha+\gamma} \, \big(\gamma + (1-\alpha-\gamma) \,
		{\textstyle\sum_{i\neq j}} \Delta_i \Delta_j \big) \,
		f\big( \Delta / \Sigma_1 \big)}.
\]
Results from \cite{ford2009alphagamma} and \cite{hmpw2008continuumtree}
ensure that the family $q$ satisfies Assumption~\hyperref[assumption:scaling]{$(\mathtt S)$}:
when $n\to\infty$, $n^\gamma (1-s_1)\, \bar q_n(\D\bfs)$
converges weakly towards $(1-s_1) \, \nu_{\alpha,\gamma}(\D\bfs)$.

We also define the immigration measure $I_{\alpha,\gamma}$
for all measurable functions $F:\Sdec\to\bbR_+$ by
\[
	\int_{\Sdec} F \, \D I_{\alpha,\gamma}
	:= \frac {\gamma/\alpha} {\Gamma(1-\gamma/\alpha)} \,
		\int_0^\infty \frac {\esp [1] {F(t^{1/\alpha}\,\Delta)}} {t^{1+\gamma/\alpha}} \, \D t.
\]

\begin{proposition}\label{prop:alpha-gamma-scaling-limits}
	Let $T$ be distributed according to $\AG_{\smash{\alpha,\gamma}}^\infty$
	and endow it with $\mu_T$, the counting measure on the set of its leaves.
	With respect to the $\upD_\GHP$ topology,
	$(T/R,\mu_T/R^{1/\gamma})$ converges in distribution to
	a $(\gamma, \nu_{\alpha,\gamma}, I_{\alpha,\gamma})$
	fragmentation tree with immigration.
\end{proposition}

\begin{proof}
	Let $\Lambda$ be such that $(\infty,\Lambda)$ follows $q_\infty$.
	For all $R\geq 1$, set $q^{(R)}$ as the distribution of $R^{-1/\gamma} \Lambda$.
	In light of Theorem~\ref{thm:scaling-limits-infinite-mb-trees},
	it is sufficient to prove that $R \, (1\wedge\norm\bfs) \, q^{(R)}(\D\bfs)
	\Rightarrow(1\wedge\norm\bfs) \, I_{\alpha,\gamma}(\D\bfs)$ when $R\to\infty$.
	
	To prove this claim, we may proceed as in the proof of Proposition~\ref{prop:kesten-stable-immigration}.
	The only significant difference is that the constant $\beta$
	used near the end of that proof must now belong to the open interval $(\gamma,\alpha)$. 
\end{proof}

\begin{remark}
	Let $\beta$ be in $(1,2)$ and set $\alpha=1/\beta$, $\gamma=1-\alpha$.
	It was proved in~\cite{marchal2008stabletree}
	that the distribution $\AG_{1/\beta,1-1/\beta}^n$
	coincides with $\GW_{\smash\xi}^{\smash\calL,n}$,
	where the generating function of $\xi$ is given by $s\mapsto s + \beta^{-1} (1-s)^\beta$.
	The results of Propositions~\ref{prop:alpha-gamma-local-limit} and~\ref{prop:alpha-gamma-scaling-limits}
	are then consistent with those of Proposition~\ref{prop:gw-local-limit} and Remark~\ref{nb:kesten-scaling-limits-leaves}.
\end{remark}

\subsubsection{Ford's \texorpdfstring{$\boldsymbol\alpha$}{alpha}-model}\label{sec:ford-alpha-model}
When $\alpha=\gamma$, no weight is ever assigned to vertices.
Consequently, the trees generated by this algorithm are \as binary (\ie each vertex has either two children or none).
Furthermore, the sequence $(q_n)_n$ of associated first split distributions is much simpler:
$q_1 (\varnothing)$ still equals $1$,
and for $n\geq 2$, if $\alpha<1$, for all $1\leq k\leq n/2$,
\[
	q_n (n-k,k) = (2-\ind_{2k=n}) \, \binom n k \,
	\frac {\Gamma(n-k-\alpha)\,\Gamma(k-\alpha)} {\Gamma(1-\alpha)\,\Gamma(n-\alpha)}
	\, \Bigg( \frac \alpha 2 + \frac {(1-2\alpha) \, (n-k) \, k} {n \, (n-1)} \Bigg),
\]
finally if $\alpha=1$, $q_n (n-1,1) = 1$.

Moreover, if $\alpha$ is positive, for all $n\geq 1$,
$q_\infty (\infty,n) = \alpha \, \Gamma(n-\alpha) / [ \Gamma(1-\alpha) \, n!]$
and $q_\infty(\lambda) = 0$ if $p(\lambda)\neq 2$ or $m_\infty(\lambda)\neq 1$.
As a result, a tree with distribution $\AG_{\alpha,\alpha}^\infty$ is obtained
by grafting at each height of an infinite spine a single tree with distribution $\AG_{\alpha,\alpha}^{N+1}$
where $N$, its number of leaves minus $1$, has beta geometric distribution of parameters $(\alpha,1-\alpha)$.

\paragraph{Scaling limits of Ford's $\alpha$ model}
Let $\alpha\in(0,1)$.
Results from~\cite[Section~5.2]{hmpw2008continuumtree} ensure that $(T_n)_n$
satisfies Assumption~\hyperref[assumption:scaling]{$(\mathtt S)$}:
when $n\to\infty$, $n^\alpha \, (1-s_1) \, \bar q_n(\D\bfs)
\Rightarrow (1-s_1) \, \smash{\nu^{(\upF)}_\alpha} (\D\bfs)$
where $\smash{\nu^{(\upF)}_\alpha}$ is the binary dislocation measure defined
for all measurable $f:\Sdec_{\leq 1}\to\bbR_+$ by
\[
	\int f \, \D\smash{\nu^{(\upF)}_\alpha} = \frac 1 {\Gamma(1-\alpha)} \int_{1/2}^1
		\bigg(\frac {\alpha} {[x(1-x)]^{1+\alpha}} + \frac {2-4\alpha} {[x(1-x)]^\alpha}\bigg)
		\, f(x,1-x,0,0,\dots) \, \D x.
\]
Furthermore, $q_\infty$ is \as binary
and Stirling's approximation ensures that $q_\infty(\infty,n)$
is equivalent to $[\alpha/\Gamma(1-\alpha)] \, n^{-1-\alpha}$ when $n\to\infty$.
Consequently, if $\Lambda$ is such that $(\infty,\Lambda)$ follows $q_\infty$
and $q^{(R)}$ denotes the distribution of $\Lambda/R^{1/\alpha}$,
Proposition~\ref{prop:unary-immigration-criterion} proves that
$R \, (1\wedge\norm\bfs) \, q^{(R)}(\D\bfs) \Rightarrow
(1\wedge\norm\bfs) \, [\alpha/\Gamma(1-\alpha)] \, I^{\operatorname {un}}_\alpha(\D\bfs)$ as $R\to\infty$.
Therefore, if we set $I^{(\upF)}_\alpha := \alpha/\Gamma(1-\alpha) \cdot I^{\operatorname {un}}_\alpha$,
we may use Theorem~\ref{thm:scaling-limits-infinite-mb-trees} and Proposition~\ref{prop:volume-growth-cv}
to get the following result:

\begin{proposition}\label{prop:ford-scaling}
	Let $T$ be a $\AG_{\smash {\alpha,\alpha}}^\infty$ tree with $\alpha$ in $(0,1)$
	and endow it with the counting measure on the set of its leaves.
	Then, $(T/R,\mu_T^\calL/R^{1/\alpha})$ converges in distribution to
	a $(\alpha, \smash{\nu^{(\upF)}_\alpha}, \smash{I^{(\upF)}_\alpha})$-fragmentation
	tree with immigration with respect to the topology induced by $\upD_\GHP$.
\end{proposition}

\begin{remark}
	When $\alpha=1/2$, \ie in R\'emy's algorithm,
	these results coincide with Proposition~\ref{prop:gw-local-limit}
	and Proposition~\ref{prop:kesten-scaling-limit}~$(i')$
	for $\xi(0) = \xi(2) = 1/2$.
\end{remark}

\paragraph{When $\alpha=1$}
In this case, the algorithm's output is deterministic:
for each $n\geq 2$, a tree $T_n$ with distribution $\AG_{1,1}^n$
is simply equal to a branch of length $n-1$ upon which a single leaf
has been grafted at each non-leaf vertex (a ``comb'' of length $n$).
Similarly, an infinite tree with distribution $\AG_{1,1}^\infty$
is the ``infinite comb'', obtained by attaching a single leaf to
all the vertices of the infinite branch.

As a result, if $T$ has distribution $\AG_{1,1}^\infty$
and $\mu_T$ denotes the counting measure on the set of its leaves,
then clearly, $(T/R, \mu_T/R)$ converges as $R\to\infty$
to the metric space $\bbR_+$ rooted at $0$ and endowed with the usual Lebesgue.

\paragraph{When $\alpha=0$}
Observe that $q_n (n-k,k) = (2-\ind_{k=n/2})/(n-1)$.
Then for all $K\geq 1$ and $n$ large enough,
\[
	\prob {\Lambda^\calL(T_n)\wedge K = \infty_2\wedge K}
	= 1 - \frac {K-1} {n-1} \xrightarrow [n\to\infty] {} 1,
\]
which implies $\Lambda^\calL(T_n) \to (\infty,\infty)$ \as when $n\to\infty$.
Theorem~\ref{thm:local-limits-markov-branching} then ensures that $T_n$
converges in distribution to the complete infinite binary tree (in which every vertex has $2$ children).
Moreover, since $T_n\subset T_{n+1}$ \as,
this convergence happens almost surely.

\subsection[Aldous' \texorpdfstring{$\beta$}{beta}-splitting model]{Aldous' $\boldsymbol\beta$-splitting model}

This section will focus on the study a model of binary random trees introduced
in \cite[Section~4]{aldous1996cladograms} as a Markov branching model.
Let $\beta>-2$ be fixed.
Set $q_1(\varnothing) := 1$ and for all $n\geq 2$ and $1\leq k\leq n/2$,
\[
	q_n(n-k,k) := \frac {2-\ind_{2k=n}} {Z_n} \,
		\frac {\Gamma(n-k+1+\beta)} {(n-k)!} \,
		\frac {\Gamma(k+1+\beta)} {k!}
\]
where $Z_n$ is a normalising constant.
For all $n\geq 1$, let $T_n$ be a random tree with distribution $\MB^{\calL,q}_n$.

\begin{remark}
\begin{itemize}
	\item The constant $Z_n$ is given by
	\[
		Z_n := {\textstyle\sum_{k=1}^{n-1}}
			\frac {\Gamma(n-k+1+\beta)} {(n-k)!} \,
			\frac {\Gamma(k+1+\beta)} {k!}.
	\]
	When $\beta>-1$, it simplifies to
	$Z_n = [\Beta(1+\beta,1+\beta) - 2 \, \Beta(n+1+\beta,1+\beta)] \cdot \Gamma(n+2+2\beta)/n!$
	(where $\Beta$ denotes the usual Beta function)
	and when $\beta=-1$, it becomes $Z_n = 2/n \cdot \sum_{k=1}^{n-1} k^{-1}$.
	
	\item When $\beta=-3/2$, observe that the sequence $(q_n)_n$
	is the same as that of the $\alpha$-model with $\alpha=1/2$ (see Section~\ref{sec:ford-alpha-model}).
	Therefore, like R\'emy's algorithm,
	this model generates uniform binary trees with any given number of leaves.
\end{itemize}
\end{remark}

There are three regimes in this model,
respectively $\beta>-1$, $\beta=-1$ and $\beta\in(-2,-1)$.
The asymptotic behaviour of $q_n$ were studied in \cite[Section~5]{aldous1996cladograms}
in these three regimes.

\subsubsection{Local limits}
In this paragraph, we will focus on the study of the local limits of $T_n$.
We will once again rely on the Markov branching nature of the model 
and on Theorem~\ref{thm:local-limits-markov-branching}.

\begin{proposition}\label{prop:beta-local-limit}
	\noindent$\beta\geq -1:$\quad
	In the sense of the local limit topology,
	$T_n$ converges in distribution to the infinite binary tree.
	
	\smallskip
	
	\noindent$\beta\in(-2,-1):$\quad
	Let $X$ follow the beta geometric distribution with parameters $(2+\beta,-1-\beta)$
	(see Section~\ref{sec:alpha-gamma}).
	Define $q_\infty$, a probability measure on $\calP_\infty$,
	by $q_\infty(\infty,k) = \prob{X=k-1}$ for any $k\geq 1$
	and $q_\infty(\lambda) = 0$ if $p(\lambda) \neq 2$ or $m_\infty(\lambda) \neq 1$.
	With these notations, $T_n$ converges in distribution to $\MB^{\smash\calL,q,q_{\smash\infty}}_\infty$
	with respect to the local limit topology.
\end{proposition}

\begin{remark}
	Suppose $\beta\in(-2,-1)$
	and let $(X_n,\tau_n)_{n\geq 0}$ be an \iid sequence such that
	for each $n$, $X_n$ has beta geometric distribution with parameters $(2+\beta,-1-\beta)$
	and conditionally on $X_n = k-1$, $\tau_n$ is distributed like $T_k$.
	Finally, denote by $T_\infty$ the tree obtained by attaching by a single edge
	the tree $\tau_n$ respectively at each height $n$ of an infinite branch,
	\ie $T_\infty := \ttb_\infty \bigotimes_{n\geq 0} \big(\ttv_n,
	\lBrack \tau_n \rBrack \big)$.
	The tree $T_\infty$ hence obtained has distribution $\MB^{\smash\calL,q,q_{\smash\infty}}_\infty$.
\end{remark}

\begin{proof}
	Observe that in light of Stirling's approximation, $\Gamma(n+1+\beta)/n! \sim n^\beta$
	when $n\to\infty$.
	
	\smallskip
	
	\noindent$\beta\geq -1:$\quad
	When $\beta>-1$, using Stirling's approximation once again,
	we get that $Z_n\sim \Beta(1+\beta,1+\beta) \, n^{-1-2\beta}$
	so if $k\geq 1$ is a fixed integer,
	$q_n(n-k,k) = O(n^{1+\beta})$ when $n\to\infty$.
	
	When $\beta=-1$, $Z_n \sim 2/n \cdot \log n$
	hence, for any fixed $k\geq 1$,
	$q_n(n-k,k) \sim 1/(k\, \log n)$ as~$n\to\infty$.
	
	Therefore, for any $\beta\geq -1$, if $K\geq 1$,
	\[
		q_n\big[\mu\in\calP_n : \mu\wedge K = (K,K)\big]
		= 1 - {\textstyle\sum_{k=1}^K} q_n(n-k,k)
		\xrightarrow [n\to\infty] {} 1.
	\]
	Lemma~\ref{lem:d-p-criterion} then ensures that $q_n\Rightarrow \delta_{(\infty,\infty)}$.
	It follows from Theorem~\ref{thm:local-limits-markov-branching}
	that $T_n$ converges in distribution to the (deterministic) infinite binary tree.
	
	\smallskip
	
	\noindent$\beta\in(-2,-1):$\quad
	Let $\beta\in (-2,-1)$.
	Stirling's formula ensures that the sequence
	$\smash {\big(i^{-\beta} \, \Gamma(i+1+\beta)/i!\big)_{i\geq 1}}$
	is bounded by a finite constant.
	As a result, the dominated convergence theorem ensures that
	\begin{align*}
		\frac {Z_n} {n^\beta}
		& = \sum_{k\geq 1} \frac {\Gamma(k+1+\beta)} {k!}
			\, \frac {\Gamma(n-k+1+\beta)} {(n-k)^\beta \, (n-k)!}
			\, \frac {(n-k)^\beta} {n^\beta}
			\, (2-\ind_{2k=n}) \, \ind_{2k\leq n}\\ 
		& \qquad\qquad \xrightarrow [n\to\infty] {}
		2 \sum_{k\geq 1} \frac {\Gamma(k+1+\beta)} {k!}
		= 2 \frac {\Gamma(2+\beta+)} {-1-\beta}
	\end{align*}
	where we have used the definition of the beta geometric distribution with parameters $(2+\beta,-1-\beta)$
	as introduced in Section~\ref{sec:alpha-gamma}.
	
	Consequently, for any fixed positive integer $k$,
	\[
		\lim_{n\to\infty} q_n(n-k,k)
		= \lim_{n\to\infty} 2 \, \frac {\Gamma(k+1+\beta)} {k!}
			\frac {\Gamma(n-k+1+\beta)} {Z_n \, (n-k)!}
		= \frac {(-1-\beta) \, \Gamma(k+1+\beta)} {\Gamma(2+\beta) \, k!}
		= q_\infty(\infty,k).
	\]
	We may then conclude with Corollary~\ref{cor:local-limits-mb-one-spine}.
\end{proof}

\subsubsection{Scaling limits}
We will now study the scaling limits of the $\beta$-splitting model
when $\beta\in(-2,-1)$ with the help of Theorem~\ref{thm:scaling-limits-infinite-mb-trees}.

Let $\nu^{(\upB)}_{\smash\beta}$ be the dislocation measure
such that for all measurable $f:\Sdec_{\leq 1}\to\bbR_+$,
\[
	\int f \, \D\nu^{(\upB)}_\beta
	:= \frac {-1-\beta} {\Gamma(2+\beta)}
		\int_0^{1/2} t^\beta \, (1-t)^\beta \, f(1-t,t,0,0,\dots) \, \D t.
\]
It follows from Section~5.1 in~\cite{hmpw2008continuumtree} that $(q_n)_{n\geq 1}$
satisfies Assumption~\hyperref[assumption:scaling]{$(\mathtt S)$}
for $\gamma=-1-\beta$ and $\nu = \nu^{(\upB)}_{\smash\beta}$
More precisely, $n^{-1-\beta} \, (1-s_1) \, \bar q_n(\D\bfs)$
converges weakly to $(1-s_1) \, \nu^{(\upB)}_{\smash\beta}(\D\bfs)$
as finite measures on $\Sdec_{\leq 1}$.

\smallskip

Let $\Lambda$ denote a random integer such that $(\infty,\Lambda)$ has distribution $q_\infty$
and for all $R\geq 1$, set $q^{(R)}$ as the distribution of $\Lambda/R^{1/(-1-\beta)}$.
Just like in Section~\ref{sec:ford-alpha-model}, Stirling's approximation and Proposition~\ref{prop:unary-immigration-criterion}
ensure that Assumption~\hyperref[assumption:immigration]{$(\mathtt I)$} is met for
$\gamma=-1-\beta$ and the immigration measure
$I^{(\upB)}_\beta := (-1-\beta)/\Gamma(2+\beta) \cdot I^{\operatorname {un}}_{-1-\beta}$.
As a result,

\begin{proposition}\label{prop:beta-scaling}
	Fix $\beta\in(-2,-1)$.
	Let $T$ be a $\MB^{\smash\calL,q,q_{\smash\infty}}_\infty$ tree
	and endow it with $\mu_T$, the counting measure on the set of its leaves.
	In the topology induced by $\upD_\GHP$,
	$(T/R,\mu_T^\calL/R^{1/(-1-\beta)})$ converges in distribution to
	a $(-1-\beta, \smash{\nu^{(\upB)}_\beta}, \smash{I^{(\upB)}_\beta})$-fragmentation
	tree with immigration.
\end{proposition}

\subsection[\texorpdfstring{$k$}{k}-ary growing trees]{$\boldsymbol k$-ary growing trees}

Let $k\geq 2$ be an integer.
In this section, we will study a model of $k$-ary trees,
\ie trees in which vertices have either $0$ or $k$ children,
described in~\cite{haas2014kary}.
This model is yet another generalisation of R\'emy's algorithm~\cite{remy1985binarygw}
(which corresponds to $k=2$).

The following algorithm allows us to get a sequence $(T_n)_{n\geq 0}$ of $k$-ary trees
such that for all $n$, $T_n$ has $n$ internal vertices (vertices that aren't leaves)
or, equivalently, $kn+1$ vertices or $(k-1)n + 1$ leaves.
First, let $T_0$ be the trivial tree $\{\varnothing\}$
and for $n\geq 1$, conditionally on $T_{n-1}$:
\begin{itemize}
	\item Pick an edge of $T_{n-1}$ (considered as a planted tree) uniformly at random,
	\item Place a new vertex on that edge and attach $k-1$ new leaves to it,
\end{itemize}
and call $T_n$ the resulting tree.
We will note $\GT_k^n$ the distribution of $T_n$.

\paragraph{The negative Dirichlet multinomial distribution}
Let $\Pi$ be a $(k-1)$-dimensional Dirichlet variable with $k$ parameters $(1/k,\dots,1/k)$,
\ie $\Pi$ takes its values in the $(k-1)$-dimensional simplex
$\{\boldsymbol x\in(0,\infty)^k : x_1+\dots+x_k=1\}$.
Conditionally on $\Pi$, let $X=(X_1,\dots,X_{k-1})$
have negative multinomial distribution of parameters $(1;\Pi)$,
\ie for each $i\in\{1,\dots,k-1\}$, $X_i$ counts the number of type $i$ results
before the first type $k$ result (failure)
in a sequence of \iid trials with $k$ possible results with respective probabilities $\Pi_1,\dots,\Pi_k$.
For any non-negative integers $n_1,\dots,n_{k-1}$ and with $N=n_1+\dots+n_{k-1}$,
we have
\[
	\prob [1] {X=(n_1,\dots,n_{k-1})}
	= \esp [4] {\frac {N!} {n_1! \dots n_{k-1}!} \, \prod_{i=1}^{k-1} \Pi_i^{n_i} \, \Pi_k}
	= \frac 1 k \, \frac 1 {1 + N} \, \prod_{i=2}^k \frac {\Gamma(n_i+1/k)} {\Gamma(1/k) \, n_i!}.
\]
The random variable $X$ is said to follow a $(k-1)$-dimensional
\emph{negative Dirichlet multinomial} distribution with parameters $(1;1/k,\dots,1/k)$
which is a multidimensional generalisation of the beta geometric distribution.
Further observe that the sum $\norm X = X_1+\dots+X_{k-1}$ has beta geometric distribution
with parameters $(1/k,1-1/k)$
and that conditionally on $\norm X = n$, $X$ follows
a $(k-1)$-dimensional Dirichlet multinomial distribution with parameters
$(n \,;\, 1/k, \dots, 1/k)$.

\paragraph{Corresponding infinite tree}
Let $(X_n,\tau_{n,1},\dots\tau_{n,k-1})_{n\geq 0}$ be a sequence of \iid variables
such that for all $n\geq 0$, $X_n$ is distributed according to a $(k-1)$-dimensional
$(1;1/k,\dots,1/k)$ negative Dirichlet multinomial distribution
and conditionally on $X_n = (m_1,\dots,m_{k-1})$,
$\tau_{n,1}, \dots, \tau_{n,k-1}$ are independent
and have respective distributions $\GT_k^{m_1},\dots,\GT_k^{m_{k-1}}$.

Conditionally on $(X_n,\tau_{n,1},\dots\tau_{n,k-1})_{n\geq 0}$,
let $T_\infty$ be the tree obtained after grafting at each height $n\geq 0$ of an infinite branch
the concatenation of $\tau_{n,i}$, $1\leq i\leq k-1$,
\ie set
\[
	\textstyle
	T_\infty := \ttb_\infty \bigotimes_{n\geq 0} \big(\ttv_n,
		\lBrack \tau_{n,1}, \dots \tau_{n,k-1} \rBrack \big),
\]
and let $\GT_k^\infty$ be the distribution of $T_\infty$.

\medskip

Section~\ref{sec:k-ary-local-limits} will prove the following proposition.

\begin{proposition}\label{prop:k-ary-local-limits}
	In the sense of the local limit topology,
	$\GT_k^n$ converges weakly to $\GT_k^\infty$
	when $n$ goes to $\infty$.
\end{proposition}

Let $\Pi$ be a $(k-1)$-dimensional Dirichlet variable with parameters $(1/k,\dots,1/k)$.
Following \cite[Section~3.1]{haas2014kary}, we define the dislocation measure $\nu_k^\GT$
such that for all measurable $f:\Sdec_{\leq 1}\to\bbR_+$
\[
	\int_{\Sdec_{\leq 1}} f \, \D\nu_k^\GT
	= \frac {\Gamma(1/k)} k
	\, \esp [4] {\frac {f\big[(\Pi,0,0,\dots)^\downarrow\big]} {1-\Pi_1}}.
\]
Let $\Delta$ be a $(k-2)$-dimensional Dirichlet variable with parameters $(1/k,\dots,1/k)$.
We also define the immigration measure $I_k^\GT$ for all measurable
functions $F:\Sdec\to\bbR_+$ by
\[
	\int_{\Sdec} F(\bfs) \, I_k^\GT(\D\bfs)
	:= \frac {1/k} {\Gamma(1-1/k)} \int_0^\infty t^{-1-1/k}
		\esp [2] {F\big(t\,(\Delta,0,0,\dots)^\downarrow\big)} \,\D t.
\]
The aim of Section~\ref{sec:k-ary-scaling-limits} will be to prove the next proposition.

\begin{proposition}\label{prop:k-ary-scaling-limits}
	Let $T$ be a $\GT_k^\infty$-distributed tree and endow it with $\mu^\circ_T$,
	the counting measure on the set of its internal vertices.
	With respect to the topology induced by $\upD_\GHP$, when $R$ grows to infinity,
	$(T/R,\mu^\circ_T/R^{\smash k})$ converges in distribution
	to a $(1/k, \nu_k^\GT, I_k^\GT)$-fragmentation tree with immigration.
\end{proposition}

\subsubsection{Markov branching property and local limits}\label{sec:k-ary-local-limits}
For any $\ttt$ in $\ttT$,
we define $\Lambda^\circ(\ttt)$ as the decreasing rearrangement of the number of internal vertices
of the sub-trees of $\ttt$ attached to its root,
\ie we let $\Lambda^\circ(\ttt) := \Lambda(\ttt) - \Lambda^{\smash\calL}(\ttt)$.
In the setting of $k$-ary growing trees, $\Lambda^\circ(T_0) = \varnothing$ \as
and if $n\geq 1$, $\Lambda^\circ(T_n)$ takes its values in the set of
decreasing families of $(\bbZ_+)^k$ with sum $n-1$.
Because of the deterministic relationship between $n$, $\# T_n$ and $\#_\calL T_n$,
we have $\Lambda(T_0) = \Lambda^{\smash\calL}(T_0) = \varnothing$
and for $n\geq 1$, $\Lambda(T_n) = k \Lambda^\circ(T_n) + (1,\dots,1)$ in $\calP_{kn}$
and $\Lambda^{\smash\calL}(T_n) = (k-1) \Lambda^\circ(T_n) + (1,\dots,1)$ in $\calP_{(k-1)n+1}$.
For all $n\geq 1$, note $q^\circ_{n-1}$ the distribution of $\Lambda^\circ(T_n)$,
that is the first-split distribution of $T_n$ with respect to internal vertices.

Proposition~3.3 from~\cite{haas2014kary} states that $(T_n)_{n\geq0}$
satisfies the Markov branching property and the distribution of $T_n$
may be expressed as either $\MB^q_{kn+1}$ or $\MB^{\smash{\calL,q^\calL}}_{(k-1)n+1}$
where $q$ and $q^{\smash\calL}$ are both easily obtained from $(q^\circ_n)_{n\geq 0}$.
Rewriting the formula from this last proposition for our purposes
(where partition blocs are arranged in decreasing order),
for all $n\geq 1$ and $\lambda = (\lambda_1,\dots,\lambda_k)$
decreasing with sum $n$,
we get that
\[
	q^\circ_{n-1}(\lambda)
	= \frac {(k-1)!} {\prod_{j\geq 1} m_j(\lambda)!} \frac 1 k
	\frac {\Gamma(1/k)} {\Gamma(n+1+1/k)}
	\prod_{i=1}^k \frac {\Gamma(\lambda_i+1/k)} {\Gamma(1/k)\, \lambda_i!}
	\sum_{i=1}^k \bigg( m_{\lambda_i}(\lambda) \lambda_i!
	\sum_{j=0}^{\lambda_i} \frac {(j+n-\lambda_i)!} {j!}\bigg).
\]

We can rewrite $\GT_k^\infty$ as the distribution $\MB^{q,q_{\smash\infty}}_\infty$
or $\MB^{\smash{\calL,q^\calL,q_\infty^\calL}}_\infty$
of an infinite Markov branching tree.
The corresponding measures $q_\infty$ and $q^{\smash\calL}_\infty$ on $\calP_\infty$
can also be easily deduced from the measure $q^\circ_\infty$ on the set of decreasing
$k$-tuples of $\bbZ_+\cup\{\infty\}$ with infinite sum
such that $q^\circ_\infty(\lambda) = 0$ if $\lambda_2$ is infinite
and
\[
	q^\circ_\infty(\infty,\lambda_2, \dots, \lambda_k)
	= \frac {(k-1)!} {\prod_{j\geq 1} m_j(\lambda)!} \, \frac 1 k \, \frac 1 {\norm\lambda+1}
		\, \prod_{i=2}^k \frac {\Gamma(\lambda_i+1/k)} {\Gamma(1/k)\, \lambda_i!}
\]
for any integers $\infty>\lambda_2\geq\dots\geq\lambda_p\geq 0$.
Observe that $q^\circ_\infty(\infty,\lambda_2, \dots, \lambda_k)
= \prob [1] {X^{\smash\downarrow} = (\lambda_2,\dots,\lambda_k)}$
where $X$ is a $(k-1)$-dimensional negative Dirichlet multinomial
variable with parameters $(1;1/k,\dots,1/k)$.
As a result, $q^\circ_\infty$ is a probability measure.

\smallskip

\begin{proof}[of~Proposition~\ref{prop:k-ary-local-limits}]
	Let $\lambda = (\lambda_2,\dots,\lambda_k)$ be a decreasing sequence of $(\bbZ_+)^{k-1}$
	and set $L=\lambda_2+\dots+\lambda_k$.
	For $n$ large enough, we have
	\begin{align*}
		& q^\circ_n(n-L,\lambda_2,\dots,\lambda_k)\\
		& \qquad =
		\frac {(k-1)!} {\prod_{j\geq 1} m_j(\lambda)!} \frac 1 k
		\prod_{i=2}^k \frac {\Gamma(\lambda_i+1/k)} {\Gamma(1/k)\, \lambda_i!}
		\smash {\overbrace {\vphantom {\sum_{j=0}^{n-L}}
			\frac {\Gamma(n-L+1/k)} {\Gamma(n+1+1/k)}
			}^{\substack {n^{\mathrlap {-L-1}}\\
			\rotatebox [origin=c] {-90} {$\displaystyle\sim\,$} \, \mathrlap {\;\; n\to\infty}}}}
		\Bigg[ \smash {\overbrace {\sum_{j=0}^{n-L} \frac {(j+L)!} {j!}
			}^{\substack {n^{L+1} \int_0^1 x^L \D x\\
				\rotatebox [origin=c] {-90} {$\displaystyle\,\sim\,$} \, \mathrlap {\;\; n\to\infty}}}}
		+ \smash {\overbrace {\sum_{i=2}^k \sum_{j=0}^{\lambda_i}
			\frac {\lambda_i! \, (j+n-\lambda_i)!} {(n-L)! \, j!}
			}^{\substack {O(n^L)\\
				\rotatebox [origin=c] {-90} {$\displaystyle\,=\,$} \, \mathrlap {\;\; n\to\infty}}}} \Bigg]\\
		& \qquad\qquad\qquad \xrightarrow [n\to\infty] {}
			\frac {(k-1)!} {\prod_{j\geq 1} m_j(\lambda)!} \, \frac 1 k \, \frac 1 {L+1}
				\, \prod_{i=2}^k \frac {\Gamma(\lambda_i+1/k)} {\Gamma(1/k)\, \lambda_i!}
		= q^\circ_\infty(\infty,\lambda).
	\end{align*}
	Corollary~\ref{cor:local-limits-mb-one-spine} concludes this proof.
\end{proof}

\subsubsection{Scaling limits}\label{sec:k-ary-scaling-limits}
Proposition~3.1 in~\cite{haas2014kary}
states that $n^{1/k} \, (1-s_1) \, \bar q^\circ_n(\D\bfs) \Rightarrow (1-s_1) \, \nu_k^\GT (\D\bfs)$
as $n\to\infty$ in the sense of finite measures on $\Sdec_{\leq 1}$.
Assumption~\hyperref[assumption:scaling]{$(\mathtt S)$} of Theorem~\ref{thm:scaling-limits-infinite-mb-trees}
is thus met for the sequence $q^\circ$.
To prove Proposition~\ref{prop:k-ary-scaling-limits}, we will need the following lemma.
Let $X=(X_1,\dots,X_{k-1})$ denote a negative Dirichlet multinomial
variable with parameters $(1;1/k,\dots,1/k)$.

\begin{lemma}\label{lem:neg-dirichlet-mult-immigration-cv}
	Let $\Delta$ be a $(k-2)$-dimensional Dirichlet $(1/k,\dots,1/k)$ variable.
	For all Lipschitz-continuous functions $G:[0,\infty)^{k-1}\longrightarrow\bbR_+$
	such that $G(\boldsymbol x) \leq 1\wedge\norm{\boldsymbol x}$
	for all $\boldsymbol x$ in $[0,\infty)^{k-1}$,
	\[
		R\,\esp [3] {G\bigg(\frac X {R^k}\bigg)}
		\xrightarrow [R\to\infty] {}
		\frac {1/k} {\Gamma(1-1/k)} \int_0^\infty t^{-1-1/k} \esp {G(t\,\Delta)} \,\D t.
	\]
\end{lemma}

\begin{proof}
	Let $(Y_n)_{n\geq 1}$ be \iid and such that conditionally on $\Delta$,
	$Y_n$ is multinomial with parameters $(1;\Delta)$.
	Moreover, set $Z_n := Y_1 + \dots + Y_n$.
	The law of large numbers ensures that $Z_n/n$ converges almost surely to $\Delta$.
	Let $N$ be independent of $\Delta$ and $(Z_n)_n$
	and have beta geometric distribution with parameters $(1/k,1-1/k)$.
	Observe that $X$ has the same distribution as $Z_N$.
	
	Define $g:\bbR_+\to\bbR_+$ by $g(t) := \esp {G(t\,\Delta)}$.
	The dominated convergence theorem implies that it is continuous
	and it clearly satisfies $g(t) \leq 1\wedge t$.
	Lemma~\ref{lem:unary-immigration-prelim} then ensures that $R \, \esp {g(N/R^k)}
	\to \big[ k \, \Gamma(1-1/k)\big]^{-1} \int_0^\infty t^{-1-1/k} g(t) \, \D t$.
	
	Since $Z_n/n$ \as converges to $\Delta$ and because $\norm {(Z_n/n) - \Delta} \leq 2$,
	we can use the dominated convergence theorem to state that for all positive $\varepsilon$,
	there exists $n_\varepsilon$ such that $\esp {\norm {(Z_n/n)-\Delta}} < \varepsilon$
	as soon as $n\geq n_\varepsilon$.
	Therefore, if $K$ is the Lipschitz constant of $G$,
	\begin{align*}
		& \abs [4] {
			R\,\esp [3] {G\bigg(\frac X {R^k}\bigg)}
			- R \, \esp [3] {g\bigg(\frac N {R^k}\bigg)}}
		\leq R \, \esp [4] {
			\abs [3] {G\bigg[ \frac N {R^k} \frac {Z_N} N \bigg]
			- G\bigg[ \frac N {R^k} \, \Delta \bigg]}}\\
		& \qquad\qquad \leq R \, \esp [4] { 1 \wedge \bigg(
			K\varepsilon \frac N {R^k} \bigg)}
		+ \frac {2K n_\varepsilon} {R^{k-1}}
		\xrightarrow [R\to\infty] {}
		\frac {1/k} {\Gamma(1-1/k)} \int_0^\infty
			\frac {1\wedge (K\varepsilon \, t)} {t^{1+1/k}} \D t
	\end{align*}
	where we have used Lemma~\ref{lem:unary-immigration-prelim}.
	This last quantity in turn converges to $0$ when $\varepsilon\to 0$
	which proves the desired result.
\end{proof}

\begin{proof}[of Proposition~\ref{prop:k-ary-scaling-limits}]
	Recall that if $\Lambda$ is such that $(\infty,\Lambda)$ follows $q^\circ_\infty$,
	then $\Lambda$ is distributed like $X^{\smash\downarrow}$.
	We may then deduce from Lemma~\ref{lem:neg-dirichlet-mult-immigration-cv}
	and Lemma~\ref{lem:portmanteau-extension}
	that Assumption~\hyperref[assumption:immigration]{$(\mathtt I)$}
	holds for $q_\infty^\circ$, $I = I_k^\GT$ and $\gamma = 1/k$.
	As a result, Theorem~\ref{thm:scaling-limits-infinite-mb-trees} concludes this proof.
\end{proof}

\bigskip

\subsubsection*{Acknowledgements}
The author would like to thank his advisor, B\'en\'edicte Haas,
for many helpful discussions and suggestions as to how to improve this paper.

\begin{small}
\fancyhead[L]{\footnotesize\nouppercase{References}}
\fancyhead[R]{}
\phantomsection

\end{small}

\end{document}